\documentclass[a4paper,reqno, 11pt]{amsart}   	%11pt
\usepackage[DIV=12, oneside]{typearea}			%DIV=12
\usepackage[utf8]{inputenc}
\usepackage[T1]{fontenc}
\usepackage[english]{babel}
\usepackage[centertags]{amsmath}
\usepackage{amstext,amssymb,amsopn,amsthm}
\usepackage{mathrsfs}
\usepackage{dsfont}
\usepackage{bbm}
\usepackage{thmtools}
\usepackage{graphicx}
\usepackage[titletoc,title]{appendix}
\usepackage{etexcmds}

\usepackage[backgroundcolor=white, bordercolor=blue,
linecolor=blue]{todonotes}

\usepackage[colorlinks=true, linkcolor=black, citecolor=black]{hyperref}
\usepackage{enumitem}
\setlist[enumerate]{itemsep=0mm}

\addto\extrasenglish{}
\addto\extrasenglish{}
\addto\extrasenglish{}
\theoremstyle{plain}
\declaretheorem[title=Theorem, parent=section]{theorem}
\declaretheorem[title=Lemma,sibling=theorem]{lemma}
\declaretheorem[title=Proposition,sibling=theorem]{proposition}
\declaretheorem[title=Corollary,sibling=theorem]{corollary}

\theoremstyle{definition}
\declaretheorem[title=Definition,sibling=theorem]{definition}
\declaretheorem[title=Remark,sibling=theorem]{remark}
\declaretheorem[title=Remark, numbered=no]{remark*}

\declaretheorem[title=Assumption, numbered=no]{assumption*}

\numberwithin{equation}{section}
\newcommand{\N}{\mathds{N}}
\newcommand{\R}{\mathds{R}}

\newcommand{\Z}{\mathds{Z}}

\def\hmath$#1${\texorpdfstring{{\rmfamily\textit{#1}}}{#1}}

\newcommand{\cB}{\mathcal{B}}
\newcommand{\cW}{\mathcal{W}}
\newcommand{\cH}{\mathcal{H}}

\newcommand{\cC}{\mathcal{C}}

\newcommand{\cQ}{\mathcal{Q}}
\newcommand{\cD}{\mathcal{D}}

\newcommand{\cT}{\mathcal{T}}

\newcommand{\eps}{\varepsilon}

\newcommand{\BIGOP}[1]
{
	\mathop{\mathchoice%
		{\raise-0.22em\hbox{\huge $#1$}}%
		{\raise-0.05em\hbox{\Large $#1$}}{\hbox{\large $#1$}}{#1}}}

\def\XXint#1#2#3{{\setbox0=\hbox{$#1{#2#3}{\int}$}
		\vcenter{\hbox{$#2#3$}}\kern-.5\wd0}}

\newcommand{\BIGboxplus}{\mathop{\mathchoice%
		{\raise-0.35em\hbox{\huge $\boxplus$}}%
		{\raise-0.15em\hbox{\Large $\boxplus$}}{\hbox{\large $\boxplus$}}{\boxplus}}}

\DeclareMathOperator{\dist}{dist}

\DeclareMathOperator{\diam}{diam}

\DeclareMathOperator{\supp}{supp}

\renewcommand{\d}{\textnormal{d}}

\newcommand{\1}{{\mathbbm{1}}}

\usepackage{esint}
\newcommand{\norm}[1]{\left\lVert#1\right\rVert} % norm 
\newcommand{\abs}[1]{\ensuremath{\left\vert#1\right\vert}} % absolute value
\newcommand{\bigabs}[1]{\ensuremath{\big\vert#1\big\vert}} % big absolute value
\DeclareMathOperator*{\ext}{ext} %extension local
\DeclareMathOperator*{\extns}{Ext_s} % extension fractional s
\DeclareMathOperator*{\trns}{Tr_s} % trace fractional

\newcommand{\inr}[1]{\text{inr}(#1)} % inner radius
\newcommand{\ie}{i.e.\,}

\newcommand{\eg}{e.g.\,}

\newcommand{\distw}[2]{\dist(#1, #2)} %%% distance

\makeatletter
\providecommand\@dotsep{5}
\def\listtodoname{List of Todos}
\def\listoftodos{\@starttoc{tdo}\listtodoname}
\makeatother

\usepackage{comment}

\newcommand{\vs}[2]{V^{#1, #2}}
\newcommand{\VspOm}{V^{s,p}(\Omega\,|\, \R^d)}
\newcommand{\VspOmNull}{V^{s,p}_0(\Omega\,|\, \R^d)}
\newcommand{\VsOneOm}{V^{s,1}(\Omega\,|\, \R^d)}

\begin{document}
\allowdisplaybreaks
\title{Robust nonlocal trace and extension theorems}

\author{Florian Grube}
\author{Moritz Kassmann}

\address{Fakultät für Mathematik, Universität Bielefeld, Postfach 10 01 31, 33501 Bielefeld, Germany}
\email{fgrube@math.uni-bielefeld.de}

\address{Fakult\"{a}t f\"{u}r Mathematik\\Universit\"{a}t Bielefeld\\Postfach 100131\\D-33501 Bielefeld}
\email{moritz.kassmann@uni-bielefeld.de}
\urladdr{https://www.uni-bielefeld.de/math/kassmann}

\makeatletter
\@namedef{subjclassname@2020}{%
	\textup{2020} Mathematics Subject Classification}
\makeatother

\thanks{Financial support by the German Research Foundation (GRK 2235
	- 282638148, SFB 1283 - 317210226) is gratefully acknowledged.}

\begin{abstract} 
We prove trace and extension results for Sobolev-type function spaces that are well suited for nonlocal Dirichlet and Neumann problems including those for the fractional $p$-Laplacian. Our results are robust with respect to the order of differentiability. In this sense they are in align with the classical trace and extension theorems. 
\end{abstract}

\maketitle

%\setcounter{tocdepth}{1} 
%\tableofcontents

\section{Introduction}\label{sec:intro}

The article is concerned with well-posedness of nonlinear nonlocal equations in bounded domains, such as 
\begin{align}
\begin{split} \label{eq:p-Lap-eq}
(-\Delta)^s_p u &= f \qquad \text{ in } \Omega \,, \\
u &= g \qquad \text{ in } \R^d \setminus \Omega \,,
\end{split}
\end{align}
where the fractional $p$-Laplace is defined via
\begin{align*}
(-\Delta)^s_p u (x) &= (1-s) \operatorname{pv.} \int\limits_{\R^d} |u(x)-u(y)|^{p-2} \big(u(x) - u(y) \big) \frac{\d y}{|x-y|^{d+sp}} \,.
\end{align*}

A standard approach to problems like \eqref{eq:p-Lap-eq} is the variational approach, which is based on an energy functional and corresponding function spaces. Since the operator $(-\Delta)^s_p u$ is nonlocal, it is necessary to prescribe values $u(x)$ for $x \in \R^d \setminus \Omega$ in order for \eqref{eq:p-Lap-eq} to be well-posed. A possible yet restrictive option is to work in the Sobolev-Slobodeckij space $W^{s,p}(\R^d)$. Note that an assumption of the type $g \in W^{s,p}(\R^d)$ imposes unnatural restrictions since the problem \eqref{eq:p-Lap-eq} does not involve any regularity of $g$ in $\R^d \setminus \overline{\Omega}$ other than some weighted integrability. Popular workarounds include assumptions of the type $g \in W^{s,p}(\Omega_\varepsilon) \cap L^{p}(\R^d;(1+|x|)^{-d-sp} \d x)$ for some enlarged domain $\Omega_\varepsilon=\{ x\in \R^d\,|\, \distw{x}{\overline{\Omega}}<\eps \}$. 

\medskip

In this work we introduce and study trace spaces on $\R^d \setminus \Omega$ that allow for a natural variational approach to nonlocal nonlinear problems. An important feature of our approach is the robustness of our results as $s \to 1^{-}$. This allows for a theory of well-posedness for problems like \eqref{eq:p-Lap-eq} that is continuous in the parameter $s$ at $s=1$. In this case, problem \eqref{eq:p-Lap-eq} reduces to 
\begin{align*}
- \operatorname{div} (|\nabla u|^{p-2} \nabla u) = f \text{ in } \Omega, \qquad u=g \text{ on } \partial \Omega.
\end{align*}

In order to derive the setting of the variational approach, let us explain the definition of a weak solution to our model example \eqref{eq:p-Lap-eq}. Given a sufficiently regular solution $u$ to \eqref{eq:p-Lap-eq} and a regular test function $\varphi : \R^d \to \R$ with compact support in $\Omega$, the following should hold true
\begin{align*}
\int\limits_{\Omega} (-\Delta)^s_p u \; \varphi = \int\limits_{\Omega} f \, \varphi \,,
\end{align*}
which, after an application of Fubini's theorem, reads
\begin{align}\label{eq:tested-eq}
\frac{1-s}{2} \iint\limits_{(\Omega^c\times \Omega^c)^c}  |u(x)-u(y)|^{p-2} \big(u(x) - u(y) \big) \big(\varphi(x) - \varphi(y) \big) \frac{\d y \,\d x}{|x-y|^{d+sp}} = \int\limits_{\Omega} f \varphi \,.
\end{align}
This line motivates the following definition of an energy space. For a bounded open set $\Omega\subset \R^d$ and $1\le p< \infty$ we consider the fractional Sobolev-type space 
\begin{align}
	\VspOm&:= \{ u:\R^d \to \R \text{ measurable }\,|\, [u]_{\VspOm} <\infty   \}, \label{eq:V-def} \\
	[u]_{\vs{s}{p}(A\,|\,B)}^p &:= (1-s) \iint\limits_{A\, B} \frac{\abs{u(x)-u(y)}^p}{\abs{x-y}^{d+s\,p}} \d x \d y \quad (A,B\in \cB( \R^d) ). \label{eq:V-seminorm-def} 
\end{align}
We endow this space with the norm $\norm{u}_{\VspOm}^p := \norm{u}_{L^p(\Omega)}^p + [u]_{\VspOm}^p$. The space $\VspOm$ is a separable Banach space and reflexive for $p>1$, see e.g. \cite[Chapter 3.4]{Fog20}. It is well known that this space converges to $W^{1,p}(\Omega)$ for $1<p<\infty$ as $s\to 1^{-}$, see \cite[Theorem 2]{BBM01}, \cite[Theorem 1.1, Theorem 1.3, Theorem 1.5]{Foghem_2023}. In his famous work \cite{Gag57} Gagliardo proved that the classical trace $\gamma: W^{1,p}(\Omega)\to W^{1-1/p, p}(\partial \Omega)$ is linear and continuous and has a continuous right inverse. We are concerned with the search for a trace theorem and extension result for the fractional Sobolev spaces-type $\VspOm$ onto the nonlocal boundary $\Omega^c$ such that the result is robust in the limit $s\to 1^{-}$. 

\begin{remark}
In some more applied fields such as peridynamics, one studies nonlocal problems in bounded open sets $\Omega$, where data are prescribed in a bounded open set $E \supset \Omega$, see \cite{MeDu16}. Then, there is no need to discuss the decay at infinity but the main challenge remains: quantify local behavior of functions across the boundary $\partial \Omega$. Our results apply to such problems directly as $\R^d$ can be replaced by a general set $E$.
\end{remark}

\medskip

\subsection*{Main results}
We introduce a space of functions $\cT^{s,p}(\Omega^c)$ defined on $\Omega^c$, see \eqref{eq:Tsp_definition}, and prove trace and extension results which are robust in the parameter $s$, see \autoref{th:trace-extension}, \autoref{th:trace-p-1}. Lastly, we prove the asymptotic of the spaces $\cT^{s,p}(\Omega^c)$ as well as some related weighted $L^p$ spaces as $s\to 1^-$, see \autoref{th:asym_trace}.

\medskip 

Due to the nonlocality of the operators under consideration, problems like \eqref{eq:p-Lap-eq} can be formulated in open sets, which are not necessarily connected. Since our main results do not require $\Omega$ to be connected, we define $\Omega \subset \R^d$ to be a Lipschitz domain, if it is open and has a uniform Lipschitz boundary, see \autoref{sec:prelims}. We define measures \begin{equation}\label{eq:def_mu_s}
	\mu_s(\d x):= \1_{\Omega^c}(x) (1-s)d_x^{-s}(1+d_x)^{-d-s(p-1)}\d x
\end{equation} 
on the Borel $\sigma$-algebra $\cB(\R^d)$, $s\in (0,1)$, $1\le p<\infty$ where $d_x:= \distw{x}{\partial\Omega}$ for $x\in \R^d$. We simply write $\mu_s(x)$ for the density of the measure $\mu_s$ with respect to the Lebesgue measure on $\R^d$. Given an open bounded set $A \subset \R^d$, note that $\mu_s(A) \asymp (1-s) \int_{A \cap \Omega^c} d_x^{-s} \d x$ and $\mu_s$ converges weakly for $s \to 1^-$ to the Hausdorff measure on $\partial \Omega \cap A$, see \autoref{lem:weak_convergence_measures}.

\medskip

We introduce for $s\in (0,1), 1\le p < \infty$ our trace spaces
\begin{align}\label{eq:Tsp_definition}
	\begin{split}
	\cT^{s,p}(\Omega^c)&:= \{ g:\Omega^c \to \R \text{ measurable } \,|\, \norm{g}_{\cT^{s,p}(\Omega^c)}<\infty \},\\
	\norm{g}_{\cT^{s,p}(\Omega^c)}^p&:= \norm{g}_{L^{p}(\Omega^c;\,\mu_s)}^p+	[g]_{\cT^{s,p}(\Omega^c)}^p, \\
	[f,g]_{\cT^{s,p}(\Omega^c)}^p &:= \iint\limits_{\Omega^c \, \Omega^c} \frac{\abs{f(x)-f(y)}^{p-2} (f(x)-f(y)) (g(x)-g(y)) }{  ((\abs{x-y}+d_x+d_y)\wedge 1)^{d+s(p-2)}} \mu_s(\d x) \mu_s(\d y)
	\end{split}
\end{align}
The space $\cT^{s,p}(\Omega^c)$ is a separable Banach space (Hilbert for $p=2$) and reflexive for $p>1$, see \autoref{lem:banach}. Now we state the trace result and extension result for $p>1$.
\begin{theorem}\label{th:trace-extension}
	Let $\Omega\subset \R^d$ be a bounded Lipschitz domain, $s\in(0,1)$, $1< p<\infty$. Then the trace operator
	\begin{align*}
		\trns: \VspOm &\to \cT^{s,p}(\Omega^c)\\
		u&\mapsto u|_{\Omega^c}
	\end{align*}
	is continuous and linear and there exists a continuous linear right inverse
	\begin{align*}
		\extns: \cT^{s,p}(\Omega^c) &\to   \VspOm\\
		g&\mapsto \extns(g)
	\end{align*}
	which we call the nonlocal extension operator. Moreover, the continuity constants of the linear trace and extension operator only depend on $\Omega$, a lower bound on $s$ as well as a lower and upper bound on $p$. 
\end{theorem}

An extension of \autoref{th:trace-extension} to the case $p=1$ requires a refined consideration. Analogous to the case $p>1$, one might guess that the limit space of $V^{s,1}(\Omega\,|\, \R^d)$ as $s\to 1^{-}$ is $W^{1,1}(\Omega)$. But, in fact, the Sobolev space $W^{1,1}(\Omega)$ is too small to capture all functions such that $\liminf_{s\to 1^{-}}\norm{f}_{\VsOneOm}$ is finite. The limit space of $\VsOneOm$ as $s\to 1^{-}$ turns out to be the space of functions of bounded variation $BV(\Omega)$, see \cite[Theorem 1]{Dav02}, \cite[Theorem 3', Corollary 2, Corollary 5]{BBM01} and \cite[Theorem 1.3, Theorem 1.4']{Foghem_2023}. It is well known that functions in $BV(\Omega)$ have a trace to the boundary $\partial \Omega$ that is integrable and the trace map to $L^1(\partial \Omega)$ is surjective, see \cite{Gag57}, \cite[Theorem 1]{Dav02} or \cite[Theorem 18.13]{Leo17}. \autoref{th:trace-extension} may be extended to the case $p=1$ as follows.

\begin{theorem}\label{th:trace-p-1}
	Let $\Omega\subset \R^d$ be a bounded Lipschitz domain, $s\in(0,1)$. Then the trace operator
	\begin{align*}
	\trns: \vs{s}{1}(\Omega|\R^d) &\to L^1(\Omega^c;\,\mu_s(\d x)), \qquad u \mapsto u|_{\Omega^c}
	\end{align*}
	is continuous and linear. There exists a continuous linear right inverse
	\begin{align*}
	\extns: \cT^{s,1}(\Omega^c) &\to   \vs{s}{1}(\Omega|\R^d), \qquad g \mapsto \ext(g) \,.
	\end{align*}
	The continuity constants of the linear trace and extension operator only depend on $\Omega$ and a lower bound on $s$. In addition, the norm of the extension operator in dimension $d=1$ also depends on a lower bound on $1-s$.
\end{theorem}

This result is analogous to the local setting where $L^1(\Omega^c;\, \mu_s)$ is a suitable replacement for $L^1(\partial \Omega)$. In particular, a direct analog of the trace result from \autoref{th:trace-extension} for $p=1$, \ie $\norm{u}_{\cT^{s,1}(\Omega^c)}\lesssim \norm{u}_{\vs{s}{1}(\Omega\,|\, \R^d)}$, cannot hold, see the counter example in \autoref{example:trace_p-1}. This is in alignment with the local setting. Recall that there exists a nonlinear bounded extension operator from $L^1(\partial \Omega)$ to $BV(\Omega)$, see e.g. \cite[Theorem 1.2]{MSS18}. It was shown in \cite{Pee79} that a continuous extension map of integrable functions on $\partial \Omega$ to a function of bounded variation in $\Omega$ cannot be linear. If we restrict ourselves to the Besov space $B_{1,1}^0(\partial \Omega)\subset L^1(\partial \Omega)$, then a continuous linear extension to functions $BV(\Omega)$ that is right inverse to the trace map exists, see \cite[Theorem 1.1]{MSS18}. A function $f\in L^1(\partial \Omega)$ is in the Besov space $B_{1,1}^0(\partial \Omega)$ whenever the seminorm $[f]_{B_{1,1}^0(\partial \Omega)}$ is finite, where 
\begin{equation*}
[f]_{B_{1,1}^0(\partial \Omega)}:=\int\limits_{\partial \Omega\times \partial \Omega} \frac{\abs{f(x)-f(y)}}{\abs{x-y}^{d-1}} (\sigma\otimes \sigma)(\d (x,y)).
\end{equation*}
The Besov space $B_{1,1}^0(\partial \Omega)$ is a Banach space endowed with the norm $\norm{f}_{B_{1,1}^0(\partial \Omega)}:= \norm{f}_{L^1(\partial \Omega)}+ [f]_{B_{1,1}^0(\partial \Omega)}$. In step 1 of the proof of \autoref{th:asym_trace}, see \autoref{sec:convergence_pointwise}, we show that our trace norm $	\norm{f}_{\cT^{s,1}(\Omega^c)}$ converges to $\norm{f}_{B_{1,1}^0(\partial \Omega)}$ as $s\to 1^{-}$ for any $f\in C_c^{0,1}(\R^d)$. In this regard, we recover the local extension theorem to $BV(\Omega)$ functions in the limit $s\to 1^{-}$ as the extension operator in \autoref{th:trace-p-1} has a uniformly bounded norm in the same limit.

\medskip

As mentioned above, the spaces $V^{s,p}(\Omega\,|\, \R^d)$, $1<p<\infty$, converge to the traditional Sobolev space $W^{1,p}(\Omega)$ as the order of differentiability $s$ reaches $1$. Having established the robust continuity of the trace and extension operators from \autoref{th:trace-extension} and \autoref{th:trace-p-1}, our second goal is to study the limiting properties of the spaces $\cT^{s,p}(\Omega^c)$ for $s\to 1^{-}$ and to recover the classical trace and extension results for Sobolev spaces.
\begin{theorem}\label{th:asym_trace}
	Let $\Omega\subset \R^d$ be a bounded Lipschitz domain, $s\in(0,1)$, $1<  p<\infty$. 
	\begin{alignat*}{2}
		\norm{\trns u}_{L^{p}(\Omega^c;\,\mu_s)} &\to \norm{\gamma u}_{L^p(\partial \Omega)} \qquad &&(u\in W^{1,p}(\R^d)),\\
		[\trns u]_{\cT^{s,p}(\Omega^c)} &\to [\gamma u]_{W^{1-1/p,p}(\partial \Omega)}  \qquad &&(u\in W^{1,p}(\R^d)),\\
			\norm{\trns u}_{L^{1}(\Omega^c;\,\mu_s)} &\to \norm{\gamma u}_{L^1(\partial \Omega)} \qquad &&(u\in BV(\R^d)),\\
		[\trns u]_{\cT^{s,1}(\Omega^c)} &\to [\gamma u]_{B^{0}_{1,1}(\partial \Omega)} \qquad &&(u\in C_c^{0,1}(\R^d))
	\end{alignat*}
	as $s\to 1^{-}$. Here, $\gamma$ denotes the classical trace operator and $B^{0}_{1,1}(\partial \Omega)$ the Besov space defined above.
\end{theorem}

\begin{remark}
	In the case of a bounded connected $C^{1,1}$-domain $\Omega$ and $p=2$, \autoref{th:trace-extension} and \autoref{th:asym_trace} have been established in \cite{GrHe22}, see the discussion of related literature below. 
\end{remark}

\medskip

\subsection*{Applications to the Dirichlet problem} Let us present a well-posedness result for \eqref{eq:p-Lap-eq}. We define the space of test functions for the Dirichlet problem as follows:
\begin{align}
V^{s,p}_0(\Omega | \R^d) = \{ v \in \VspOm | \; v = 0 \text{ a.e. on } \R^d \setminus \Omega\}
\end{align}

\begin{definition} 	Let $\Omega\subset \R^d$ be a bounded Lipschitz domain, $s\in(0,1)$, $1< p<\infty$. Let $g \in \cT^{s,p}(\Omega^c)$ and $f \in \VspOm' \supset L^{p'}(\Omega)$. We say that $u \in \VspOm$ is a weak solution of \eqref{eq:p-Lap-eq} if for every $\varphi \in V^{s,p}_0(\Omega | \R^d)$ the equation \eqref{eq:tested-eq} holds true.
\end{definition}

Here is our result on well-posedness of the Dirichlet problem.

\begin{corollary}\label{cor:well_posedness} 	Let $\Omega\subset \R^d$ be a bounded Lipschitz domain, $s_\star \leq s < 1$, $1< p<\infty$. Let $g \in \cT^{s,p}(\Omega^c)$ and $f \in \VspOm' \supset L^{p'}(\Omega)$. Then there exists a unique weak solution $u \in \VspOm$ to problem \eqref{eq:p-Lap-eq}. Moreover, there is a constant $c>0$, depending only on $p, \Omega, s_\star$ such that
	\begin{align}\label{eq:estim-energy}
	\|u\|_{\VspOm} \leq c \big( \|g\|_{\cT^{s,p}(\Omega^c)} + \|f\|_{\VspOm'} \big) \,.
	\end{align}
\end{corollary}

\begin{proof}
	Let $V_g^{s,p}(\Omega|\R^d)$ be the set of all functions $v$ of the form $v = \extns g + v_0$ with $v_0 \in \VspOmNull$ and $\extns g$ as in \autoref{th:trace-extension}. This set is a closed convex subset of $\VspOm$. Let $I: V_g^{s,p}(\Omega|\R^d) \to \R$ be defined by
	\begin{align*}
	I(v) =  \frac{1-s}{2p}  \iint\limits_{(\Omega^c\times \Omega^c)^c} \frac{\abs{v(x)-v(y)}^p}{\abs{x-y}^{d+s\,p}} \d y \d x - f(v) \,.
	\end{align*}
	The functional $I$ is strictly convex and weakly lower semi-continuous on the reflexive, separable Banach space $V_g^{s,p} (\Omega | \R^d)$. Since \begin{equation*}
		|f(v)| \leq \|f\|_{\VspOm'} \|v\|_{\VspOm} \leq \delta \|v\|^p_{\VspOm} + (p')^{-1}(\delta p )^{-1/(p-1)} \|f\|^{p'}_{\VspOm'}
	\end{equation*} for every $\delta \in (0,1)$, we can apply the Poincaré inequality, see \autoref{pro:poinc-robust}, to the function $v-\extns(g)$ to obtain
	\begin{align*}
	I(v) \geq  \frac{1}{4p} [v]^p_{\VspOm} +  c_1^{-1} \|v\|^p_{L^p(\Omega)}-  c_1 \|f\|^{p'}_{\VspOm'} - c_1 \norm{\extns g}_{\VspOm}^p
	\end{align*}
	for some constant $c_1$ depending on $p$ and on the constant from \autoref{pro:poinc-robust}. Thus, the functional $I$ is coercive in the sense that $I(v) \to +\infty$ for $\|v\|_{\VspOm} \to +\infty$. We have shown that $I$ attains a unique minimizer $u$ on the set $V_g^{s,p}(\Omega|\R^d)$. It is now straightforward to show that the function $u$ solves problem \eqref{eq:p-Lap-eq}. The claimed estimate follows from $I(u) \leq I(\extns g)$, the above estimate and \autoref{th:trace-extension}.
\end{proof}

Let us quickly review some related results on problems for nonlocal operators in bounded domains with given exterior data. Note that there are also approaches to nonlocal problems in bounded domains $\Omega$ with data given on $\partial \Omega$ such as \cite{Grub15}, which we do not take into account here. 

\medskip

Some early well-posedness results for variational nonlocal problems of the type \eqref{eq:p-Lap-eq} can be found in \cite{SeVa12}, \cite{SeVa13}, \cite{FKV15}. The case of homogeneous problems, i.e. when $g=0$, is particularly simple and has been treated by several authors. Note that the vector space $\widetilde{D}^{s,p}(\Omega)$ in \cite{PiPu17} equals the space $\VspOmNull$. Existing results for non-zero data $g$ often assume $g$ to be regular in all of $\R^d$, e.g. \cite[Theorem 2.3]{DKP16}, \cite{DKP16}, \cite[Theorem 8]{LiLi17} and \cite[Proposition 2.2]{ABH19}. As \cite[Example 1]{KKP17} shows, optimal results require extra care and more regularity than just suitable integrability of $g$ in $\R^d$. Also $g\in W^{s,p}(\Omega)\cap L^{p}(\R^d; (1+\abs{x})^{-d-sp}\d x)$ does not imply well-posedness as claimed in \cite{Pal18}, which is not essential at all for the main results of \cite{Pal18}. Workarounds avoiding global $W^{s,p}(\R^d)$-regularity are used in  \cite{KKP16}, \cite[Lemma 6]{KKP17}, and \cite[Definition 2.10]{BLS18}. These approaches assume $g \in W^{s,p}(\Omega_\varepsilon) \cap L^{p}(\R^d; (1+\abs{x})^{-d-sp}\d x)$ for some enlarged domain $\Omega_\varepsilon=\{ x\in \R^d\,|\, \distw{x}{\overline{\Omega}}<\eps \}$. Concerning the case $p=1$, we refer to \cite{BDLMV23} for results on existence and regularity of solutions to \eqref{eq:p-Lap-eq} with given exterior data.

\medskip

Note that well-posedness and energy estimates similar to \eqref{eq:estim-energy} are proved for $p=2$ in \cite{FoKa22} and for general $p$ in \cite{Fog23}. The present work resolves the matter of optimal assumption on exterior data $g$, which has been achieved for $p=2$ and $C^{1,1}$-domains in \cite{GrHe22}.

\begin{remark}
	Note that the fractional $p$-Laplacian is well defined in a point $x \in \R^d$ if $u$ is sufficiently regular in a neighborhood of $x$ and $u \in L^{p-1}(\R^d; (1+\abs{x})^{-d-sp}\d x)$. For a variational approach, the tail-space $L^{p}(\R^d; (1+\abs{x})^{-d-sp}\d x)$ is more natural, but modifications are possible. 
\end{remark}

\begin{remark} For demonstration purposes, we have presented the well-posedness result for the fractional $p$-Laplacian. It is straightforward to extended it to more general nonlinear operators of the form
	\begin{align*}
	\operatorname{pv.} \int\limits_{\R^d} \Phi \big(x, |u(x)-u(y)| \big) \big(u(x) - u(y) \big) k(x,y) \d y \,.
	\end{align*}
	for appropriate functions $\Phi$ and kernels $k$.
\end{remark}

\medskip

\subsection*{Related results} Let us discuss related results concerning function spaces, in particular trace theorems. As explained above, the main new feature of the energy space $\VspOm$ is that functions in $\VspOm$ satisfy some incremental regularity across the boundary plus some integrability at infinity. \cite{DyKa19} provides trace and extension results for $\VspOm$ for rather general domains $\Omega$\footnote{Note that in \cite{DyKa19} the domain of integration in (1.6) and (1.7) has to be changed from $\Omega^c \times \Omega^c$ to $M \times \Omega^c$ with $M = \{x \in \Omega^c | \dist(x, \partial \Omega) < 1\}$}. The proof is based on a Whitney decomposition of $\Omega$ and $\Omega^c$, which we employ here, too. However, the construction of the extension operator in \cite{DyKa19} is much simpler and uses the Lebesgue measure. Thus, for $s \to 1^{-}$, one does not recover the classical extension result. In order to resolve this problem, we introduce the measure $\mu_s$ on $\Omega^c$, which converges to the surface measure on $\partial \Omega$.

In \cite{BGPR20} the authors prove a version of the Douglas identity and provide trace and extensions results for spaces like $V^{s,2}(\Omega | \R^d)$, where they allow for a large class of L\'{e}vy measures  $\nu(\d h)$ instead of $|h|^{-d-2s} \d h$. The proof is based on a careful study of the Poisson kernel and provides a representation of the energy of the solution $u$ to problems like \eqref{eq:p-Lap-eq} ($p=2$) in terms of its trace on $\Omega^c$. The article leaves open the question of robustness as $s \to 1^{-}$. Unlike \cite{BGPR20}, we define the trace space for general $p \geq 1$ with the help of explicitly given norms that allow for robustness and limit results as $s \to 1^{-}$. Extensions of the results in \cite{BGPR20} to some nonlinear cases, still based on $L^2$-Lévy integrable kernels, can be found in \cite{BGPR20b}. 

A systematic study of generalizations of the energy space $\VspOm$ in the case of $p=2$ and a L\'{e}vy measure $\nu(\d h)$ can be found in \cite{FoKa22}, where functional inequalities, well-posedness results and some nonlocal-to-local convergence results are provided. The trace space is shown to contain a certain weighted $L^2$-space of functions on $\Omega^c$. \cite{Fog23} contains extensions to the general case $p > 1$. Nonlocal energy spaces appear also in the context of Markov jump processes in \cite{Von21}. Here, the author considers the intersection with $L^2(\R^d; m)$, where $m(x) = \1_{\Omega}(x) + \mu(x) \1_{\Omega^c}(x)$ and $\mu(x)$ behaves like $\dist(x, \partial \Omega) ^{-2s}$ for $x$ close to $\partial \Omega$, see Remark 2.37 in \cite{FoKa22} for detailed comments. This approach together with functional inequalities and questions of well-posedness has been studied for more general kernels in \cite{Vu22}.

\medskip

The present work can be seen as an extension of results in \cite{GrHe22}. In the present work, we treat general bounded Lipschitz domains and the full range $p \geq 1$ instead of bounded $C^{1,1}$-domains and $p=2$ in \cite{GrHe22}. Both works use the measure $\mu_s$ but the construction of the extension operator is different. In the present work we employ the Whitney decomposition technique and not the Poisson extension. The study of nonlocal Neumann problems as in \cite{GrHe22} together with the asymptotic behavior for $s \to 1^{-}$ is possible in our framework, too. In order to keep the scope of this work reasonable, we defer this line of research until a later date.

\medskip

Last, let us mention recent trace and extension results for nonlocal function spaces, where problems similar to ours occur but the setup is conceptually different. In \cite{DuTi17} the trace space $H^{1/2}(\partial \Omega)$ is recovered as the trace space of a certain $L^2(\Omega)$-space with a nonlocal interaction kernel that has a localizing property at the boundary $\partial \Omega$. The analogous result for $W^{s-\frac{1}{p},p}(\partial \Omega)$ is proved in \cite{DMT22}. The result is extended further to domains with very rough boundaries including those with spatially varying dimension in \cite{Fos21}. See \cite{ScDu23} for applications to nonlocal equations with Dirichlet data given on $\partial \Omega$. Given a localization parameter $\delta >0$ and a domain $\Omega$, the authors of \cite{DTWY22} study trace and extension operators between the domain and a layer $\{x \in \Omega ^c | \dist (x, \partial \Omega) < \delta\}$. The operators are shown to be robust as $\delta \to 0$, which makes it possible to recover classical trace results. For more details we refer to the discussion in \cite[Section 1.2]{GrHe22}

\medskip

The development of nonlocal functions spaces and related trace and extension results benefits greatly from classical results for Sobolev, Sobolev-Slobodeckij, or Besov spaces. Early results on trace spaces for $W^{1,p}(\Omega)$ can be found in \cite{trace_aronzajin}, \cite{trace_Slobodeckij_original}, \cite{trace_prodi} \cite{Gag57}, \cite{trace_Slobodeckij} and the monograph \cite{Nec67}. See \cite{Nec12} for an English translation and, in particular, Chapter 2.5 therein. Lipschitz domains and fractional-order spaces are covered in \cite{Gri85}, e.g., in Theorem 1.5.1.3 and Theorem 1.5.2.1. For domains with corners see also \cite{Jakovlev}. The corresponding state-of-the-art around the year 1970 is summarized in Chapter 1, Sections 7--9 of \cite{lions}. \cite[Chapter IV]{BIN75} is another standard reference focusing on contributions of researchers from the Soviet Union. Another important monograph in this direction is \cite{triebel_1}, in particular Chapter 3.3.3 and Chapter 3.3.4. Trace and extensions results are provided in \cite{Mar87} under minimal regularity assumptions on the domains. A survey of results on boundary value problems for higher-order elliptic equations with degeneracies along the boundary is given in \cite{NLW88}. \cite{Kim07} extends well-known trace assertions for weighted Sobolev spaces.The aforementioned list is rather selective and not complete at all. Even some fundamental problems such as a trace result for $H^s(\Omega)$, $1 < s < \frac32$ and $\Omega$ a bounded Lipschitz domain are not covered in the list above, see \cite{Din96}.

\medskip

Very useful references for our work are contributions of A. Jonsson and H. Wallin in \cite{JoWa78, JoWa84, Jon94}. The setting in the aforementioned references includes results for subsets of the Euclidean space endowed with general doubling measures. This is related to our framework because we consider measure spaces $(\Omega^c; \mu_s)$ with $\mu_s$ as in \eqref{eq:def_mu_s}. Moreover, the construction used in the proof of the extension result \autoref{th:trace-extension} is inspired by the corresponding results Theorem 3.1 and Theorem 4.1 in \cite{JoWa78}.

\medskip

\subsection*{Outline} In \autoref{sec:prelims} we fix the notation and shortly introduce function spaces used throughout this work. The trace embeddings are studied in \autoref{sec:tracetheorem}. We divide the proofs of the trace results from \autoref{th:trace-extension} and \autoref{th:trace-p-1} into the $L^p$-embedding, see \autoref{prop:trace_Lp_estimate}, and the seminorm inequality, see \autoref{prop:trace_seminorm_estimate}. We construct the extension operator in \autoref{sec:extension}. The extension theorems are proven in \autoref{prop:extension_continuity_Lp_part} as well as \autoref{prop:extension_continuity_seminorm} with precise dependencies of the operator norms. Lastly, the limiting properties of the spaces $\cT^{s,p}(\Omega^c)$, see \autoref{th:asym_trace}, are proven in \autoref{sec:convergence_pointwise}.

\subsection*{Acknowledgment} 

The authors thank Juan Pablo Borthagaray for helpful discussions and Solveig Hepp and Thorben Hensiek for comments to make the arguments more perspicuous.

\section{Preliminaries}\label{sec:prelims}

\subsection{Notation}

For two real numbers $a,b\in \R$ we write $a\wedge b = \min\{ a,b \}$, $a \vee b := \max \{ a,b \}$ and $\lfloor a \rfloor= \max (-\infty,a]\cap \Z$. The ball of radius $r>0$ centered at $x\in \R^d$ in the $d$-dimensional Euclidean space is written as $B_r(x)$ or $B_r^{(d)}$ whenever we want to specify the dimension. For a set $A$ we denote by $\1_{A}$ the indicator function on $A$. An open set $\Omega \subset \R^d$ is said to have a uniform Lipschitz boundary, if there exists a localization radius $r>0$ and a constant $L>0$ such that for any boundary point $z\in \partial \Omega\ne \emptyset$ there exists a translation and rotation $T_z:\R^d\to \R^d$ satisfying $T_z(z)=0$ as well as a Lipschitz continuous function $\phi_z:\R^{d-1}\to \R$ whose Lipschitz constant is bounded by $L$ such that $T_z(\Omega\cap B_{r}(z))=\{ (x',x_d)\in B_{r}(0)\,|\, \phi_z(x')>x_d \}$, see e.g. \cite[Definition 13.11]{Leo17} and the discussion in \cite[Chapter 1.2.1]{Gri85}. An open set $\emptyset \ne B\subset \R^d$ is said to satisfy the uniform exterior cone condition, if we find an opening angle $\alpha$ and a height $h_0>0$ such that for any $z\in \partial \Omega$ there exists an exterior cone $\cC_z\subset \overline{\Omega}^c$ with apex at $z$, height $h_0$. The notion of uniform interior cone condition is defined analogously. Note that an open set with a uniform Lipschitz boundary satisfies both uniform interior and exterior cone condition. The interior cones (resp. the exterior cones) can simply be constructed via $\cC_z:= T_z^{-1}\{ (x',x_d)\in B_r(0)\,|\, x_d<-L/2 \abs{x'} \}$ for $z\in \partial \Omega$. We say that $\Omega \subset \R^d$ is a Lipschitz domain, if it is open and has a uniform Lipschitz boundary. Notice that we do not assume $\Omega$ to be connected. Nevertheless, a bounded Lipschitz domain has finitely many connected components since the uniform interior cone condition bounds the volume of each connected component from below by a uniform positive constant. By $\distw{x}{A}= \inf\{ \abs{x-a}\,|\, a\in A \}$ we denote the distance of $x$ to a closed set $A\subset \R^d$. When the dependencies are clear, we write for short $d_x:= \distw{x}{\partial \Omega}$ for any $x\in \R^d$. Furthermore, we use for $r>0$ the notation 
\begin{align}\label{eq:def-Omega-decomp}
\begin{split}
	\Omega_r^{\text{ext}}:= \{ x\in \overline{\Omega}^c\,|\, d_x<r\}, \quad \Omega^r_{\text{ext}}:= \{ x\in \overline{\Omega}^c\,|\, d_x \ge r\}.
\end{split}
\end{align}
We denote by $\cH^{(d-l)}$ the normalized $(d-l)$-dimensional Hausdorff measure on $\R^d$. The surface measure of the $(d-1)$-dimensional unit sphere will be written for short as $\cH^{(d-1)}(\partial B_1)=\omega_{d-1}$. To shorten the notation, we write $\sigma$ for the surface measure on $\partial \Omega$. The inner radius of the domain $\Omega$, we denote by 
\begin{equation*}
	\inr{\Omega}:=\sup\{ r\,|\, B_r\subset \Omega \}.
\end{equation*}
We will use lower-case letters $c_1,c_2, \dots$ with running indices as constants in our proofs and reset them after every proof. When we introduce a new constant, we write $C=C(\cdots)$ to denote what the constant depends on, \ie $C=C(d,\Omega)>0$ depends only on the dimension $d$ and the set $\Omega$.  

\subsection{Function spaces} The following function spaces will be used throughout this work. We denote by $W^{s,p}(\Omega)$ (resp. $W^{s,p}(\partial \Omega)$), $s\in (0,1), p\ge 1$, the Sobolev-Slobodeckij space of functions in $u\in L^p(\Omega)$ satisfying $[u]_{W^{s,p}(\Omega)}:=[u]_{V^{s,p}(\Omega\,|\, \Omega)}<\infty$ endowed with the norm $\norm{u}_{W^{s,p}(\Omega)}^p:= \norm{u}_{L^p(\Omega)}^p + [u]_{W^{s,p}(\Omega)}^p$ (resp. $\partial \Omega$ with the surface measure). See \eqref{eq:V-seminorm-def} for the definition of the seminorm $[\cdot]_{V^{s,p}(A\,|\, B)}$. We write $BV(\Omega)$ for the space of functions $u\in L^1(\Omega)$ with bounded variation endowed with the norm $\norm{u}_{BV(\Omega)}:= \norm{u}_{L^1(\Omega)}+ \abs{\nabla u}(\Omega)$. The Bessel-potential spaces $H^{s,p}(\R^d)$ are defined in \eqref{eq:bessel_potential_spaces}. As mentioned in the introduction, a variational approach to equations like \eqref{eq:p-Lap-eq} leads naturally to function spaces like $V^{s,p}(\Omega\,|\, \R^d)$ which we introduced in \eqref{eq:V-def}. These function spaces are the focus of our study. They were first introduced in the works \cite{SeVa12, SeVa14, FKV15} for the case $p=2$. We also refer to the work \cite{DRV17}, in which the nonlocal normal operator was introduced, and \cite{Fog20,FoKa22,Fog23} for an intensive study of these spaces for general $p$. It is well known that $V^{s,p}(\Omega\,|\, \R^d)$ is a separable Banach space (Hilbert space for $p=2$) which is reflexive for $1<p<\infty$, see e.g. \cite[Chapter 3.4]{Fog20}

\medskip

The spaces $V^{s,p}(\Omega\,|\, \R^d)$ allow for a Poincar\'{e} inequality, which is an important ingredient for the proof of well-posedness for the Dirichlet problem \eqref{eq:p-Lap-eq} together with an energy estimate, see \autoref{cor:well_posedness}. We will need a version of the Poincaré inequality that is robust as $s$ reaches $1$. 

\begin{proposition}[{\cite[Theorem 10.1]{Fog23}}] \label{pro:poinc-robust}
	Let $p > 1$ and $s_\star \in (0,1)$. Let $\Omega \subset \R^d$ be a bounded Lipschitz domain. Then there exists $c>0$ such that for all $s_\star \leq s < 1$ and $u \in \VspOmNull$
	\begin{align}
	\|u\|_{L^p(\Omega)} \leq c \|u\|_{\VspOm} \,.
	\end{align}
\end{proposition}	

Let us recall our trace spaces $\cT^{s,p}(\Omega^c)$, which are introduced in \eqref{eq:Tsp_definition}. For $s\in (0,1), 1\le p < \infty$ and $A,B\in \cB(\Omega^c)$ we define 
\begin{align}\label{eq:Tsp_definition-seminorm}
[f,g]_{\cT^{s,p}(A\,|\, B)}^p &:= \iint\limits_{A\, B} \frac{\abs{f(x)-f(y)}^{p-2} (f(x)-f(y)) (g(x)-g(y)) }{  ((\abs{x-y}+d_x+d_y)\wedge 1)^{d+s(p-2)}} \mu_s(\d x) \mu_s(\d y)
\end{align}
with the convention $[g]_{\cT^{s,p}(A\,|\,B)}=[g,g]_{\cT^{s,p}(A\,|\, B)}$. Note that $[f,g]_{\cT^{s,p}(\Omega^c)}=[f,g]_{\cT^{s,p}(\Omega^c\,|\, \Omega^c)}$. We employ standard techniques to prove that these spaces are separable Banach spaces respectively Hilbert spaces if $p=2$. 
	\begin{lemma}\label{lem:banach}
			Let $\Omega$ be an open set. The space $\cT^{s,p}(\Omega^c)$ is a separable Banach space, reflexive for $1<p<\infty$ and in case $p=2$ it is a separable Hilbert space with inner product
			\begin{equation*}
				(u,v)_{\cT^{s,2}(\Omega^c)}= (u,v)_{L^2(\Omega^c,\, \mu_s )} + [u,v]_{\cT^{s,2}(\Omega^c)}^2.
			\end{equation*}
		\end{lemma}
		\begin{proof}
			To prove completeness, we take a Cauchy sequence $\{u_n\}_n\subset \cT^{s,p}(\Omega^c)$. Then $v_n(x):= u_n(x)\mu_s(x)^{1/p}$ is Cauchy in $L^p(\Omega^c)$ with limit $v\in L^p(\Omega^c)$. Define $u(x):= v(x)\mu_s(x)^{-1/p}$ then $u$ is the limit of $u_n$ w.r.t. $\norm{\cdot}_{L^{p}(\Omega^c;\,\mu_s)}$. Take a subsequence $\{ u_{n_l} \}_l$ converging a.e. to $u$ on $\R^d$. Then, by Fatou's lemma we have
			\begin{align*}
				[u-u_{n_l}]_{\cT^{s,p}(\Omega^c)}^p \le \liminf\limits_{k\to \infty} [u_{n_k}-u_{n_l}]_{\cT^{s,p}(\Omega^c)}^p\to 0  \text{ as } l\to \infty.
			\end{align*}
			Separability follows from the fact that the map $\iota: \cT^{s,p}(\Omega^c)\to L^p(\Omega^c)\times L^p(\Omega^c\times \Omega^c)$
			\begin{equation*}
			\iota:u\mapsto \Bigg(x\mapsto u(x) \mu_s(x)^{1/p},\, (x,y)\mapsto \frac{u(x)-u(y) }{ ((\abs{x-y}+d_x+d_y)\wedge 1)^{d/p+s(p-2)/p}}\mu_s(x)^{1/p}\mu_s(y)^{1/p} \Bigg)
			\end{equation*}
			is an isometric injection. As $\iota(\cT^{s,p}(\Omega^c))$ is closed and since $L^p(\Omega^c)\times L^p(\Omega^c\times \Omega^c)$ is separable, so is $ \cT^{s,p}(\Omega^c)$. In the same manner, as $L^p(\Omega^c)\times L^p(\Omega^c\times \Omega^c)$ is reflexive for $1<p<\infty$ so is $\cT^{s,p}(\Omega^c)$.
		\end{proof}

	The functions from $\cT^{s,p}(\Omega^c)$ have some regularity at the boundary because the weight in the seminorm $[\cdot,\cdot]_{\cT^{s,p}(\Omega^c)}$ becomes $\big((\abs{x-y})\wedge 1\big)^{-d-s(p-2)}$ as $x,y\to \partial \Omega$. Thereby, for sufficiently large $s$ the functions in the trace space $\cT^{s,p}(\Omega^c)$ themselves have a trace onto the boundary $\partial \Omega$. This is a direct consequence of \autoref{th:trace-extension}.
	\begin{corollary}\label{cor:trace_of_trace}
		The space $\cT^{s,p}(\Omega^c)$ is continuously embedded in $W^{s-1/p, p}(\partial \Omega)$ for any $s\in (1/p, 1)$ and $p\in (1,\infty)$. The embedding is surjective. The continuity constant depends only on $\Omega$, $p$ and a lower bound on $s$.
	\end{corollary}
	\begin{proof}
		By \autoref{th:trace-extension}, the extension $\extns: \cT^{s,p}(\Omega^c)\to \VspOm$ is continuous and the continuity constant $c_1>0$ depends only on $\Omega$, $p$ and a lower bound on $s$. The space $\VspOm$ is embedded in $W^{s,p}(\Omega)$ with the embedding constant depending only on a lower bound on $s$. The result follows from the classical trace result $W^{s,p}(\Omega)\to W^{s-1/p, p}(\partial \Omega)$. The embedding is surjective since we can extend a function from $W^{s-1/p,p}(\partial \Omega)$ to an element from $W^{s,p}(\R^d) \hookrightarrow V^{s,p}(\Omega\,|\, \R^d)\hookrightarrow \cT^{s,p}(\Omega^c)$.
	\end{proof}

	\section{Trace Theorem}\label{sec:tracetheorem}
	The aim of this section is to prove the trace part results of \autoref{th:trace-extension} and \autoref{th:trace-p-1}. This proof is carried out in \autoref{prop:trace_Lp_estimate} and \autoref{prop:trace_seminorm_estimate}. Essential building blocks in the respective proofs are an approximation to the classical $L^p$-trace embedding in  \autoref{th:trace_classical_approximation_inequality_1} and for $p=1$ a Hardy-type inequality provided in \autoref{th:Hardy}. On a more technical level, we use upper and lower bounds of the distance function, see \autoref{lem:distance_upper} and \autoref{lem:distance_lower}. 
	
	\medskip
	
	To prove \autoref{th:trace_classical_approximation_inequality_1} we apply techniques developed in \cite{JoWa84}. In particular, we use the interpolation between Bessel potential spaces on $\R^d$. For this reason we need a Sobolev extension operator for fractional Sobolev spaces $W^{s,p}(\Omega)$ whose continuity constant is independent of $s$. The existence of such an extension is well known in the literature. We provide this result in the following theorem for the convenience of the reader.
	
	\begin{theorem}[{\cite[Chapter VI.2 Theorem 3]{JoWa84}, \cite{Tri95}}]\label{th:sobolev_extension}
	
	Let $\Omega\subset \R^d$ be a connected Lipschitz domain. There exists a linear map $E$, which extends measurable functions $f:\Omega \to \R$ such that $E:L^p(\Omega)\to L^p(\R^d)$ for all $p \geq 1$ and, with some constant $C=C(d,\Omega, p)>0$, for any $0 < s \leq 1$
		\begin{equation*}
			\norm{Ef}_{W^{s,p}(\R^d)}\le C\, \norm{f}_{W^{s,p}(\Omega)}.
		\end{equation*}
	\end{theorem}
	\begin{remark}
		The extension is constructed via a Whitney decomposition of $\overline{\Omega}^c$, a smooth partition of unity and copying mean values of $f$ from inside to respective Whitney cubes. The construction of the extension $Ef$ is independent of $s,p$ and satisfies $E: W^{1,p}(\Omega)\to W^{1,p}(\R^d)$. Real interpolation allows to choose the constant $C(d,\Omega, p)$ in the theorem independent of $s$.
	\end{remark}
	
	Analogous to the measure $\mu_s$ from \eqref{eq:def_mu_s}, we define for $s\in (0,1)$ the measure 
	\begin{equation}\label{eq:tau_s_trace}
		\tau_{s}(\d x)= \frac{1-s}{d_x^s}\1_{\Omega}(x) \d x 
	\end{equation}
	on the $\sigma$-algebra $\cB(\R^d)$. Recall that $d_x= \distw{x}{\partial\Omega}$. The measure $\tau_{s}(\d x)$ plays the same role as $\mu_s$ but is supported inside $\Omega$. We use it in \autoref{th:trace_classical_approximation_inequality_1} for the proof of the trace part of our main theorems, see also \autoref{prop:trace_Lp_estimate} and \autoref{prop:trace_seminorm_estimate}. In contrast to $\mu_s$ the measure $\tau_s$ does not need the additional term $(1+d_x)^{-d-s(p-1)}$ for the decay at infinity since the open set $\Omega$ is assumed to be bounded throughout this work. The following lemma shows how balls scale under $\tau_s$. This scaling plays a crucial role in \autoref{th:trace_classical_approximation_inequality_1}.

	\begin{lemma}\label{lem:estimates_balls_tau_s}
		Let $\Omega\subset \R^d$ be a bounded Lipschitz domain with a localization radius $r_0>0$. There exists a constant $C=C(d,\Omega)>0$ such that for any $s\in (0,1)$, $0<r\le r_0/2$, and $x\in \Omega$
		\begin{equation}\label{eq:doubling_balls_tau_s}
			\tau_s(B_r(x))\le C\, r^{d-s}.
		\end{equation}
	\end{lemma}
	\begin{proof} Let $d\ge 2$. If $r\le d_x$, \ie $B_r(x)\subset\subset \Omega$, then $d_y\ge r-\abs{x-y}$ for any $y\in B_r(x)$ and, thus, 
		\begin{align*}
			\tau_s(B_r(x))= \int\limits_{B_r(x)} \frac{1-s}{d_y^s}\d y\le \int\limits_{B_r(x)} \frac{1-s}{(r-\abs{x-y})^s}\d y = \omega_{d-1} \int\limits_{0}^r \frac{1-s}{(r-t)^s} t^{d-1} \d t \le \omega_{d-1} r^{d-s}.
		\end{align*}
		Now we consider the case $r>d_x$, \ie $B_r(x)\cap \partial \Omega\neq \emptyset$. Without loss of generality we assume that $0\in \partial \Omega$ is a minimizer of $d_x$. Since $\Omega$ is a Lipschitz domain, we find a Lipschitz map $\phi:\R^{d-1}\to \R$ such that $\Omega \cap B_{r_0}= \{ (y',y_d)\in B_{r_0}\,|\, y_d<\phi(y') \}$. The Lipschitz constant of $\phi$ is bounded by $L>0$ independent of $x$. A simple calculation yields for any $y=(y',y_d)\in B_{r}(x)\cap \Omega$
		\begin{align}\label{eq:distance_estimate}
			\abs{y}&\le \abs{x}+\abs{y-x}\le 2r,\nonumber\\
			\abs{y_d}&\le \inf\limits_{ (\tilde{y}', \phi(\tilde{y}'))\in B_{r_0} }\abs{y_d-\phi(\tilde{y}')} + \abs{\phi(y')-\phi(\tilde{y}')}\nonumber\\
			&\le 2^{1/2}(1+L) \inf\limits_{ (\tilde{y}', \phi(\tilde{y}'))\in B_{r_0} } \abs{y-(\tilde{y}',\phi(\tilde{y}'))}= 2^{1/2}(1+L) d_y.
		\end{align}
		In the case that the minimizer of $d_y$ is not in graph of $\phi$, we simply pick a smaller $r_0$ depending only on the constant $L$. Therefore, 
		\begin{align}\label{eq:distance_integral_estimate}
			\begin{split}
			\tau_s(B_r(x))&\le 2^{s/2}(1+L)^{s} \int\limits_{B_{2r} \cap \{ y_d<\phi(y') \}} \frac{1-s}{\abs{y_d-\phi(y')}^s}\d (y',y_d)\\
			&\le 2(1+L)^s \omega_{d-2}(2r)^{d-1} \int\limits_{0}^{(2+L)r} \frac{1-s}{y_d^s}\d y_d\le 2^{d+1}(2+L)\omega_{d-2} r^{d-s}.
		\end{split}
		\end{align}
		The proof in case of $d=1$ is straightforward. Note that similar arguments as in this proof are employed in the proof of \autoref{lem:estimate_a_{Q,s}}.
	\end{proof}
		
	In the proof of \autoref{th:trace_classical_approximation_inequality_1} we use interpolation results, which we explain now. Let $G_\alpha$, $\alpha>0$ be the Bessel potential kernel. We introduce the Bessel potential spaces \begin{equation}\label{eq:bessel_potential_spaces}
		H^{\alpha,p}(\R^d):=\{ g:\R^d\to \R \,|\, \exists f\in L^p(\R^d): g= G_\alpha \ast f \}
	\end{equation} with the canonical norm $\norm{g}_{H^{\alpha,p}(\R^d)}:= \norm{f}_{L^p(\R^d)}$ if $g\in H^{\alpha,p}(\R^d)$ and $g=G_\alpha \ast f$. We recall the real interpolation result 
	\begin{equation}\label{eq:interpolation}
		[H^{\alpha_0,p}(\R^d), H^{\alpha_1,p}(\R^d)]_{p\theta}= W^{s,p}(\R^d)
	\end{equation}
	where $0<\alpha_0<\alpha_1$, $\theta\in (0,1)$, $s=(1-\theta) \alpha_0 + \theta \alpha_1$ and $p\ge 1$, see \eg \cite[Theorem 6.2.4]{BeLo76}. Analogous to \cite[Chapter V]{JoWa84}, we calculate bounds on the Bessel potential kernel $G_{\alpha_i}$ for some $0<\alpha_0<s<\alpha_1$, see \autoref{lem:help_trace_C}, and prove an approximate trace result inside $\Omega$, see \autoref{th:trace_classical_approximation_inequality_1}. 

	The following lemma is a slight modification of \cite[Lemma C]{JoWa84} that fits our setting. The well-known estimates of the Bessel potential kernel, it's gradient and decay at infinity is crucial in the proof. For more details on the Bessel potential we refer to \cite[Chapter IV]{Tai64}. In particular, we need to pay attention to the constants and their dependencies. 
	\begin{lemma}[{\cite[Chapter V Lemma C]{JoWa84}}]\label{lem:help_trace_C}
		Let $\Omega\subset \R^d$ be a bounded connected Lipschitz domain, $0<s_\star\le s<1$ and $1<p_\star\le p\le p^\star<\infty$. We set 
		\begin{equation}\label{eq:alpha_i}
			\alpha_0:= s\frac{1+p}{2p}, \quad \alpha_1:= 1+\frac{s}{2p}.
		\end{equation}
		and $\beta_i:= \alpha_i - s/p$ for $i \in \{0,1\}$. There exists a constant $C=C(d,\Omega, p_\star,p^\star, s_\star)>0$ such that for all $0<r\le r_0/2$ and $f\in L^p(\R^d)$ we have
		\begin{align}
			\iint\limits_{ \substack{ \Omega\times \Omega \\ \abs{x-y}<r } } \abs{G_{\alpha_i} \ast f(x) - G_{\alpha_i} \ast f(y)}^p \tau_s(\d y)\tau_s(\d x)&\le C r^{p\beta_i} \norm{f}_{L^p(\R^d)}^p,\label{eq:help_trace_C_1}\\
			\int\limits_{\Omega} \abs{G_{\alpha_i} \ast f(x)}^p \tau_s(\d x)&\le C\, \norm{f}_{L^p(\R^d)}^p.\label{eq:help_trace_C_2} 
		\end{align}
	\end{lemma}
	\begin{proof}
		In \cite[Chapter V Lemma C]{JoWa84} the statement is proven for doubling measures satisfying \eqref{eq:doubling_balls_tau_s} under the assumption $0<\beta_i<1$ and $0<\alpha_i \ne d$. The proof uses estimates of the Bessel potential kernel $G_{\alpha}$, see \cite[Chapter V Lemma 1]{JoWa84}, \cite[Chapter V Lemma A]{JoWa84} and \cite[Chapter V Lemma B]{JoWa84}. Carefully inspecting the proof of \cite[Chapter V Lemma C]{JoWa84}, we find that the resulting constant depends on the constant $C$ from \eqref{eq:doubling_balls_tau_s}, a lower bound $0<\beta_{i,\star}\le \beta_i$ and an upper bound $\beta_i\le \beta_i^\star<1$ as well as a lower bound on $\abs{d-\alpha_i}$. We calculate
		\begin{align*}
			0<s_\star \frac{p_\star-1}{2p_\star}\le s\frac{p-1}{2p} &=\beta_0\le \frac{1}{2} <1,\\
			0<1-\frac{1}{p_\star}< 1-\frac{s}{2p}&= \beta_1 \le 1-\frac{s_\star}{p^\star}<1.
		\end{align*}
		Furthermore, we have
		\begin{equation*}
		\abs{d-\alpha_0}= d-s\frac{1+p}{2p}\ge (d-1)+ \frac{p-1}{2p}\ge \frac{p_\star-1}{2p_\star}> 0.
		\end{equation*}
		and
		\begin{align*}
			\abs{d-\alpha_1}=\begin{cases}
				d-1-\frac{s}{2p}\ge 1-\frac{1}{2p_\star}>0 &, d\ge 2\\
				\frac{s}{2p}\ge \frac{s_\star}{2p^\star}>0&, d=1.
			\end{cases}
		\end{align*}
		This yields the estimates with dependencies of the constants as claimed.
	\end{proof}
	\begin{theorem}[Approximate trace inequality]\label{th:trace_classical_approximation_inequality_1}
		Let $\Omega \subset \R^d$ be a bounded Lipschitz domain, $1<p_\star< p^\star<\infty$ and $s_\star\in(0,1)$ there exists a constant $C=C(d,\Omega,p_\star,p^\star, s_\star)>0$ such that for every $s\in (s_\star, 1)$, $p_\star\le p\le p^\star$ and $u\in W^{s,p}(\Omega)$
		\begin{align}
			\int\limits_{\Omega} \abs{u(x)}^p \tau_s(\d x) + \int\limits_{\Omega}\int\limits_{ \Omega} \frac{\abs{u(x)-u(y)}^p}{((\abs{x-y}+d_x+d_y)\wedge 1)^{d+s(p-2)}}\tau_s(\d y)\tau_s(\d x) &\le C\, \norm{u}_{W^{s,p}(\Omega)}^p\label{eq:approx_trace_inequality}.
		\end{align}
	\end{theorem}
	Before we give the proof of this theorem we want to motivate it. In anticipation of \autoref{sec:convergence_pointwise}, the LHS of \eqref{eq:approx_trace_inequality} converges in the limit $s\to 1$ to 
	\begin{equation*}
		\int_{\partial \Omega} \abs{u}^p \d \sigma + \int_{\partial\Omega\times \partial\Omega} \frac{\abs{u(x)-u(y)}^p}{(\abs{x-y}\wedge 1)^{d-1 + p(1-1/p)}} (\sigma\otimes\sigma)(x,y)\asymp \norm{u}_{W^{1-1/p, p}(\partial \Omega)}^p.
	\end{equation*}
	Thereby, we retrieve the classical trace result $W^{1,p}(\Omega)\to W^{1-1/p,p}(\partial \Omega)$ in the limit $s\to 1^{-}$.
	\begin{proof}
		Since $\Omega$ is a bounded open set with a uniform Lipschitz boundary, $\Omega$ decomposes into finitely many connected components $\Omega_i$, $i\in \{1, \dots, I\}$ and each $\Omega_i$ is a connected bounded Lipschitz domain. First, we prove \eqref{eq:approx_trace_inequality} for each $\Omega_i$.
		
		We define $\alpha_0$ and $\alpha_1$ as in \eqref{eq:alpha_i} depending on $p,s$. We set $\theta:= s(p-1)/((2+s)p)\in (0,1)$ and notice $0<s_\star(p_\star-1)/((2+s_\star)p_\star)\le \theta\le (p^\star-1)/(3\,p^\star)<1$. Most importantly, the relation $s=(1-\theta)\alpha_0+ \theta \alpha_1$ is true. By \autoref{th:sobolev_extension}, it is sufficient to prove the existence of a constant $C>0$ such that 
		\begin{equation*}
			\int\limits_{\Omega_i} \abs{u(x)}^p \tau_s(\d x)+
			\int\limits_{\Omega_i}\int\limits_{ \Omega_i} \frac{\abs{u(x)-u(y)}^p}{((\abs{x-y}+d_x+d_y)\wedge 1)^{d+s(p-2)}}\tau_s(\d y)\tau_s(\d x)  \le C\, \norm{u}_{W^{s,p}(\R^d)}^p
		\end{equation*}
		for any $u\in W^{s,p}(\R^d)$. Let $c_1>0$ be the constant from \autoref{lem:help_trace_C}. The equality \eqref{eq:help_trace_C_2} proves the continuity of the restriction operator $Ru(x)=u|_{\Omega_i}$ as a map from $R:H^{\alpha_i,p}(\R^d)\to L^p(\Omega_i, \tau_s(\d x))$, $i=0,1$. Real interpolation yields the continuity of 
		\begin{equation*}
			R:[H^{\alpha_0,p}(\R^d), H^{\alpha_1,p}(\R^d)]_{p\theta}=W^{s,p}(\R^d)\to [L^p(\Omega_i, \tau_s),L^p(\Omega_i, \tau_s)]_{p\theta}= L^p(\Omega_i, \tau_s)
		\end{equation*} 
		with the continuity constant $c_1$, see \eg \cite{BeLo76}. Now we consider the second term on the LHS of \eqref{eq:approx_trace_inequality} with $\Omega_i$ in place of $\Omega$. Let $u\in W^{s,p}(\R^d)$. 
		\begin{multline*}
			 \int\limits_{\Omega_i}\int\limits_{ \Omega_i} \frac{\abs{u(x)-u(y)}^p}{((\abs{x-y}+d_x+d_y)\wedge 1)^{d+s(p-2)}}\tau_s(\d y)\tau_s(\d x)\\
			 \le 2 \sum\limits_{n=0}^\infty 2^{ns(p-1)}\hspace{-1cm}\iint\limits_{\substack{\Omega_i\times \Omega_i\\ 2^{-n-1}\le \abs{x-y}<2^{-n} }}\hspace{-1cm} \abs{u(x)-u(y)}^p\frac{(\tau_s\otimes \tau_s)(\d (y,x))}{\abs{x-y}^{d-s}}+ \iint\limits_{\substack{\Omega_i\times \Omega_i\\ 1\le \abs{x-y} }} \hspace{-0.3cm}\abs{u(x)-u(y)}^p\tau_s(\d y)\tau_s(\d x)\\=: (\text{I}) + (\text{II}).
		\end{multline*}
		We estimate $(\text{II})$ using the continuity of $R$ shown above. We have
		\begin{align*}
			(\text{II})\le 2^{p} \tau_s(\Omega_i) \int\limits_{ \Omega_i } \abs{R u(x)}^p \tau_s(\d x)\le c_1^p \, 2^{p} \tau_s(\Omega_i) \norm{u}_{W^{s,p}(\R^d)}^p.
		\end{align*}
		We set for short $H:=L^p(\Omega_i\times \Omega_i, \abs{x-y}^{-d+s}\tau_s(\d y) \tau_s(\d x))$ and define for any $1<q\le \infty$, $\beta>0$ the space of sequences 
		\begin{align*}
			l^{\beta,q}&:=\{ (h_n)_n\,|\, h_n\in H \},\\
			\norm{(h_n)}_{l^{\beta,q}}&:= \norm{\Big( 2^{n\beta} \norm{h_n}_H \Big)_n}_{l^q(\N)}.
		\end{align*}
		Notice that 
		\begin{equation}\label{eq:continuity_trace_seminorm_sequence_connection}
			(\text{I})= \norm{\Big( (u(x)-u(y))\1_{2^{-n-1}\le \abs{x-y}<2^{-n}} \Big)_n}_{l^{s-s/p, p}}^p.
		\end{equation}
		We define the linear map
		\begin{align*}
			Tf(x,y):= \Big( (f(x)-f(y))\1_{2^{-n-1}\le \abs{x-y}<2^{-n}} \Big)_n, \quad f:\R^d\to \R.
		\end{align*}
		\autoref{lem:help_trace_C}, in particular \eqref{eq:help_trace_C_1}, shows the continuity of $T: H^{\alpha_i,s}\to l^{\beta_i,\infty}$ with $\beta_i= \alpha_i-s/p$ and the continuity constant $c_1$, $i=0,1$. Real interpolation yields the continuity of $T:[H^{\alpha_0,p}(\R^d), H^{\alpha_1,p}(\R^d)]_{p\theta}=W^{s,p}(\R^d)\to [l^{\beta_0,\infty}, l^{\beta_1,\infty}  ]_{p\theta}= l^{(1-\theta)\beta_0+\theta\beta_1,p}$ with the continuity constant $c_1$, see \eg \cite{BeLo76}. This proves the claimed inequality for each connected component $\Omega_i$ by \eqref{eq:continuity_trace_seminorm_sequence_connection} and
		\begin{equation*}
			(1-\theta)\beta_0+\theta\beta_1= (1-\theta)\alpha_0+ \theta \alpha_1 -s/p = s-s/p.
		\end{equation*}
	
		Simply summing over $i\in \{1,\dots, I\}$ yields a constant $c_2=c_2(d,\Omega, p_\star, p^\star,s_\star)>0$ such that 
		\begin{equation}\label{eq:Lp_embedding_Omega}
			\norm{u}_{L^p(\Omega;\, \tau_s)}^p \le c_2 \norm{u}_{W^{s,p}(\Omega)}^p.
		\end{equation}
		It remains to prove the existence of a constant $c_3=c_3(d,\Omega,p_\star, p^\star, s_\star)>0$ such that for any $i\ne j$ 
		\begin{equation*}
			\int\limits_{\Omega_i}\int\limits_{ \Omega_j} \frac{\abs{u(x)-u(y)}^p}{((\abs{x-y}+d_x+d_y)\wedge 1)^{d+s(p-2)}}\tau_s(\d y)\tau_s(\d x)\le c_3 \norm{u}_{W^{s,p}(\Omega)}^p. 
		\end{equation*}
		Since the distance between any two connected components is bounded from below by a uniform constant, this is an immediate consequence of the triangle inequality, \autoref{lem:estimates_balls_tau_s} as well as \eqref{eq:Lp_embedding_Omega}. 
	\end{proof}

\autoref{th:trace_classical_approximation_inequality_1} is not true in the case $p=1$, see \autoref{example:trace_p-1}. If we only keep the first term on the left-hand side in the estimate \eqref{eq:approx_trace_inequality}, then it is a fractional Hardy inequality, see e.g. \cite{Dyd04}. In \cite[Theorem 4]{DyKi22b} such a Hardy inequality is proven with a constant whose dependency on the parameter $s$ are not known. Since the dependency on $s$ is crucial in our setup, we prove the following theorem based on a Hardy inequality on the half space with the best constant, see \autoref{th:hardy_half_space}. 
\begin{theorem}[Hardy inequality]\label{th:Hardy}
	Let $\Omega\subset \R^d$ be a bounded Lipschitz domain and $s\in (0,1)$. There exists a constant $C=C(d,\Omega)>0$ such that 
	\begin{align*}
	(1-s)\int\limits_{\Omega} \frac{\abs{u(x)}}{d_x^s}\d x\le C\,\Bigg(\norm{u}_{L^1(\Omega)}+ s(1-s)\int\limits_{\Omega\times \Omega} \frac{\abs{u(x)-u(y)}}{\abs{x-y}^{d+s}} \d (x,y)\Bigg)
	\end{align*}
	for any $u\in W^{s,1}(\Omega)$. 
\end{theorem}
Before we state the proof of the theorem, let us remark that the previous inequality, in the limit $s\to 1^-$, yields the classical trace embedding $W^{1,1}(\partial \Omega)\to L^1(\partial \Omega)$ since the measure $\tau_s$ converges weakly to the surface measure on $\partial \Omega$. 
\begin{proof}
It is sufficient to prove the statement for any connected component of $\Omega$ in place of $\Omega$. Thus, we can assume without loss of generality that $\Omega$ is a connected bounded Lipschitz domain. Therefore, we can cover the boundary with finitely many neighborhoods $U_i$ and bi-Lipschitz maps $\phi_i:U_i\to B_1(0)$ such that $\phi_i( U_i\cap \Omega ) = B_1(0)_+:=\{ (x',x_d)\in B_1(0)\,|\, x_d>0 \}$, $i\in \{1,\dots, N\}$, see e.g. \cite[Chapter 1.2.1]{Gri85}. We denote the distance of $\Omega\cap \bigcap_{i=1}^N U_i^c$ to the boundary $\partial \Omega$ by $2r_0>0$. We fix $U_0:= \{ x\in \R^d\,|\, \distw{x}{\Omega\cap \bigcap_{i=1}^N U_i^c} <r_0  \}\subset \Omega$. Notice that $\{U_i\,|\, i=0, \dots, N\}$ is an open cover of $\overline{\Omega}$ and $\distw{U_0}{\partial \Omega}\ge r_0$. Next, we pick a partition of unity $\eta_i\in C_c^\infty(U_i)$ adapted to $U_i$, \ie $\sum_{i=0}^{N}\eta_i=1$ on $\overline{\Omega}$. We define $\tilde{\eta}_i:=\eta_i\circ \phi_i^{-1}\in C_{c}^{0,1}(B_1(0))$. Let $c_1=c_1(\tilde{\eta}_1, \dots, \tilde{\eta}_N)\ge 1$ such that $[\tilde{\eta}_i]_{C^{0,1}}\le c_1$ for all $i=1, \dots, N$. Without loss of generality we assume that $\tilde{\eta}_i=1$ in $B_{1/2}(0)$ for all $i=1, \dots, N$. Then 
\begin{align*}
\int\limits_{\Omega} \frac{\abs{u(x)}}{d_x^s}\d x &= \sum\limits_{i=1}^N \int\limits_{B_1(0)_+} \frac{\abs{u(\phi_i^{-1}(x) )}}{d_{\phi_i^{-1}(x)}^s} \eta_i(\phi_i^{-1}(x)) \abs{\det (D\phi_i^{-1}(x)) }  \d x   \\
&\qquad+ \int\limits_{U_0} \eta_0(x) \frac{\abs{u(x)}}{d_x^s}\d x =: \sum_{i=1}^{N} (\text{I}_i) + (\text{II}).
\end{align*} 
We define $u_i:= u \circ \phi_i^{-1}$ for all $i=1, \dots, N$. By the bi-Lipschitz continuity of the $\phi_i$'s we find a constant $c_2=c_2(\phi_1, \dots, \phi_N)>1$ such that $c_2^{-1} x_d\le d_{\phi_i^{-1}(x)}\le c_2 x_d$ for any $x\in B_1(0)_+$ and $i\in \{1,\dots, N\}$, see \eqref{eq:distance_estimate}. Further, we find a constant $c_3=c_3(\phi_1, \dots, \phi_N)\ge 1$ such that both $[\phi_i^{-1}]_{C^{0,1}}$ and $[\phi_i]_{C^{0,1}}$ are bounded from above by $c_3$ and from below by $c_3^{-1}$ for all $i$. We apply \autoref{th:hardy_half_space} to the function $\tilde{\eta}_i u_i$ to find:
\begin{align*}
(\text{I}_i)&\le c_2c_3^d \int\limits_{\R_+^d} \frac{\abs{\tilde{\eta}_i(x)u_i(x)}}{x_d^{s}} \d x \le c_2 c_3^d \cD_{s,1}^{-1} \int\limits_{\R^d_+ \times \R^d_+} \frac{\abs{\tilde{\eta}_i(x)u_i(x)-\tilde{\eta}_i(y)u_i(y)}}{\abs{x-y}^{d+s}}\d(x,y)\\
&\le c_2 c_3^d \,\cD_{s,1}^{-1} \Bigg( \,\int\limits_{B_1(0)_+\times B_1(0)_+} \!\!\!\!\!\!\! \tilde{\eta}_i(x) \frac{\abs{u_i(x)-u_i(y)}}{\abs{x-y}^{d+s}} \d (x,y) + \!\!\!\int\limits_{B_1(0)_+\times B_1(0)_+} \!\!\!\!\!\!\! \abs{u_i(y)} \frac{\abs{\tilde{\eta}_i(x) - \tilde{\eta}_i(y)}}{\abs{x-y}^{d+s}}\d (x,y)\\
&\qquad +2\, \int\limits_{B_1(0)_+} \tilde{\eta}_i(x) \abs{u_i(x)} \int\limits_{B_1(0)^c}\abs{x-y}^{-d-s}\d y \d x \Bigg)=: (\text{III}_i)+(\text{IV}_i)+(\text{V}_i)
\end{align*}
The first term in the previous estimate, \ie $(\text{III}_i)$, can be simply estimated using a change of variables and the bi-Lipschitz continuity of $\phi_i$: 
\begin{align*}
(\text{III}_i)\le c_2 c_3^{4d+s} \,\cD_{s,1}^{-1}\int\limits_{(U_i\cap \Omega)\times (U_i \cap \Omega)} \frac{\abs{u(x)-u(y)}}{\abs{x-y}^{d+s}} \d (x,y)
\end{align*}	
To estimate $(\text{IV}_i)$ we calculate 
\begin{align*}
\int\limits_{B_1(0)} \frac{\abs{\tilde{\eta}_i(x) -\tilde{\eta}_i(y) }}{\abs{x-y}^{d+s}}\d x \le c_1 \frac{\omega_{d-1}2^{1-s}}{1-s}.
\end{align*}
Using this, we find 
\begin{align*}
(\text{IV}_i)\le c_1 c_2c_3^d\, \cD_{s,1}^{-1} \frac{2\omega_{d-1}}{1-s}\int\limits_{B_1(0)_+} \abs{u_i(y)}\d y \le c_1 c_2 c_3^{2d}\, \cD_{s,1}^{-1} \frac{2\omega_{d-1}}{1-s} \int\limits_{U_i\cap \Omega} \abs{u(x)} \d x.
\end{align*}
Now, we estimate $(\text{V}_i)$. Since $\tilde{\eta}_i\in C_c^{0,1}(B_1)$, we find a constant $c_4\ge 1$ such that $\tilde{\eta}_i(x)\le c_4 (1-\abs{x})$ for all $i=1, \dots, N$. We notice for any $x\in B_1(0)$ 
\begin{align*}
\int\limits_{B_1(0)^c} \abs{x-y}^{-d-s}\d y \le \int\limits_{B_{(1-\abs{x})}(x)^c} \abs{x-y}^{-d-s}\d y =  \omega_{d-1} \frac{(1-\abs{x})^{-s}}{s} 
\end{align*}
and, thus, 
\begin{align*}
(\text{V}_i)\le 2 c_2 c_3^{d} c_4 \,\frac{\cD_{s,1}^{-1}}{s} \omega_{d-1} \int\limits_{B_1(0)_+} \abs{u_i(x)}\d x\le 2c_2 c_3^{2d} c_4 \, \frac{\cD_{s,1}^{-1}}{s} \omega_{d-1} \int\limits_{U_i\cap \Omega} \abs{u(x)}\d x. 
\end{align*}
To estimate $(\text{II})$ we simply notice that the distance function is bounded from below by $r_0$ on $U_0$. So finally, we put everything together. This yields
\begin{align*}
\int\limits_{\Omega} \frac{\abs{u(x)}}{d_x^{s}} \d x \le \frac{c_6}{1-s} \int\limits_{\Omega} \abs{u(x)}\d x + c_7 s \int\limits_{\Omega\times \Omega} \frac{\abs{u(x)-u(y)}}{\abs{x-y}^{d+s}}\d (x,y).
\end{align*} 
Here 
\begin{equation*}
c_6:= (r_0\wedge 1)^{-1}+  2\omega_{d-1} N c_1c_2c_3^{2d}c_4c_5 , \quad c_7:= N c_2 c_3^{4d+1} c_5 
\end{equation*}
and $c_5$ is the constant from \autoref{lem:constant}.
\end{proof}

The following two lemmata are technical tools which we employ in the proof of the trace result, see \autoref{prop:trace_Lp_estimate}, \autoref{prop:trace_seminorm_estimate}. They allow us to rewrite the distance functions appearing in the measure $\mu_s$ as an integral over $\Omega$. This enables us to use the regularity of the functions from $\VspOm$ in $\Omega$ when we prove the trace result. 
\begin{lemma}\label{lem:distance_upper}
	Let $\emptyset \ne B\subset \R^d$ be an open set, $s>0$ and $f:[0,\infty)\to [0,\infty)$ be a non-increasing function. For any $x\in \overline{B}^c$
	\begin{align*}
	\int\limits_{ B }\frac{f(\abs{x-z})}{\abs{x-z}^{d+s}} \d z \le \frac{\omega_{d-1}}{s}\,\frac{f(\distw{x}{B})}{\distw{x}{B}^s}
	\end{align*}  
	holds. If $B$ is bounded, then there exists a constant $C=C(d,B)$ such that for any $x\in \overline{B}^c$
	\begin{equation*}
	\int\limits_{ B }\frac{f(\abs{x-z})}{\abs{x-z}^{d+s}} \d z \le \frac{C}{s} \frac{f(\distw{x}{B})}{\distw{x}{B}^s\, (1+\distw{x}{B})^d}
	\end{equation*}
\end{lemma}
\begin{proof}
	Fix $x\in \overline{B}^c$. We use $B\subset B_{\distw{x}{B}}(x)^c$ and apply polar coordinates.
	\begin{equation*}
	\int\limits_{ B }\frac{f(\abs{x-z})}{\abs{x-z}^{d+s}} \d z\le \int\limits_{ B_{\distw{x}{B}}(x)^c}\frac{f(\abs{x-z})}{\abs{x-z}^{d+s}} \d z= \omega_{d-1} \int\limits_{\distw{x}{B}}^\infty f(t)\,t^{-1-s} \d t \le  \omega_{d-1} \frac{f(\distw{x}{B})}{s\,\distw{x}{B}^{s}}. 
	\end{equation*}
	In case that $\distw{x}{B}<1$, the second claim for bounded $B$ is a direct consequence of the first statement. If $B$ is bounded and $\distw{x}{B}\ge 1$, then 
	\begin{equation*}
	\int\limits_{ B }\frac{f(\abs{x-z})}{\abs{x-z}^{d+s}} \d z \le \abs{B} \frac{f(\distw{x}{B})}{\distw{x}{B}^{d+s}}\le \abs{B}2^{d} \frac{f(\distw{x}{B})}{\distw{x}{B}^{s}(1+\distw{x}{B})^d}.
	\end{equation*}
\end{proof}
\begin{lemma}\label{lem:distance_lower}
	Let $\emptyset \ne B\subset \R^d$ be an open set satisfying the uniform interior cone condition with a compact boundary. Then there exist two constants $C=C(d,B)>0$ such that for any non-increasing function $f:[0,\infty)\to [0,\infty)$ and any $s>0$
	\begin{equation*}
	\frac{f(2 \distw{x}{\overline{B}})}{\distw{x}{\overline{B}}^s(1+\distw{x}{\overline{B}})^d}  \le C\,  \int\limits_{ B } \frac{f(\abs{x-z})}{\abs{x-z}^{d+s}}\1_{ B_{1+\distw{x}{\overline{B}}}(x)}(z) \d z
	\end{equation*}
	for all $x\in \overline{B}^c$.
\end{lemma}
\begin{proof}Fix $x\in \overline{B}^c$ and a minimizer $x_0\in \partial B$ of the distance $\distw{x}{\overline{B}}$. Since $B$ satisfies the uniform interior cone condition we find an interior cone $C$ with apex at $x_0$ whose height $h_0$ and an open angle which is independent of $x_0$. Wlog. we assume $h_0\le 1$. Let $\widetilde{C}:= \{ z\in C\,|\, \abs{z-x_0}<\distw{x}{\overline{B}} \}$ be a subcone with a reduced height. Notice $\widetilde{C}\subset B_{1+\distw{x}{\overline{B}}}(x)\cap B$. For any $z\in \widetilde{C}$ we have $\abs{x-z}\le \abs{x-x_0}+\abs{x-z}\le \distw{x}{\overline{B}}+ \min\{\distw{x}{\overline{B}}, h_0\}\le 2 \distw{x}{\overline{B}}$. Thus, the claim simply follows from
\begin{multline*}
	\int\limits_{ B } \frac{f(\abs{x-z})}{\abs{x-z}^{d+s}}\1_{ B_{1+\distw{x}{\overline{B}}}(x)}(z) \d z \ge \frac{f(2\distw{x}{\overline{B}})}{(2\distw{x}{\overline{B}})^{d+s}} |\widetilde{C}|\\= f(2\distw{x}{\overline{B}}) \frac{c_1 \big( \min\{ \distw{x}{\overline{B}}, h_0 \} \big)^d}{(2\distw{x}{\overline{B}})^{d+s}}\ge \frac{c_1 h_0^d }{2^{d+1}}  \frac{f(2\distw{x}{\overline{B}})}{\distw{x}{\overline{B}}^{s}(1+\distw{x}{\overline{B}} )^d }.
\end{multline*}
Here we used $|\widetilde{C}|= c_1 \big( \min\{ \distw{x}{\overline{B}}, h_0 \} \big)^d$ where $c_1>0$ depends only on $d$ and the opening angle of $\widetilde{C}$ which is independent of $x_0$. 
\end{proof}

We are now in the position to prove the trace part in \autoref{th:trace-extension} and \autoref{th:trace-p-1}. We split the proof into two propositions. The following proposition contains the trace embedding $\VspOm\to L^p(\Omega^c;\,\mu_s(\d x))$ for all $1\le p<\infty$. The estimates of the seminorm $[\cdot]_{\cT^{s,p}(\Omega^c)}$ for $1<p<\infty$ are proven thereafter in \autoref{prop:trace_seminorm_estimate}. Recall the definition of the sets $\Omega_r^{\text{ext}}, \Omega^r_{\text{ext}}$ in \eqref{eq:def-Omega-decomp} for given $r>0$.
\begin{proposition}\label{prop:trace_Lp_estimate}
	Let $d\in \N$, $\Omega\subset \R^d$ be a bounded Lipschitz domain, $s_\star\in(0,1)$, $1< p^\star<\infty$. There exists a constant $C=C(\Omega, p^\star, s_\star)>0$ such that 
	\begin{equation}\label{eq:trace_Lp_continuity}
		\norm{\trns u}_{L^p(\Omega^c;\, \mu_s)}\le C\, \norm{u}_{\VspOm}
	\end{equation}
	for any $s\in (s_\star, 1)$, $1\le p\le p^\star$ and $u\in \VspOm$ 
\end{proposition}
\begin{proof}
	We split the integration domain of $\norm{\trns u}_{L^{p}(\Omega^c;\,\mu_s)}^p$ into $\Omega_1^{\text{ext}}$ and $\Omega^1_{\text{ext}}$. Let $c_1=c_1(d, \Omega)>0$ be the constant from \autoref{lem:distance_lower} when applied to $B= \Omega$ and $f=1$. We have
	\begin{align*}
	\norm{\trns u}_{L^{p}(\Omega_1^\text{ext};\,\mu_s)}^p&\le  (1-s) \int\limits_{\Omega_1^\text{ext}} \frac{\abs{u(x)}^p}{\distw{x}{\Omega}^s} \d x\le 2^{d} (1-s) \, c_1\int\limits_{\Omega_1^\text{ext}} \int\limits_{\Omega} \frac{\abs{u(x)}^p}{\abs{x-z}^{d+s}} \d z \d x\\
	&\le 2^{d+p} (1-s) \, c_1\Big[ \int\limits_{\Omega_1^\text{ext}} \int\limits_{\Omega} \frac{\abs{u(x)-u(z)}^p}{\abs{x-z}^{d+s}} \d z \d x+\int\limits_{\Omega_1^\text{ext}} \int\limits_{\Omega} \frac{\abs{u(z)}^p}{\abs{x-z}^{d+s}} \d z \d x \Big]\\
	&=: (\text{I})+ (\text{II}).
	\end{align*}
	The term $(\text{I})$ is estimated easily via
	\begin{equation*}
	(\text{I})\le 2^{d+p^{\star}}\, c_1 \, (\diam{\Omega}+ 1)^{p^{\star}-1}\, [u]_{\vs{s}{p}(\Omega\,|\, \Omega_1^{\ext})}^p.
	\end{equation*}
	An application of \autoref{lem:distance_upper} with $B=\Omega_1^{\text{ext}}$ and \autoref{th:trace_classical_approximation_inequality_1} respectively \autoref{th:Hardy} in the case $p=1$ yields the following bound on the term $(\text{II})$. 
	\begin{align*}
	(\text{II})&\le 	2^{d+p} (1-s)\, c_1\, \frac{\omega_{d-1}}{s}\,\int\limits_{\Omega} \frac{\abs{u(z)}^p}{d_z^s} \d z \le  2^{d+p^{\star}} \, \frac{\omega_{d-1}}{s_\star}\,c_1 c_2\big([u]_{W^{s,p}(\Omega)}^p + \norm{u}_{L^p(\Omega)}^p \big)
	\end{align*} 
	Here $c_2>0$ is the constant from \autoref{th:trace_classical_approximation_inequality_1} in the case $p>1$. In the case $p=1$ let $c_2$ be the constant from \autoref{th:Hardy}. For the estimate of $\Omega^1_{\text{ext}}$, we define $\diam{\Omega}+1=: c_3\ge 1$ and notice that $d_x\ge  \tfrac{\abs{x-z}}{c_3}$ as well as $\Omega_{\text{ext}}^1 \subset B_1(z)^c$ holds for any $z\in \Omega$ and $x\in \Omega_{\text{ext}}^1$. Thus,
	\begin{align}\label{eq:trace_estimate_Lp_Omega^1}
	\norm{\trns u}_{L^{p}(\Omega^1_{\text{ext}};\,\mu_s)}^p&\le 2^{p^{\star}}(1-s)  \int\limits_{\Omega^1_{\text{ext}}} \fint\limits_\Omega \frac{\abs{u(x)-u(z)}^p+\abs{u(z)}^p  }{d_x^s (1+d_x)^{d+s(p-1)}} \d z \d x\nonumber \\
	&\le 2^{p^{\star}} c_3^{d+sp^{\star}}(1-s) \int\limits_{\Omega^1_{\text{ext}}} \fint\limits_\Omega \frac{\abs{u(x)-u(z)}^p}{\abs{x-z}^{d+sp}} \d z \d x\nonumber\\
	&\qquad+  2^{p^{\star}}(1-s)  \int\limits_{\Omega^1_{\text{ext}}} \fint\limits_\Omega \frac{\abs{u(z)}^p  }{d_x^{d+sp}} \d z \d x\nonumber\\
	&\le \tfrac{ 2^{p^{\star}} c_3^{d+sp^{\star}}}{\abs{\Omega}} [u]_{\vs{s}{p}(\Omega\,|\Omega_{\text{ext}}^1)}^p+ \tfrac{\omega_{d-1}\,2^{p^{\star}}c_3^{d+sp^{\star}} (1-s)}{sp\abs{\Omega}} \norm{u}_{L^p(\Omega)}^p.
	\end{align}
	In the last step we used
	\begin{align}\label{eq:distance_integrated_at_infinity}
		\int\limits_{\Omega_{\text{ext}}^1} \frac{1}{d_x^{d+sp}}\d x \le c_3^{d+sp} \int\limits_{B_{1}(x_0)^c}\frac{1}{\abs{x_0-x}^{d+sp}}\d x = c_3^{d+sp} \frac{\omega_{d-1}}{sp},
	\end{align}
	where $x_0 \in \Omega$ is a fixed point. Combining the estimates of $(\text{I})$, $(\text{II})$ as well as \eqref{eq:trace_estimate_Lp_Omega^1} yields \eqref{eq:trace_Lp_continuity}. 
\end{proof}
\begin{proposition}\label{prop:trace_seminorm_estimate}
	Let $d\in \N$, $\Omega\subset \R^d$ be a bounded Lipschitz domain, $s_\star\in(0,1)$, $1<p_\star< p^\star<\infty$. There exists a constant $C=C(\Omega, p_\star,p^\star, s_\star)>0$ such that for any $s\in (s_\star, 1)$, $p_\star\le p\le p^\star$ and $u\in \VspOm$ 
\begin{equation}\label{eq:trace_continuity}
	[\trns u]_{\cT^{s,p}(\Omega^c)}\le C\, \norm{u}_{\vs{s}{p}(\Omega\,|\R^d)}.
	\end{equation}
\end{proposition}
\begin{proof}
	We fix $\rho:= \inr{\Omega}>0$ and divide the integration domain of $[\trns u]_{\cT^{s,p}(\Omega^c)}^p$ into $\Omega_\rho^{\text{ext}}\times \Omega_\rho^{\text{ext}}$, $\Omega^c\times \Omega^\rho_{\text{ext}}$ and $\Omega^\rho_{\text{ext}}\times \Omega^c$. By symmetry the estimates for $\Omega^c\times \Omega^\rho_{\text{ext}}$ and $\Omega^\rho_{\text{ext}}\times \Omega^c$ are equivalent. Thus, we settle on $\Omega^\rho_{\text{ext}}\times \Omega^c$. Since $\abs{x-y}+d_x+d_y\ge \rho$ for any $x\in \Omega^c$ and $y\in \Omega_{\text{ext}}^\rho$, we have 
	\begin{align}
	[\trns u]_{\cT^{s,p}(\Omega^\rho_{\text{ext}}\,|\, \Omega^c)}^p &\le   \int\limits_{\Omega^\rho_{\text{ext}}} \int\limits_{\Omega^c}  \frac{\abs{u(x)-u(y)}^{p}}{(1\wedge \rho)^{d+p^\star}}  \mu_s(\d x) \mu_s(\d y)\nonumber\\
	&\le \frac{2^{p} (1-s)^2}{(1\wedge \rho)^{d+p^\star}} \, \Bigg[\,\, \int\limits_{\Omega^c}  \frac{\abs{u(x)}^{p}}{d_x^s (1+d_x)^{d+s(p-1)}} \int\limits_{\Omega^\rho_{\text{ext}} } \frac{1}{ d_y^s(1+d_y)^{d+s(p-1)}}  \d y \d x\nonumber\\
	&\qquad+ \int\limits_{\Omega^\rho_{\text{ext}}}  \frac{\abs{u(y)}^{p}  }{ d_y^s (1+d_y)^{d+s(p-1)} }\int\limits_{\Omega^c} \frac{1}{d_x^s (1+d_x)^{d+s(p-1)}}\d x \d y\Bigg].\label{eq:trace_estimate_Omega^1_Omega^c_prev}
	\end{align}
	After covering $\Omega_1^{\text{ext}}$ by finitely many balls, a calculation similar to \eqref{eq:distance_integral_estimate} yields a constant $c_1>0$, independent of $s$, such that $\int_{\Omega_1^{\text{ext}}} d_x^{-s}\d x \le c_1(1-s)^{-1}$. By possibly enlarging the constant, we assume $c_1\ge \omega_{d-1}(1+\diam(\Omega))^{d+p^\star} $. With this observation and \eqref{eq:distance_integrated_at_infinity} we find:
	\begin{equation}\label{eq:trace_estimate_Omega^1_Omega^c_prev_2}
	\int\limits_{\Omega^c} \frac{1}{d_x^s (1+d_x)^{d+s(p-1)}}\d x\le \int\limits_{\Omega_1^{\text{ext}}} \frac{1}{d_x^s }\d x+ \int\limits_{ \Omega_{\text{ext}}^1}d_x^{-d-sp}\d x\le \frac{c_1 }{s(1-s)}
	\end{equation}
	Now, we combine this estimate with \autoref{prop:trace_Lp_estimate}:
	\begin{equation*}
	[\trns u]_{\cT^{s,p}(\Omega^\rho_{\text{ext}}\,|\, \Omega^c)}^p\le \frac{2^{p^\star+1}}{(1\wedge \rho)^{d+p^\star}}  \frac{c_1}{s} \norm{\trns u}_{L^{p}(\Omega^c;\,\mu_s)}^p\le \frac{2^{p^\star+1}}{(1\wedge \rho)^{d+p^\star}}  \frac{c_1 c_2}{s_\star} \norm{u}_{\VspOm}^p
	\end{equation*}
	Here $c_2=c_2(\Omega, p^\star, s_\star)>0$ is the constant from \autoref{prop:trace_Lp_estimate}. 
	
	\medskip
	
	Lastly, we prove the inequality for the more delicate part of the seminorm, where higher order singularities close to the boundary may occur. For $x,y\in \Omega_{\rho}^{\text{ext}}$ we have $(\abs{x-y}+d_x+d_y)\le 4\rho + \diam(\Omega)+1=:c_3$ and, thus, $(\abs{x-y}+d_x+d_y)\wedge 1 \ge c_3^{-1}(\abs{x-y}+d_x+d_y)$. We apply \autoref{lem:distance_lower} twice with the monotone decreasing function $f(r)= (\abs{x-y}+r/2+d_y)^{-d-s(p-2)}$ and then again with the function $f(r)= (\abs{x-y}+ \abs{x-z}/2+r/2)^{-d-s(p-2)}$. Let $c_4=c_4(\Omega,d)>0$, $r>0$ be the constants from \autoref{lem:distance_lower}. This yields
	\begin{align*}
	&[\trns u]_{\cT^{s,p}(\Omega_\rho^{\text{ext}}\,|\, \Omega_\rho^\text{ext})}^p \\
	&\le (1-s) c_3^{d+p}c_4 \int\limits_{\Omega_\rho^\text{ext}}\int\limits_{\Omega_\rho^{\text{ext}}} \int\limits_{\Omega}  \frac{\abs{u(x)-u(y)}^{p}  }{\abs{x-z}^{d+s}  (\abs{x-y}+\abs{x-z}/2+d_y)^{d+s(p-2)} } \d z\d x \mu_s(\d y)\\
	&\le (1-s)^2 c_3^{d+p} c_4^2 \int\limits_{\Omega_\rho^\text{ext}}\int\limits_{\Omega_\rho^{\text{ext}}} \int\limits_{\Omega}\int\limits_{ \Omega}  \frac{\abs{u(x)-u(y)}^{p} \abs{x-z}^{-d-s}  \abs{y-w}^{-d-s} }{ (\abs{x-y}+(\abs{x-z}+\abs{y-w})/2)^{d+s(p-2)}} \d w\d z\d x \d y\\
	&\le 2^{p^\star} 4^{d+p^\star} c_3^{d+p^\star}c_4^2 \Big( (\text{III})+ 2\,(\text{IV})\Big).
	\end{align*}
	Here we added $\pm u(z)\pm u(w)$ and used triangle inequality. The terms are
	\begin{align*}
	(\text{III})&:=  (1-s)^2\, \int\limits_{\Omega}\int\limits_{ \Omega} \abs{u(z)-u(w)}^{p} a(z,w) \d w \d z,\\
	a(z,w)&:=\int\limits_{ \Omega_\rho^{\text{ext}}}\int\limits_{ \Omega_\rho^{\text{ext}}} \frac{ 1}{\abs{x-z}^{d+s}  \abs{y-w}^{d+s}(\abs{x-y}+2\abs{x-z}+2\abs{y-w})^{d+s(p-2)}}\d x\d y,\\
	(\text{IV})&:=  (1-s)^2\,\int\limits_{\Omega_\rho^\text{ext}} \int\limits_{\Omega} \abs{u(x)-u(w)}^{p} b(x,w) \d w \d x,\\
	b(x,w)&:=  \int\limits_{ \Omega_\rho^{\text{ext}}} \int\limits_{\Omega}\frac{ 1}{\abs{x-z}^{d+s}  \abs{y-w}^{d+s} (\abs{x-y}+2\abs{x-z}+2\abs{y-w})^{d+s(p-2)}}\d z\d y.
	\end{align*}
	\textit{Estimate of (III):} Our goal is to find an appropriate estimate of kernel $a$ to apply \autoref{th:trace_classical_approximation_inequality_1}. Notice that $\abs{x-y}+2\abs{x-z}+2\abs{y-w}\ge \abs{z-w}+\abs{x-z}+\abs{y-w}$ for any $x,y,w,z\in \R^d$. Thus, for any $z,w\in \Omega$
	\begin{equation*}
	a(z,w)\le \int\limits_{ \Omega_\rho^{\text{ext}}}\int\limits_{ \Omega_\rho^{\text{ext}}} \frac{ 1}{\abs{x-z}^{d+s}  \abs{y-w}^{d+s}(\abs{z-w}+\abs{x-z}+\abs{y-w})^{d+s(p-2)}}\d x\d y.
	\end{equation*}
	Now, we apply \autoref{lem:distance_upper} twice with $B=\Omega_\rho^{\text{ext}}$. We use the function $f(t_1,t_2)= (\abs{z-w}+t_1+t_2)^{-d-s(p-2)}$ which is decreasing in both $t_1$ and $t_2$. Thereby, 
	\begin{equation}\label{eq:estimate_for_a}
	a(z,w)\le \frac{\omega_{d-1}^2}{s^2}\, \frac{1}{d_z^s\,d_w^s(\abs{z-w}+d_z+d_w)^{d+s(p-2)}}.
	\end{equation}
	This yields the desired estimate for $(\text{III})$ via \autoref{th:trace_classical_approximation_inequality_1}:
	\begin{align*}
	(\text{III})&\le \frac{\omega_{d-1}^2}{s^2}\,(1-s)^2\, \int\limits_{\Omega}\int\limits_{ \Omega} \frac{\abs{u(z)-u(w)}^{p} }{d_z^s\,d_w^s(\abs{z-w}+d_z+d_w)^{d+s(p-2)}} \d w \d z\\
	&\le c_5 \,\frac{\omega_{d-1}^2}{s^2} \, \norm{u}_{W^{s,p}(\Omega)}^p\le c_5 \,\frac{\omega_{d-1}^2}{s_\star^2} \, \norm{u}_{\VspOm}^p
	\end{align*}
	Here $c_5=c_5(d, \Omega, p_\star, p^\star, s_\star)>0$ is the constant from \autoref{th:trace_classical_approximation_inequality_1}.
	
	\textit{Estimate of (IV):} Our approach to estimate $(\text{IV})$ is similar to the proof of \cite[Theorem 5]{DyKa19}. Analogous to the estimate of $a$, see \eqref{eq:estimate_for_a}, we find 
	\begin{equation*}
	b(x,w)\le \frac{\omega_{d-1}^2}{s^2}\, \frac{1}{d_x^s\,d_w^s(\abs{x-w}+d_x+d_w)^{d+s(p-2)}}
	\end{equation*}
	for any $x\in \Omega_\rho^{\text{ext}}$ and $w\in \Omega$. We define
	\begin{align*}
		(\text{V}):= (1-s)^2\,\int\limits_{\Omega_\rho^\text{ext}} \int\limits_{\Omega}  \frac{\abs{u(x)-u(w)}^{p}}{d_x^s\,d_w^s(\abs{x-w}+d_x+d_w)^{d+s(p-2)}} \d w \d x.
	\end{align*}
	The previous estimate of $b$ yields $(\text{IV})\le \omega_{d-1}^2s_\star^{-2} (\text{V})$.
	
	\textbf{Claim: }We will show that there exists a constant $c_6=c_6(d,\Omega,\rho,p^\star)>0$ such that
	\begin{equation*}
	(\text{V})\le c_6\, \Bigg( \frac{1}{s^2\,(p_\star-1)} [u]_{\vs{s}{p}(\Omega\,|\, \Omega_\rho^{\text{ext}})}^p+ (1-s)^2\, \int\limits_{\Omega}\int\limits_{\Omega} \frac{\abs{u(z)-u(w)}^{p} }{d_z^s\,d_w^s(\abs{z-w}+d_z+d_w)^{d+s(p-2)}} \d w \d z\Bigg).
	\end{equation*}
	Let $\cW(\R^d\setminus \overline{\Omega})$ be the Whitney decomposition from \autoref{sec:whitney} and recall that $\rho= \inr{\Omega}$. We define the set of Whitney cubes $Q\in \cW(\R^d \setminus \overline{\Omega})$ with $\diam(Q)\le \rho$ by $\cW_\rho(\R^d\setminus \overline{\Omega})$. As in \autoref{sec:whitney} and \cite{DyKa19} we denote by $\widetilde{Q}\subset \Omega$ the reflected Whitney cube for any cube $Q\in \cW_\rho(\R^d\setminus \overline{\Omega})$. The collection of reflected Whitney cubes satisfies a bounded overlap property, see \eqref{eq:bounded_overlap}. Let $N\in \N$ be the constant from the bounded overlap property. Furthermore, the distance to the boundary as well as the diameter of the reflected cubes are comparable to the original cubes with the constant $M>0$ from \eqref{eq:reflected_cubes_comparable}. By the covering properties of the Whitney cubes \eqref{eq:prop_cubes_1} and the reflecting cubes \eqref{eq:prop1_reflec_cubes} we find
	\begin{equation}\label{eq:estimate_V}
	(\text{V})\le \sum\limits_{Q_1, Q_2\in \cW_\rho(\R^d\setminus\overline{\Omega})} (1-s)^2\,\int\limits_{Q_1} \int\limits_{\widetilde{Q}_2}  \frac{\abs{u(x)-u(w)}^{p}}{d_x^s\,d_w^s(\abs{x-w}+d_x+d_w)^{d+s(p-2)}} \d w \d x. 
	\end{equation}
	For the moment we fix two cubes $Q_1,Q_2\in \cW(\R^d\setminus \overline{\Omega})$ satisfying $ \diam{Q_1}, \diam{Q_2}\le \rho $. For each $x\in Q_1$, $w\in \widetilde{Q}_2$ we define
	\begin{equation*}
	z(x,w):= q_{\widetilde{Q}_1} + \Big( \frac{x-q_{Q_1}}{2l(Q_1)} + \frac{w-q_{\widetilde{Q}_2}}{2l(\widetilde{Q}_2)} \Big)\,l(\widetilde{Q}_1)\in \widetilde{Q}_1.
	\end{equation*}
	The map $z$ connects points in $Q_1, \widetilde{Q}_2$ with points in $\widetilde{Q}_1$ in a continuous way. We will use it for a change of variables in either $x$ or $w$. Therefore, 
	\begin{align}
	&\int\limits_{Q_1} \int\limits_{\widetilde{Q}_2}  \frac{\abs{u(x)-u(w)}^{p}}{d_x^s\,d_w^s(\abs{x-w}+d_x+d_w)^{d+s(p-2)}} \d w \d x\le  2^p \int\limits_{Q_1} \int\limits_{\widetilde{Q}_2}  \frac{\abs{u(z(x,w))-u(w)}^{p}}{d_x^s\,d_w^s(\abs{x-w}+d_x+d_w)^{d+s(p-2)}} \d w \d x\nonumber\\
	&\quad+ 2^p \int\limits_{Q_1} \int\limits_{\widetilde{Q}_2}  \frac{\abs{u(x)-u(z(x,w))}^{p}}{d_x^s\,d_w^s(\abs{x-w}+d_x+d_w)^{d+s(p-2)}} \d w \d x
	=: 2^p\Big((\text{VI})+ (\text{VII})\Big).\label{eq:estimate_one_term_V}
	\end{align}
	We make a few observations before estimating $(\text{VI})$. For $x\in Q_1$ and $w\in \widetilde{Q}_2$ we have $d_x\ge M^{-1}d_{z(x,w)}$ as well as
	\begin{align*}
	&\abs{x-w}+ d_x+d_w\ge M^{-1}\big( \abs{x-w}+ \distw{Q_1}{\widetilde{Q}_1} \big) + (1-M^{-1}) \abs{x-w} + d_w\\
	&\qquad\ge M^{-1}\big( \distw{Q_1}{\widetilde{Q}_2}+ \distw{Q_1}{\widetilde{Q}_1} \big) + (1-M^{-1}) \big( \distw{Q_1}{\partial \Omega}+ \distw{\widetilde{Q}_2}{\partial \Omega} \big) + d_w\\
	&\qquad\ge  M^{-1} \distw{\widetilde{Q}_1}{\widetilde{Q}_2}+ (1-M^{-1}) M^{-1} \distw{\widetilde{Q}_1}{\partial \Omega}+ (2-M^{-1})\distw{\widetilde{Q}_2}{\partial \Omega} \\
	&\qquad\ge  \frac{(1-M^{-1})M^{-1}\Big( \distw{\widetilde{Q}_1}{\widetilde{Q}_2}+  \sum_{i=1}^2 \distw{\widetilde{Q}_i}{\partial \Omega} \Big)(\abs{z(x,w)-w}+d_{z(x,w)}+ d_w)}{\big(\diam{\widetilde{Q}_1}+\distw{\widetilde{Q}_1}{\widetilde{Q}_2}+ \diam{\widetilde{Q}_2}\big)+ \sum_{i=1}^2 \big( \distw{\widetilde{Q}_i}{\partial \Omega}+\diam{\widetilde{Q}_i} \big) }\\
	&\qquad\ge (1-M^{-1})M^{-1}\, \frac{2}{3} \big( \abs{z(x,w)-w}+d_{z(x,w)}+ d_w \big).
	\end{align*}
	We define $c_7:= 3/2 \, M^2(M-1)^{-1}>1$ and use the previous calculation to estimate $(\text{VI})$ by 
	\begin{align}\label{eq:first_term_in_whitney}
	(\text{VI})&\le M^s\,c_7^{d+s(p-2)} \int\limits_{\widetilde{Q}_2} \int\limits_{Q_1}  \frac{\abs{u(z(x,w))-u(w)}^{p}}{d_{z(x,w)}^s\,d_w^s(\abs{z(x,w)-w}+d_{z(x,w)}+d_w)^{d+s(p-2)}}\d x \d w \nonumber\\
	&\le M^s\,c_7^{d+s(p-2)} \Big(\frac{2l(Q_1)}{l(\widetilde{Q}_1)}\Big)^d \int\limits_{\widetilde{Q}_2} \int\limits_{\widetilde{Q}_1}  \frac{\abs{u(z)-u(w)}^{p}}{d_{z}^s\,d_w^s(\abs{z-w}+d_{z}+d_w)^{d+s(p-2)}}  \d z\d w\nonumber\\
	&\le M\,c_7^{d+(0\vee (p^\star-2))} (2M)^d\int\limits_{\widetilde{Q}_1} \int\limits_{\widetilde{Q}_2}  \frac{\abs{u(z)-u(w)}^{p}}{d_{z}^s\,d_w^s(\abs{z-w}+d_{z}+d_w)^{d+s(p-2)}} \d w \d z.
	\end{align}
	Here we used the change of variables $z=z(x,w)$. We set $c_8:= 2^{d+p^\star}\,c_7^{d+(0\vee (p^\star-2))} M^{d+1}$. Now we sum $(\text{VI})$ over all Whitney cubes in $\cW_\rho(\R^d\setminus \overline{\Omega})$. By the bounded overlap property of the Whitney decomposition, see \eqref{eq:bounded_overlap}, we have
	\begin{align}
	& \sum \limits_{Q_1, Q_2\in \cW_\rho(\R^d\setminus\overline{\Omega})} (1-s)^2\,2^p (\text{VI}) \nonumber \\
	& \quad \le c_8 \,(1-s)^2 \hspace{-0.5cm}\sum\limits_{Q_1, Q_2\in \cW_\rho(\R^d\setminus\overline{\Omega})}\,  \int\limits_{\widetilde{Q}_1} \int\limits_{\widetilde{Q}_2}  \frac{\abs{u(z)-u(w)}^{p}}{d_{z}^s\,d_w^s(\abs{z-w}+d_{z}+d_w)^{d+s(p-2)}} \d w \d z\nonumber\\
	& \quad \le N^2\,c_8 \,(1-s)^2  \int\limits_{\Omega}\int\limits_{\Omega} \frac{\abs{u(z)-u(w)}^{p}}{d_{z}^s\,d_w^s(\abs{z-w}+d_{z}+d_w)^{d+s(p-2)}} \d w \d z.\label{eq:estimate_for_VI_complete}
	\end{align}
	Now we estimate $(\text{VII})$.	We make a few observations upon the choice of the Whitney decomposition and reflected cubes in \autoref{sec:whitney}. For $x\in Q_1$ and $w\in \widetilde{Q}_2$ we have
	\begin{align*}
	d_x&\ge (1+M^{-1})^{-1} (d_x+d_{z(x,w)})\ge  (1+M^{-1})^{-1} \abs{x-z(x,w)},\\
	d_w&\ge \distw{\widetilde{Q}_2}{\partial \Omega},\\
	d_x+d_w+\abs{x-w}&\ge \distw{Q_1}{\partial \Omega}+ \distw{\widetilde{Q}_2}{\partial \Omega}+ \distw{Q_1}{\widetilde{Q}_2},\\
	\abs{x-z(x,w)}&\le \distw{Q_1}{\widetilde{Q}_1}\le M \, \distw{Q_1}{\partial \Omega}.
	\end{align*}
	We set 
	\begin{equation*}
	J(Q_1, Q_2):= \frac{ \distw{Q_1}{\partial \Omega}^{d+s(p-1)}}{\distw{\widetilde{Q}_2}{\partial \Omega}^s  \big( \distw{Q_1}{\partial \Omega}+ \distw{\widetilde{Q}_2}{\partial \Omega}+ \distw{Q_1}{\widetilde{Q}_2} \big)^{d+s(p-2)}}.
	\end{equation*}
	Therefore, 
	\begin{align*}
	(\text{VII})&\le(1+M^{-1})^s\,M^{d+s(p-1)} 	J(Q_1, \widetilde{Q}_2)\,\int\limits_{Q_1} \int\limits_{\widetilde{Q}_2}  \frac{\abs{u(x)-u(z(x,w))}^{p}}{\abs{x-z(x,w)}^{d+sp}} \d w \d x  \\
	&\le (1+M^{-1})\,M^{d+p-1} 	J(Q_1, \widetilde{Q}_2)\,\Big(\frac{2\,l(\widetilde{Q}_2)}{l(\widetilde{Q}_1)}\Big)^d\,\int\limits_{Q_1} \int\limits_{\widetilde{Q}_1}  \frac{\abs{u(x)-u(z)}^{p}}{\abs{x-z}^{d+sp}} \d z \d x.
	\end{align*}
	Here we used the change of variables $z:=z(x,w)$. Set $c_9:=  (1+M^{-1})\,M^{d+p^\star-1} 2^{d+p^\star}$. By \eqref{eq:bounded_overlap}, 
	\begin{equation}\label{eq:first_step_estimate_(VII)}
	(1-s)^2 \mspace{-35mu}\sum\limits_{Q_1, Q_2\in \cW_\rho(\R^d\setminus\overline{\Omega})}\mspace{-20mu}	2^p (\text{VII})\le c_9\, (1-s)^2 \mspace{-35mu}\sum\limits_{Q_1, Q_2\in \cW_\rho(\R^d\setminus\overline{\Omega})} \mspace{-20mu}	J(Q_1, \widetilde{Q}_2)\,\Big(\frac{l(\widetilde{Q}_2)}{l(\widetilde{Q}_1)}\Big)^d\,\int\limits_{Q_1} \int\limits_{\widetilde{Q}_1}  \frac{\abs{u(x)-u(z)}^{p}}{\abs{x-z}^{d+sp}} \d z \d x.
	\end{equation}
	Therefore, it is sufficient to prove that 
	\begin{equation*}
	I(Q_1):=\sum\limits_{Q_2\in \cW_\rho(\R^d\setminus\overline{\Omega})}J(Q_1, \widetilde{Q}_2)\,\Big(\frac{l(\widetilde{Q}_2)}{l(\widetilde{Q}_1)}\Big)^d
	\end{equation*}
	is bounded independent of $Q_1$. We fix $Q_1\in \cW_\rho(\R^d\setminus \Omega)$ and set $a:= \distw{Q_1}{\partial \Omega}$. Let $\hat{q}_{Q_1}\in \partial \Omega$ be a minimizer of the distance of $q_{Q_1}$ to $\partial \Omega$, \ie $\abs{\hat{q}_{Q_1}-q_{Q_1}}= \distw{q_{Q_1}}{\partial \Omega}$. By the properties of the Whitney cubes, we have for any $w\in \widetilde{Q}_2$
	\begin{align}
	\distw{\widetilde{Q}_2}{\partial \Omega}&\ge \frac{1}{2}\big( \distw{\widetilde{Q}_2}{\partial \Omega} + \diam{\widetilde{Q}_2}\big)\ge \frac{1}{2} d_w,\label{eq:estimateVII_help1}\\
	\abs{w-b}&\le \abs{w-q_{Q_1}}+ \abs{q_{Q_1}-\hat{q}_{Q_1}}\le \diam{Q_1}+ \distw{\widetilde{Q}_2}{Q_1} + \distw{q_{Q_1}}{\partial \Omega}\nonumber\\
	&\le \distw{\widetilde{Q}_2}{Q_1}+ 3 \distw{Q_1}{\partial \Omega}\le 4 \distw{\widetilde{Q}_2}{Q_1}.\label{eq:estimateVII_help2}
	\end{align}
	We estimate $I(Q_1)$ using the properties of the Whitney cubes, \eqref{eq:estimateVII_help1}, \eqref{eq:estimateVII_help2} and \eqref{eq:bounded_overlap}
	\begin{align}
	I(Q_1)&\le M^d 4^d \sum\limits_{Q_2\in \cW_\rho(\R^d\setminus\overline{\Omega})}\frac{ \diam(\widetilde{Q}_2)^d\, a^{s(p-1)}  }{\distw{\widetilde{Q}_2}{\partial \Omega}^s  \big( a+ \distw{\widetilde{Q}_2}{\partial \Omega}+ \distw{Q_1}{\widetilde{Q}_2} \big)^{d+s(p-2)}}\nonumber\\
	&= M^d 4^d \sum\limits_{Q_2\in \cW_\rho(\R^d\setminus\overline{\Omega})} \int\limits_{\widetilde{Q}_2}\frac{ a^{s(p-1)}  }{\distw{\widetilde{Q}_2}{\partial \Omega}^s  \big( a+ \distw{\widetilde{Q}_2}{\partial \Omega}+ \distw{Q_1}{\widetilde{Q}_2} \big)^{d+s(p-2)}} \d w\nonumber\\
	&\le 2^s\,4^{d+s(p-2)}\, M^d 4^d\,N \int\limits_{\Omega}\frac{ a^{s(p-1)}  }{d_w^s  \big( a+ d_w+ \abs{w-\hat{q}_{Q_1}} \big)^{d+s(p-2)}} \d w.\label{eq:estimateVII_help3}
	\end{align}
	We claim that the integral in the last line is bounded independent of $a$ and $\hat{q}_{Q_1}$.

%	We localize the boundary in a neighborhood of $\hat{q}_{Q_1}$. There exists a rotation and translation $T$ and a Lipschitz map $\phi:\R^{d-1}\to \R$ such that $\norm{\phi}_{C^{0,1}}$ is bounded by a uniform constant $L\ge 1$. Furthermore, $T(\Omega\cap B_{r_0}(\hat{q}_{Q_1}))= \{ (w',w_d)\in B_{r_0}(0)\,|\, \phi(w')>w_d \}$. Without loss of generality, we assume $\hat{q}_{Q_1}=0$ and $T$ to be the identity mapping. By xxx
%	\begin{align*}
%		\int\limits_{\Omega\cap B_{r_0}(\hat{q}_{Q_1})}\frac{ a^{s(p-1)}  }{d_w^s  \big( a+ d_w+ \abs{w-\hat{q}_{Q_1}} \big)^{d+s(p-2)}} \d w\le  \int\limits_{B_{r_0}\cap \{ w_d<\phi(w') \}} \frac{ 2(1+L) a^{s(p-1)}  }{(\phi(w')-w_d)^s  \big( a+ \abs{w} \big)^{d+s(p-2)}} \d (w',w_d)
%	\end{align*}
	We localize the boundary in a neighborhood of $\hat{q}_{Q_1}$. Let $r_0>0$ be the localization radius, $\phi:B_{r_0}(\hat{q}_{Q_1})\to B_1(0)$ be a bi-Lipschitz flattening of the boundary. Since $\Omega$ has a uniform Lipschitz boundary. A change of variables as well as an estimate similar to \eqref{eq:distance_estimate} yields a constant $c_{10}=c_{10}(d,\Omega)\ge 1$ such that
	\begin{align*}
		\int\limits_{\Omega\cap B_{r_0}(\hat{q}_{Q_1})}\frac{ a^{s(p-1)}  }{d_w^s  \big( a+ d_w+ \abs{w-\hat{q}_{Q_1}} \big)^{d+s(p-2)}} \d w\le c_{10} \int\limits_{B_1(0)_+} \frac{a^{s(p-1)}}{w_d^s\, (a+ \abs{w})^{d+s(p-2)}} \d w.
	\end{align*}
	To calculate this integral, we apply the coarea formula, see e.g. \cite{Fed69}, with $(r,t)=(w_d, \abs{w})$ in case $d\ge 2$.
	\begin{align}
		\int\limits_{B_1(0)_+} \frac{a^{s(p-1)}}{w_d^s\, (a+ \abs{w})^{d+s(p-2)}} \d w &\le \omega_{d-2} \int\limits_{0}^{1}\int\limits_{0}^{t}\frac{ a^{s(p-1)} t^{d-2} }{r^s  \big( a+ t \big)^{d+s(p-2)}} \d r \d t\nonumber\\
		\le \frac{ \omega_{d-2} }{1-s}\int\limits_{0}^{1}\frac{ a^{s(p-1)} t^{d-1-s} }{ \big( a+ t \big)^{d+s(p-2)}}  \d t&\le  \frac{ \omega_{d-2} }{1-s}\int\limits_{0}^{1}\frac{ a^{s(p-1)}}{ \big( a+ t \big)^{1+s(p-1)}}  \d t\le \frac{ \omega_{d-2}}{(1-s)s(p-1)}\label{eq:estimateVII_help4}
	\end{align}
	Here we used:
	\begin{equation*}
		\cH^{(d-2)}\big( \{ w\in \Omega\,|\, w_d=r, \abs{w}=t \} \big)\le \omega_{d-2} \1_{r\le t} t^{d-2}
	\end{equation*}	
	A similar calculation shows the same estimate in the case $d=1$. Furthermore, the remainder of the integral on the RHS of \eqref{eq:estimateVII_help3}, \ie \begin{equation*}
		\int\limits_{\Omega \cap B_{r_0}(\hat{q}_{Q_1})^c}\frac{ a^{s(p-1)}  }{d_w^s  \big( a+ d_w+ \abs{w-\hat{q}_{Q_1}} \big)^{d+s(p-2)}} \d w, 
	\end{equation*}
	is easily bounded independent of $a$ since $a+d_w+\abs{w-\hat{q}_{Q_1}}\ge r_0$ for $w\in \Omega \cap B_{r_0}(\hat{q}_{Q_1})^c$ and $a\le 4\rho$.

	 Therefore $(1-s)I(Q_1)$ is bounded independent of $Q_1$ by a constant $c_{11}=c_{11}(d, \Omega, p_\star, p^\star, s_\star)>0$. We combine \eqref{eq:first_step_estimate_(VII)},\eqref{eq:estimateVII_help3} and \eqref{eq:estimateVII_help4} as well as \eqref{eq:bounded_overlap} to obtain
	\begin{equation}\label{eq:prefinal_estimate_VII}
	(1-s)^2 \mspace{-35mu}\sum\limits_{Q_1, Q_2\in \cW_\rho(\R^d\setminus\overline{\Omega})}\mspace{-20mu}	2^p (\text{VII})\le c_9\,c_{11}\,N\, (1-s)\int\limits_{\Omega_\rho^{\text{ext}}} \int\limits_{\Omega}  \frac{\abs{u(x)-u(z)}^{p}}{\abs{x-z}^{d+sp}} \d z \d x= c_9\,c_{11}\,N  [u]_{\vs{s}{p}(\Omega\,|\, \Omega_\rho^{\text{ext}})}^p.
	\end{equation}
	Finally, we combine \eqref{eq:estimate_V}, \eqref{eq:estimate_one_term_V}, \eqref{eq:estimate_for_VI_complete} and \eqref{eq:prefinal_estimate_VII} and the claim follows. The constant is $c_6:= N \max\{ Nc_8, c_9\,c_{11} \}$.
	
	We finish the estimate of $(\text{IV})$ using the previous claim and \autoref{th:trace_classical_approximation_inequality_1}:
	\begin{align*}
	(\text{V})\le c_6\Bigg( \frac{1}{s^2(p_\star-1)}[u]_{\vs{s}{p}(\Omega\,|\, \Omega_{\rho}^{\text{ext}})}^p + c_5 \norm{u}_{W^{s,p}(\Omega)}^p  \Bigg).
	\end{align*}
	Combining the estimates of $(\text{III})$ and $(\text{IV})$ yields \eqref{eq:trace_continuity}.
\end{proof}

\begin{remark}\label{example:trace_p-1}
	As mentioned in the introduction, the trace embedding $\vs{s}{1}(\Omega\,|\, \R^d)\to \cT^{s,1}(\Omega^c)$ cannot be continuous. This may be seen as follows. Consider the sequence of functions 
	\begin{equation*}
		u_n(x):= \begin{cases}
		0&, x\in \Omega\\
		n^{1-s}&, x\in \Omega_{1/n}^{\text{ext}}\\
		0&, x\in \Omega_{\text{ext}}^{1/n}
		\end{cases}
	\end{equation*}
	for $n\in \N$. By \autoref{lem:distance_upper} and \autoref{lem:distance_lower} one easily sees that $\norm{u_n}_{\vs{s}{1}(\Omega\,|\, \R^d)}\asymp \norm{u_n}_{L^1(\Omega^c;\,\mu_s)}\asymp 1$ but a simple calculation yields $[u_n]_{\cT^{s,1}(\Omega^c)}\asymp \ln(n)\to \infty$ as $n\to \infty$. A similar sequence of functions proves \autoref{th:trace_classical_approximation_inequality_1} to be false for $p=1$.
\end{remark}

\section{Extension results}\label{sec:extension}
	The aim of this section is to prove the extension part results of \autoref{th:trace-extension} and \autoref{th:trace-p-1}. This proof is carried out in \autoref{prop:extension_continuity_Lp_part} and \autoref{prop:extension_continuity_seminorm}. We are able to treat the cases $1 < p < \infty$ and $p=1$ together.

The method used in this section is essentially inspired by \cite[Chapter V]{JoWa84}. We decompose the domain $\Omega$ into Whitney cubes and consider neighborhoods of each cube intersected with $\Omega^c$. The extension is constructed by copying weighted mean values of the exterior data $g$ from this intersection into the respective cube, see \eqref{eq:extension_def}. The weights are taken with respect to a measure that behaves like $\mu_s$ close to the boundary $\partial \Omega$.

\medskip

 Throughout this section we fix an open nonempty proper subset $\Omega\subset \R^d$. We will introduce additional assumptions on $\Omega$ when needed. Further, we fix a dyadic Whitney-decomposition $\cW(\Omega)$ of $\Omega$ consisting of cubes with parallel sides to the axes of $\R^d$ such that
\begin{enumerate}
	\item[(a)]{ $\Omega=\bigcup_{Q\in \cW(\Omega)  } Q$. }
	\item[(b)]{ The interior of the cubes are mutually disjoint.}
	\item[(c)]{ For all cubes $Q\in \cW(\Omega)$ the diameter is comparable to the distance to the boundary of $\Omega$, \ie 
	\begin{equation}\label{eq:distance_diam_comparable}
		\diam{Q}\le d(Q, \partial \Omega)\le 4 \diam{Q}.
\end{equation} }
\end{enumerate}
We set $s_Q\in 2^{\Z}$ to be the side length of a cube $Q$ and let $q_Q\in Q$ be the center of the cube $Q$. We denote the diameter of the cube $Q$ by $l_Q=\diam(Q)=\sqrt{d}s_Q$. This decomposition satisfies the following. Suppose $Q_1,Q_2\in \cW(\Omega)$ touch each other. Then 
\begin{equation}\label{eq:touching_cubes}
	\frac{1}{4}\diam{Q_1}\le \diam{Q_2}\le 4 \diam{Q_1}.
\end{equation}
Additionally, we denote by $Q^\star$ a cube having the same center as $Q\in \cW(\Omega)$ but the side length $(1+1/8)s_Q$. We denote the collection of these scaled cubes $\cW^\star(\Omega):= \{Q^\star\,|\, Q\in \cW(\Omega) \}$. These scaled cubes satisfy a finite overlap property, \ie $\sum_{Q^\star \in \cW^\star(\Omega)}\1_{Q^\star}\le N$ where $N\in \N$ is a fixed number for the remainder of this section. Additionally, two cubes $Q_1^\star, Q_2^\star \in \cW^\star(\Omega)$ have nonempty intersection if and only if $Q_1$ and $Q_2$ touch. We define $J_i\subset \cW(\Omega)$ to be the set of all cubes with side lengths $2^{-i}$ and set
\begin{equation}\label{eq:def_D_j}
	D_i:= \bigcup_{Q\in J_i} Q, \qquad D_{\ge i}:= \bigcup_{j\ge i} D_j.
\end{equation}

Analogous to \cite[Section 1.2]{JoWa84}, we introduce a specific partition of unity on $\Omega$ which we will use to construct an extension operator $\cT^{s,p}(\Omega^c)\to \VspOm$. We emphasize that the construction of the extension is independent of $p$. Let $\psi\in C_c^\infty(\R^d)$ be a bump function such that $\psi = 1 $ on $[-1/2,1/2]^{d}$ and $\psi=0$ on $\big([-(1+1/8)/2,(1+1/8)/2]^d\big)^c$, $0\le \psi \le 1$. Then we define for any $Q\in \cW(\Omega)$ a translated and rescaled version of $\psi$ by 
\begin{equation*}
	\psi_Q(x):= \psi\big( \frac{x-q_{Q}}{s_Q} \big)
\end{equation*} 
and our partition functions 
\begin{equation}\label{eq:def_phi_Q}
	\phi_Q(x):= \frac{\psi_Q(x)}{\sum\limits_{\widetilde{Q}\in \cW(\Omega)} \psi_{\widetilde{Q}(x)} }.
\end{equation}
Then obviously $\sum_Q \phi_Q = \1_{\Omega}$ holds. Furthermore, we set 
\begin{align}\label{eq:rho_kappa_whitney_cubes}
	\rho:= \inr{\Omega}/2\wedge 1/2>0,\quad\kappa:= \lfloor\log_2(\rho/\sqrt{d})\rfloor,\nonumber \\
	 \cW_{\le \kappa}(\Omega):= \{ Q\in \cW(\Omega)\,|\, s_Q\le 2^{\kappa} \}.
\end{align}
 Since $l_Q=\sqrt{d}s_Q$ for any $Q\in \cW$ and the cubes have dyadic side lengths, $l_Q\le \rho$ for any $Q\in \cW_{\le \kappa}$. For any $x\in Q \in \cW$ we know $d_x\le l_Q + \distw{Q}{\partial \Omega}\le 5l_Q$ and $d_x\ge \distw{Q}{\partial \Omega}\ge 1/4 l_Q$. Therefore, 
\begin{equation}\label{eq:small_cubes_united}
\{ x\in \Omega \,|\, d_x<\rho/4\} \subset \bigcup_{Q\in \cW_{\le \kappa}(\Omega)  }Q\subset \{ x\in \Omega \,|\, d_x< 5 \rho \} \,.
\end{equation}
For any cube $Q\in \cW_{\le \kappa}$ such that all neighboring cubes are in $\cW_{\le \kappa}$, we have $\sum_{ Q'\in \cW_{\le \kappa}} \phi_{Q'}(x)=1$,  $x\in Q$. By \eqref{eq:touching_cubes}, all cubes $Q$ with a side length that is at most $2^{\kappa-2}$ only have neighboring cubes in $\cW_{\le \kappa}$. Therefore, 
\begin{equation}\label{eq:extension_unity_close_to_boundary}
	\sum_{ Q\in \cW_{\le \kappa}} \phi_{Q}(x)=1\qquad \big( x\in D_{\ge -\kappa+2}\big).
\end{equation}

We define for $s\in (0,1)$ the measure on $\mathcal{B}(\R^d)$
\begin{equation}\label{eq:def_mu_tilde}
	\widetilde{\mu}_s(\d z)= \1_{\Omega^c}(z)\frac{1-s}{d_z^s}\d z.
\end{equation}
This measure behaves like $\mu_s$, see \eqref{eq:def_mu_s}, near the boundary $\partial \Omega$. We will use $\widetilde{\mu}_s$ to construct the extension of a function $g:\Omega^c\to \R$, see \eqref{eq:extension_def}. In particular, the value of the extension $\extns(g)$ in a cube $Q\in \cW_{\le \kappa}$ will depend on a $\widetilde{\mu}_s$-mean of $g$ in a neighborhood of $Q$. Since $\widetilde{\mu}_s$ converges weakly to the surface measure on $\partial \Omega$, we recover a classical Whitney extension of functions in $W^{1-1/p,p}(\partial \Omega)$ in the limit $s\to 1^-$. We set for $Q\in \cW(\Omega)$
\begin{equation}\label{eq:a_{Q,s}}
	a_{Q,s}:= \Big( \widetilde{\mu}_s(B_{6l_Q}(q_Q)) \Big)^{-1}.
\end{equation}
Since $\distw{q_Q}{\partial \Omega}\le 5l_Q$, the intersection $B_{6l_Q}(q_Q)\cap \Omega^c $ has nonempty interior. The following lemma shows the order of $a_{Q,s}$ in terms of $l_Q$ and $s$ for Lipschitz domains. The estimate \eqref{eq:estimate_a_{Q,s}} is essential in \autoref{prop:extension_continuity_Lp_part} and \autoref{prop:extension_continuity_seminorm}.  
\begin{lemma}\label{lem:estimate_a_{Q,s}}
	Let $\emptyset \ne \Omega\subset \R^d$ be a Lipschitz domain. 
	\begin{align}\label{eq:estimate_a_{Q,s}}
	\begin{aligned}
	&\text{There exists a constant $C=C(d,\Omega)>1$ such that for any $s\in (0,1)$ and $Q\in \cW_{\le \kappa+6}(\Omega)$ } \\
		&\hspace{150pt} C^{-1}\,l_Q^{s-d}\le a_{Q,s}\le C\, l_Q^{s-d}.
		\end{aligned}
	\end{align}
\end{lemma}
\begin{proof}
	Let $z_Q\in \partial \Omega$ be a minimizer of the distance of $q_Q$ to the boundary $\partial \Omega$. Since the distance $\distw{q_Q}{\partial \Omega}$ is bounded from above by $5 l_Q$, we have $B_{l_Q}(z_Q)\subset B_{6l_Q}(q_Q)\subset B_{11 l_Q}(z_Q)$. Since $\Omega$ has a uniform Lipschitz boundary, we find a radius $r=r(\Omega)>0$ and a constant $L=L(\Omega)>0$, both independent of $z_Q$, a rotation and translation $T_{z_Q}:\R^d\to \R^d$ as well as a Lipschitz continuous function $\phi:\R^{d-1}\to \R$ such that $T_{z_Q}(\Omega \cap B_r(z_q))= \{ (x',x_d)\in B_r(z_Q)\,|\, x_d> \phi(x') \}$ and $[\phi]_{C^{0,1}}\le L$. Without loss of generality we assume $T_{z_Q}$ to be the identity map, in particular $z_Q=0$. By arguments similar to \eqref{eq:distance_estimate}, we find $2(1+L)d_x\ge \abs{x_d-\phi(x')}$ for any $(x',x_d)\in B_r(0)$ such that $x_d> \phi(x')$. First we assume that $11l_Q\le r$. Then proceeding like in \eqref{eq:distance_integral_estimate} yields
	\begin{align*}
		a_{Q,s}^{-1}\le (2(1+L))^s \int\limits_{B_{11l_Q}\cap \{x_d\ge \phi(x')\}}\frac{1-s}{(x_d-\phi(x'))^s}\d (x',x_d)\le 2(11+L)(11)^d l_Q^{d-s}.
	\end{align*}
	If $11l_Q>r$, then we simply cover $B_{11l_Q}(z_Q)\cap \partial\Omega$ by  finitely many balls $B_1, \dots, B_N$. Since $l_Q$ is bounded from above by $2^6 \rho$, the number of balls  of radius r, which are needed, can be picked uniformly. Set $A:= \Omega^c \cap B_{11lQ}(z_Q) \cap \bigcap_{j} B_j^c$ and $r_1:= \distw{\partial \Omega}{A}$. We pick the balls $B_1, \dots, B_N$ such that $r_1>r/2$. Then, by a similar calculation as above,
	\begin{align*}
		a_{Q,s}^{-1}&\le \sum_{j=1}^N \widetilde{\mu}_s(B_j) + \widetilde{\mu}_s(A)\le N  2(11+L)(11)^d r^{d-s} + 2^7 11^d \omega_{d-1} (r \wedge 1)^{-1} (\rho \vee 1) l_Q^{d-s} \\
		&\le \big(N  2(11+L)(11)^{2d}  + 2^7 11^d \omega_{d-1} (r \wedge 1)^{-1} (\rho \vee 1) \big) l_Q^{d-s}.
	\end{align*}
	For the upper bound in \eqref{eq:estimate_a_{Q,s}}, we simply notice 
	\begin{align*}
		a_{Q,s}^{-1}\ge \int\limits_{B_{l_Q\wedge r}\cap \{x_d\ge \phi(x')\}}\frac{1-s}{(x_d-\phi(x'))^s}\d (x',x_d)
	\end{align*}
	and proceed in a similar fashion.
	\end{proof}
For $g\in L_{\text{loc}}^p(\R^d)$ we define the extension $\extns(g)$ as follows
	\begin{equation}\label{eq:extension_def}
		\extns(g)(x):= \begin{cases}
			\sum\limits_{Q\in \cW_{\le \kappa}(\Omega)} \phi_Q(x) a_{Q,s} \int\limits_{\Omega^c\cap B_{6l_Q}(q_Q)} g(z) \widetilde{\mu}_s(\d z) & \text{ for } x\in \Omega,\\
			g(x)& \text{ for } x\in \Omega^c.	
		\end{cases}
	\end{equation} 
For any $Q\in \cW_{\le \kappa}(\Omega)$ we have $\sup\{ \distw{x}{\partial \Omega}\,|\, x\in Q^\star \}\le  \distw{Q^\star}{\partial \Omega}+ \diam(Q^\star)\le 4\rho + 9/8\rho \le 6\, \rho $. Therefore, $\extns(g)(x)=0$ for $x\in \Omega$ such that $d_x>6\, \rho$. Additionally, the definition of $\extns(g)$ inside $\Omega$ depends only on the values of $g$ on $\Omega_{6\rho}^{\text{ext}}\subset \Omega_{3\,\inr{\Omega}}^{\text{ext}}$. We could use the measure $\mu_s$ introduced in \eqref{eq:def_mu_s} instead of $\widetilde{\mu}_s$ in the definition of the extension because $\extns(g)|_{\Omega}$ does not depend on the values of $g$ far away from the boundary. But the benefit of $\widetilde{\mu}_s$ is that the extension is independent of the parameter $p$.  

We begin by proving some properties of $\extns$ analogous to \cite[Lemma 1]{Jon94}. The proof follows the same lines as \cite[Chapter V Lemma D]{JoWa84}. We define for cubes $Q_1,Q_2\in \cW_{\le \kappa}(\Omega)$ and $g\in L_{\text{loc}}^p(\R^d)$
	\begin{equation}\label{eq:definition_J_p}
		J_p(q_{Q_1}, q_{Q_2}):= \Bigg( a_{Q_1,s}\,a_{Q_2,s}\int\limits_{ B_{30l_{Q_1}}(q_{Q_1})  }\int\limits_{ B_{30l_{Q_2}}(q_{Q_2}) } \abs{g(z_1)-g(z_2)}^p \widetilde{\mu}_s(\d z_2)\widetilde{\mu}_s(\d z_1) \Bigg)^{1/p}.
	\end{equation}
	\begin{lemma}\label{lem:basic_properties_extns}
		Let $s\in (0,1)$, $1\le p<\infty$ and assume that $\Omega$ satisfies \eqref{eq:estimate_a_{Q,s}}. Further, let $Q_1, Q_2\in \cW_{\le \kappa-2}(\Omega)$. There exists a constant $C=C(d,\Omega, \psi, \cW(\Omega))>0$ such that for any $g\in L^p(\Omega^c)$ and $x\in Q_1$, $y\in Q_2$ as well as $b\in \R$
		\begin{enumerate}
			\item[(a)]{ $\abs{\extns(g)(x)-\extns(g)(y)} \le C\, J_p(q_{Q_1}, q_{Q_2})$,}
			\item[(b)]{ $\abs{\nabla \extns(g)(x)} \le C\, l_{Q_1}^{-1} \, J_p(q_{Q_1}, q_{Q_2})$,}
			\item[(c)]{ $\abs{\extns(g)(x)-b} \le C\,\Bigg( a_{Q_1,s}\int\limits_{ B_{30l_{Q_1}}(q_{Q_1})  } \abs{g(z_1)-b}^p \widetilde{\mu}_s(\d z_1) \Bigg)^{1/p}$, }
			\item[(d)]{ $\abs{\nabla\extns(g)(w)} \le C\,l_Q^{-1}\Bigg( a_{Q,s}\int\limits_{ B_{30l_Q}(q_{Q})  } \abs{g(z)}^p \widetilde{\mu}_s(\d z) \Bigg)^{1/p}$ for any $w\in Q\in \cW(\Omega)$. }
		\end{enumerate}
	\end{lemma}
	\begin{proof}
		For \textbf{(a):} By \eqref{eq:extension_unity_close_to_boundary}, $\sum\limits_{Q\in \cW_{\le \kappa}} \phi_Q(x) = 1 = \sum\limits_{Q\in \cW_{\le \kappa}} \phi_Q(y) $. By the choice of $a_{Q,s}$, we find 
		\begin{align*}
			&\extns(g)(x)-\extns(g)(y)= \sum\limits_{Q \in \cW_{\le \kappa}(\Omega)} \phi_{Q }(x) a_{Q ,s} \int\limits_{ B_{6l_{Q}}(q_{Q })} (g(z_1)- \extns(g)(y))\widetilde{\mu}_s(\d z_1) \\
			&\qquad= 	\sum\limits_{Q, \widetilde{Q} \in \cW_{\le \kappa}(\Omega)} \phi_{Q }(x)\phi_{\widetilde{Q}}(y)  a_{Q ,s} a_{\widetilde{Q},s} \int\limits_{B_{6l_{Q}}(q_{Q })}  \int\limits_{B_{6l_{\widetilde{Q}}}(q_{\widetilde{Q}})} (g(z_1)-g(z_2))\widetilde{\mu}_s(\d z_2) \widetilde{\mu}_s(\d z_1).
		\end{align*}
		We apply Hölder inequality to find
		\begin{align*}
			&\abs{\extns(g)(x)-\extns(g)(y)}\le 	\sum\limits_{Q, \widetilde{Q} \in \cW_{\le \kappa}(\Omega)} \phi_{Q }(x)\phi_{\widetilde{Q}}(y)  a_{Q ,s} a_{\widetilde{Q},s} \Big( \widetilde{\mu}_s(B_{6l_{Q}}(q_{Q }))\,\widetilde{\mu}_s(B_{6l_{\widetilde{Q}}}(q_{\widetilde{Q}})) \Big)^{1-1/p}\\
			&\qquad\qquad\qquad \cdot \Bigg(\int\limits_{B_{6l_{Q}}(q_{Q })}  \int\limits_{B_{6l_{\widetilde{Q}}}(q_{\widetilde{Q}})} \abs{g(z_1)-g(z_2)}^p\widetilde{\mu}_s(\d z_2) \widetilde{\mu}_s(\d z_1)\Bigg)^{1/p} \\
			&\qquad= \sum\limits_{Q, \widetilde{Q} \in \cW_{\le \kappa}(\Omega)} \phi_{Q }(x)\phi_{\widetilde{Q}}(y) \Bigg(a_{Q,s}a_{\widetilde{Q},s}\int\limits_{B_{6l_{Q}}(q_{Q })}  \int\limits_{B_{6l_{\widetilde{Q}}}(q_{\widetilde{Q}})} \abs{g(z_1)-g(z_2)}^p\widetilde{\mu}_s(\d z_2) \widetilde{\mu}_s(\d z_1)\Bigg)^{1/p}.
		\end{align*}
		Let $Q\in \cW_{\le \kappa}$ be a cube such that $\phi_{Q}(x)\ne 0$, then $Q$ touches $Q_1$ by the definition of $\phi_Q$. By \eqref{eq:touching_cubes} we find $\abs{t-q_{Q_1}}\le \abs{t-q_Q}+ \abs{q_Q-q_{Q_1}}\le 6l_Q + (l_Q+l_{Q_1})\le (6\cdot 4 +4+1)l_{Q_1}\le 30 l_{Q_1}$ for any $t\in B_{6l_Q}(q_Q)$. Let $c_1=c_1(d,\Omega, p)>1$ be the constant from \eqref{eq:estimate_a_{Q,s}}, then $a_{Q,s}\le c_1 l_Q^{s-d}\le c_1 4^{d-s} l_{Q_1}^{s-d}\le c_1^2 4^{d} a_{Q_1,s}$. By the finite overlap property, there exist at most $N-1$ cubes touching $Q_1$. The same holds for $Q_1$ replaced by $Q_2$. Therefore, 

		\begin{align*}
			&\abs{\extns(g)(x) -\extns(g)(y)} \\ 
			& \quad \le (N-1)^2 (c_1^2 4^d )^{2/p} \Bigg(a_{Q_1,s}a_{Q_2,s}\iint\limits_{\substack{ B_{30 l_{Q_1}}(q_{Q_1 })\\B_{30 l_{Q_2}}(q_{Q_2 }) }}  \abs{g(z_1)-g(z_2)}^p\widetilde{\mu}_s(\d z_2) \widetilde{\mu}_s(\d z_1)\Bigg)^{1/p}.
		\end{align*}
		
		For \textbf{(b):}  By \eqref{eq:extension_unity_close_to_boundary}, $\sum_{ Q\in \cW_{\le \kappa}}\phi_{Q }=1$ on $D_{\ge 2-\kappa}$ and, thus, $\sum_{Q \in \cW_{\le \kappa}} \nabla \phi_Q =0$ on $D_{\ge 2-\kappa}$. We write
		\begin{align*}
			\nabla \extns(g)(x)&= \sum\limits_{Q\in \cW_{\le \kappa}(\Omega)} \nabla \phi_Q(x)\Bigg( a_{Q,s} \int\limits_{ B_{6l_Q}(q_Q)} g(z_1) \widetilde{\mu}_s(\d z_1) - \extns(g)(y)\Bigg)\\
			&=  \sum\limits_{Q,\widetilde{Q}\in \cW_{\le \kappa}(\Omega)} \nabla \phi_Q(x)\phi_{\widetilde{Q}}(y) a_{Q,s} a_{\widetilde{Q},s}\int\limits_{B_{6l_Q}(q_Q)}\int\limits_{B_{6l_{\widetilde{Q}}}(q_{\widetilde{Q}})}  (g(z_1) -g(z_2))\widetilde{\mu}_s(\d z_2)\widetilde{\mu}_s(\d z_1)  
		\end{align*}
		By the definition of the partition of unity $\phi_{(\cdot)}$, see \eqref{eq:def_phi_Q}, there exists a positive constant $c_2=c_2(\cW(\Omega), d, \psi)$ such that for any $Q\in \cW_{\le \kappa}$ touching $Q_1$ or $Q_1$ itself, we have $\abs{\nabla \phi_Q(x)}\le c_2\,s_{Q}^{-1}\le 4\sqrt{d} c_2\, l_{Q_1}^{-1}$. We calculate using the same arguments as in the proof of (a)
		\begin{align*}
			\abs{\nabla \extns(g)(x)}\le l_{Q_1}^{-1} 4\sqrt{d} \, c_2\,(N-1)^2 (c_1^2 4^d)^{2/p} J_p(q_{Q_1}, q_{Q_2}). 
		\end{align*}
		The proofs of \textbf{(c)} and \textbf{(d)} follow the same lines as the proofs of (a) and (b). Therefore, we omit them. 
	\end{proof}
	\begin{lemma}[{\cite[Lemma 2]{Jon94}}{\cite[Chapter V Lemma 2]{JoWa84}}]\label{lem:basic_properties_extns_2}
		Let $s\in (0,1)$, $a>0$, $h:\Omega^c\to \R$. Set for $x\in \Omega$
		\begin{equation*}
			f(x)=\int\limits_{ B_{a\,l_Q}(q_Q) } h(t)\widetilde{\mu}_s (\d t), 
		\end{equation*}
		if $x\in \mathring{Q}$ for $Q\in J_i$, $i \in \N$, and $f(x)=0$ otherwise. There exist constants $C=C(d,a)>0$ and $a_0=a_0(d,a)>0$ such that for any $x_0\in \R^d$ and $0<r\le \infty$
		\begin{equation}\label{eq:basic_properties_help}
			\int\limits_{D_i\cap B_{r}(x_0)} f(x)\d x \le C\,2^{-id} \int\limits_{\Omega_{\sqrt{d}a 2^{-i}}^{\text{ext}}\cap B_{r+a_02^{-i}}(x_0) } \, h(t)\widetilde{\mu}_s(\d t).
		\end{equation}
	\end{lemma}
	
	\begin{lemma}\label{lem:new_norm_trace}
			Assume $a,b>0$ and $p^\star \geq 1$. There exists a  constant $C=C(a,b,d,p_\star)>0$ such that for $s\in (0,1)$ and $1\le p\le p^\star$
			\begin{align*}
			&\sum\limits_{j=0}^\infty \,2^{j(d+s(p-2))}\iint\limits_{\substack{ \abs{x-y}\le a 2^{-j} \\ d_x\le b2^{-j} }} \abs{g(x)-g(y)}^p\widetilde{\mu}_s(\d y)\widetilde{\mu}_s(\d x)\\
			&\qquad\qquad\le \frac{C}{d+s(p-2)}\, \iint\limits_{\substack{ \abs{x-y}\le a  \\ d_x\le b }} \frac{\abs{g(x)-g(y)}^p}{((\abs{x-y}+d_x+d_y)\wedge 1)^{d+s(p-2)}}\mu_s(\d y)\mu_s(\d x)
			\end{align*}
	\end{lemma}
	It is only due to this lemma that the norm of the extension operator in \autoref{th:trace-p-1} in the case $d=1$ depends on a lower bound of $(1-s)$.
	\begin{proof}
		The LHS of the inequality is equal to
		\begin{align*}
			&\sum\limits_{j=0}^\infty\sum\limits_{k,n\ge j}  \,2^{j(d+s(p-2))}\iint\limits_{\substack{ a2^{-n-1}\le \abs{x-y}\le a 2^{-n} \\ b2^{-k-1}\le d_x\le b2^{-k} }}\!\!\!\!\! \abs{g(x)-g(y)}^p\widetilde{\mu}_s(\d y)\widetilde{\mu}_s(\d x)\\
			&\quad = \Bigg( \sum\limits_{n\ge k\ge 0} \sum\limits_{j=0}^k + \sum\limits_{k> n\ge 0} \sum\limits_{j=0}^n \Bigg) \,2^{j(d+s(p-2))} \iint\limits_{\substack{a2^{-n-1}\le  \abs{x-y}\le a 2^{-n} \\ b2^{-k-1}\le  d_x\le b2^{-k} }} \!\!\!\!\!\abs{g(x)-g(y)}^p\widetilde{\mu}_s(\d y)\widetilde{\mu}_s(\d x)\\
			&\quad \le 2^{d+p+1} \sum\limits_{n, k\ge 0} \frac{2^{(k\wedge n)(d+s(p-2))}}{d+s(p-2)}(1+a+b)^{2d+2s(p-1)}\hspace{-1cm}\iint\limits_{\substack{ a2^{-n-1}\le \abs{x-y}\le a 2^{-n} \\ b2^{-k-1}\le  d_x\le b2^{-k} }} \!\!\!\!\!\abs{g(x)-g(y)}^p\mu_s(\d y)\mu_s(\d x) \\
			&\quad =:(\text{I}).
		\end{align*}
		For $x,y\in \Omega^c$ satisfying $\abs{x-y}\le a2^{-n}$ and $d_x\le b 2^{-k}$ we have $\abs{x-y}+d_x+d_y\le 2(a+b)2^{-(k\wedge n)}$. Therefore,
		\begin{align*}
			(\text{I})\le  \frac{(2(a+b) +1)^{4d+4p+1}}{d+s(p-2)}\iint\limits_{\substack{ \abs{x-y}\le a  \\ d_x\le b }}  \frac{\abs{g(x)-g(y)}^p}{((\abs{x-y}+d_x+d_y)\wedge 1)^{d+s(p-2)}}\mu_s(\d y)\mu_s(\d x).
		\end{align*}
	\end{proof}

	Now we are in the position to prove the continuity of the $L^p$-part. Recall the definition of the sets $\Omega_r^{\text{ext}}, \Omega^r_{\text{ext}}$ in \eqref{eq:def-Omega-decomp} for given $r>0$.
	
	\begin{proposition}\label{prop:extension_continuity_Lp_part}
		Let $s\in (0,1)$, $1\le p\le p^\star<\infty$ and assume that $\Omega$ satisfies \eqref{eq:estimate_a_{Q,s}}. Then for every measurable $g:\Omega^c\to \R$
		\begin{equation*}
			\norm{\extns(g)}_{L^p(\Omega)}\le \frac{C}{s^{1/p}}\, \norm{g}_{L^{p}(\Omega_{3\,\inr{\Omega}}^{\text{ext}  };\,\mu_s )}
		\end{equation*}
		with a constant $C=C(d,\Omega, p^\star)>0$, which is independent of $s$ and $p$.
	\end{proposition}

	\begin{proof}
		Firstly, $\int_{\Omega} \phi_Q(x)\d x\le \abs{Q^\star}\le (1+1/8)^{d}\,s_Q^{d}$ for any $Q\in \cW(\Omega)$. Let $c_1=c_1(d,\Omega)>1$ be the constant from \eqref{eq:estimate_a_{Q,s}}. Recall the definition of $\rho, \kappa$ in \eqref{eq:rho_kappa_whitney_cubes}. For any $Q\in \cW_{\le \kappa}$ and $z\in \Omega^c\cap B_{6l_Q}(q_Q)$ we know $d_z\le 6l_Q= 6 \sqrt{d} \, s_Q\le 6 \sqrt{d} \, 2^{\kappa} \leq 6 \rho$. We use the finite overlap property of the Whitney decomposition to estimate:
		\begin{align*}
			\norm{\extns(g)}_{L^p(\Omega)}^p &\le (1+1/8)^d N^p\,  \sum\limits_{Q\in \cW_{\le \kappa}(\Omega)} s_Q^{d} \Bigg(\, \fint\limits_{B_{6l_Q}(q_Q)} \abs{g(z)} \widetilde{\mu}_s(\d z)\Bigg)^p\\
			&\le 2^d c_1\, N^p\, (6 \rho + 1)^{d+p-1} \sqrt{d}^{s-d}\sum\limits_{Q\in \cW_{\le \kappa}(\Omega)} s_Q^{s}   \int\limits_{B_{6l_Q}(q_Q)} \abs{g(z)}^p \mu_s(\d z)=: (\star)
		\end{align*}
		Now, we consider $i\in \N$ and two cubes $Q_1, Q_2\in \cW_{\le \kappa}(\Omega)$ such that $s_{Q_1}=s_{Q_2}=2^{-i}$. If $\abs{q_{Q_1}-q_{Q_2}}\ge 6 (l_{Q_1}+l_{Q_2}) = 12\sqrt{d} 2^{-i}$, then $B_{6l_{Q_1}(q_{Q_1})} \cap B_{6l_{Q_2}(q_{Q_2})}= \emptyset$. The number of cubes $Q\in \cW_{\le \kappa}$ with side length $2^{-i}$ that fit in the ball $B_{12 \sqrt{d}2^{-i}}(q_{Q_1})$ is bounded from above by $\lceil 12 \sqrt{d} 2^{-i}/s_Q\rceil^d= \lceil 12 \sqrt{d}\rceil^d$. Therefore, 
		\begin{equation}\label{eq:overlaps_same_sized_cubes}
			\# \{ Q\in \cW_{\le \kappa}\,|\, s_Q =2^{-i},\, B_{6l_{Q_1}(q_{Q_1})} \cap B_{6l_{Q}(q_{Q})}\neq \emptyset \}\le \lceil 12 \sqrt{d}\rceil^d.
		\end{equation}
		We set $c_2:= 2^d c_1 N^p (6 \rho + 1)^{d+p-1}$. For any $z\in \Omega^c$ such that there exists a cube $Q\in \cW_{\le \kappa}$ which satisfies $z\in B_{6l_Q}(q_Q)$, we have $d_z\le \abs{z-q_Q}\le 6l_Q\le 6\rho$. We change the order of summation and use the above consideration to estimate
		\begin{align*}
			(\star)&\le c_2 \sum\limits_{i=0}^{\infty}\sum\limits_{Q\in \cW_{\le \kappa}(\Omega)\cap J_i} 2^{-is}   \int\limits_{B_{6\cdot2^{-i}}(q_Q)} \abs{g(z)}^p \mu_s(\d z)\\
			&\le c_2\,\lceil 12 \sqrt{d}\rceil^{d} \frac{1}{1-2^{-s}} \int\limits_{\Omega_{6\rho}^{\text{ext}}} \abs{g(z)}^p \mu_s(\d z)\le c_2\,\lceil 12 \sqrt{d}\rceil^{d}  \frac{2}{s}  \norm{g}_{L^{p}(\Omega_{3\,\inr{\Omega}}^{\text{ext}};\,\mu_s ) }^p.
		\end{align*}
		Thus, the proposition is proven with the constant $C:= 2^{d+1} c_1 N^{p^\star} (6 \rho + 1)^{d+p^\star-1}\,\lceil 12 \sqrt{d}\rceil^{d}$.
		\end{proof}

		\begin{proposition}\label{prop:extension_continuity_seminorm} Let $s\in (0,1)$, $1\le p\le p^\star<\infty$. We assume that $\Omega$ is bounded and satisfies \eqref{eq:estimate_a_{Q,s}}. Then for every $g\in \cT^{s,p}(\Omega^c)$
			\begin{equation*}
			[\extns(g)]_{\VspOm}\le \frac{C}{(d+s(p-2))^{1/p}s^{2/p}}\, \norm{g}_{\cT^{s,p}(\Omega^c)}
			\end{equation*}
			with a constant $C=C(d,\Omega,p^\star)>0$, which is independent of $s$ and $p$.
			
		\end{proposition}
		\begin{proof}
			We set $c_1:= (\sqrt{d}\,2^{\kappa+1})  \wedge 1/2^4<2\rho \wedge 1/8$, $j_0:= -\kappa -6$ where $\kappa, \rho$ are as in \eqref{eq:rho_kappa_whitney_cubes} and split the integration domain of the seminorm into 
			\begin{align*}
				[\extns(g)]_{\VspOm}^p &= (1-s)\, \int\limits_{ \Omega} \int\limits_{\{\abs{h}\ge c_1\}}\frac{\abs{\extns(g)(x)-\extns(g)(x+h)}^p}{\abs{h}^{d+sp}}\d h \d x\\				
				&\quad+\Bigg( \sum\limits_{j=j_0}^\infty (1-s)\, \int\limits_{ D_j} \int\limits_{\{\abs{h}<c_1\,2^{-j}\}}\frac{\abs{\extns(g)(x)-\extns(g)(x+h)}^p}{\abs{h}^{d+sp}}\d h \d x\Bigg)\\
				&\quad+\Bigg( \sum\limits_{j=j_0}^\infty (1-s)\, \int\limits_{ D_j} \int\limits_{\{c_1\,2^{-j}\le \abs{h}<c_1\}}\frac{\abs{\extns(g)(x)-\extns(g)(x+h)}^p}{\abs{h}^{d+sp}}\d h \d x\Bigg)\\
				&\quad + \Bigg( \sum\limits_{j<j_0} (1-s)\, \int\limits_{ D_j} \int\limits_{\{\abs{h}<c_1\}}\frac{\abs{\extns(g)(x)-\extns(g)(x+h)}^p}{\abs{h}^{d+sp}}\d h \d x\Bigg)\\
				&=: (\text{I})+ (\text{II})+ (\text{III})+ (\text{IV}).
			\end{align*}
			Recall that $D_j$ is the collection of Whitney cubes with side length $2^{-j}$, see \eqref{eq:def_D_j}. We handle these four terms separately. 
			
			\textit{Estimate of (IV):} For any $x\in D_j$, $j<j_0$ and $\abs{h}<c_1$, the distance of $x$ as well as $x+h$ to the boundary $\partial \Omega$ is bigger or equal to $d_x-\abs{h}\ge\distw{D_j}{\partial \Omega}-c_1\ge \sqrt{d}2^{-j-2}-c_1\ge  \sqrt{d}(2^{\kappa +4} - 2^{\kappa+1}) >  6\rho$. Therefore, $\extns(g)(x)=0=\extns(g)(x+h)$ and $(\text{IV})=0$. 
			
			\textit{Estimate of (I):} Using triangle inequality and the definition of the extension $\extns(g)$, we estimate $(\text{I})$ as follows:
			\begin{align*}
				(\text{I})&\le 2^p(1-s)\Bigg(\int\limits_{ \Omega} \int\limits_{\abs{h}\ge c_1}\frac{\abs{\extns(g)(x)}^p}{\abs{h}^{d+sp}}\d h \d x
				+ \int\limits_{ \Omega} \int\limits_{\abs{h}\ge c_1}\frac{\abs{\extns(g)(x+h)}^p}{\abs{h}^{d+sp}} \d h \d x\Bigg)\\
				&\le 2\frac{2^p(1-s)}{sp}c_1^{-sp}\omega_{d-1} \norm{\extns(g)}_{L^p(\Omega)}^p +  2^p\, 2^{-d-s(p-1)}\abs{\Omega}\int\limits_{ \Omega^c } \abs{g(y)}^p\mu_s(\d y)
			\end{align*}
			By \autoref{prop:extension_continuity_Lp_part} and the previous inequality, we find a constant $c_2=c_2(d,\Omega, p^\star)>0$ such that 
			\begin{equation*}
				(\text{I})\le\frac{1-s}{s^{2}}\frac{2^{p+1}\omega_{d-1}c_2^p}{pc_1^p} \norm{g}_{L^{p}(\Omega^c;\,\mu_s)}^p +  2^p\, \abs{\Omega}\norm{g}_{L^{p}(\Omega^c;\,\mu_s)}^p.
			\end{equation*}
			This is the desired estimate for $(\text{I})$. 
			
			\textit{Estimate of (II):} For the moment we fix $j\in \N$, $j\ge j_0$, $Q_j\in J_j$, $x\in Q_j$ and $\abs{h}<c_1 2^{-j}\le 1/2^4\, s_{Q_j}$. Under these assumptions $x+h\in Q^\star_j$ where $Q^\star_j$ is the cube with the same center as $Q_j$ but side length $(1+1/8)s_{Q_j}$. Thus, for any $t\in [0,1]$ the vector $x+th$ is either in $Q_j$ or in a neighboring cube, say $Q$ touching $Q_j$. By \eqref{eq:touching_cubes}, $1/4\, l_{Q_j}\le l_{Q}\le 4\, l_{Q_j}$ and $\abs{q_{Q_j}-q_Q}\le \diam(Q_j)+\diam(Q)\le (1+4)l_{Q_j}$. Further, for $z\in B_{30l_{Q}}(q_Q)$ we find $\abs{z-q_{Q_j}}\le (30\cdot 4 + 5)\,l_{Q_j}$. Set $c_3:=30\cdot 4 + 5$, let $c_4>0$ be the constant from \autoref{lem:basic_properties_extns} and $c_5>0$ the constant from \eqref{eq:estimate_a_{Q,s}}. We want to apply \autoref{lem:basic_properties_extns} (b) and (d) to estimate $(\text{II})$. We set $j_1:= j_0+8 = - \kappa + 2$ and write
			\begin{align*}
				(\text{II}) &=\sum\limits_{j=j_1}^\infty (1-s)\, \int\limits_{ D_j} \int\limits_{\{\abs{h}<c_1\,2^{-j}\}}\frac{\abs{\extns(g)(x)-\extns(g)(x+h)}^p}{\abs{h}^{d+sp}}\d h \d x\\
				&\quad+ \sum\limits_{j=j_0}^{j_1-1} (1-s)\, \int\limits_{ D_j} \int\limits_{\{\abs{h}<c_1\,2^{-j}\}}\frac{\abs{\extns(g)(x)-\extns(g)(x+h)}^p}{\abs{h}^{d+sp}}\d h \d x\\
				&=: (\text{II}_1)+ (\text{II}_2).
			\end{align*}
			
			For all $j\ge j_1$ all neighboring cubes of $Q_j$ have side lengths at most $2^{\kappa}$ and, thus, $B_{c_1 2^{-j}}(x)\subset D_{\ge -\kappa}$. Since $\extns(g)$ is smooth in $\Omega$, the fundamental theorem of calculus and \autoref{lem:basic_properties_extns} (b) yield
			\begin{align*}
				&\abs{\extns(g)(x)-\extns(g)(x+h)}= \bigabs{\int\limits_{0}^1 \nabla \extns(g)(x+th)\cdot h \d t}\le \abs{h}\, \sup\limits_{B_{c_12^{-j}}(x)} \abs{\nabla \extns(g)}\\
				&\qquad\le c_4\abs{h}\, \frac{2^{j+2}}{\sqrt{d}}\Bigg(c_5^2\sqrt{d}^{2(s-d)}\,2^{2(j+2)(d-s)}\!\!\!\! \int\limits_{( B_{c_3 l_{Q_j}}(q_{Q_j}) )^2}\!\! \abs{g(z_1)-g(z_2)}^p (\widetilde{\mu}_s\otimes \widetilde{\mu}_s)(\d (z_1,z_2))  \Bigg)^{1/p}
			\end{align*}  
			Using this, we estimate
			\begin{align*}
				(\text{II}_1)&\le  (1-s)\,c_4^p 4^{2(d-s)+p}c_5^2\sum\limits_{j=j_0}^\infty 2^{2j(d-s)+jp}\sum\limits_{Q\in J_j} \,\int\limits_{ \abs{h}<c_12^{-j} } \abs{h}^{-d+p(1-s)}\d h\\
				&\qquad\qquad \cdot \int\limits_{Q} \int\limits_{(B_{c_3l_{Q}}(q_{Q}) )^2} \abs{g(z_1)-g(z_2)}^p(\widetilde{\mu}_s\otimes \widetilde{\mu}_s)(\d (z_1,z_2)) \d x \\
				&\le  \frac{\omega_{d-1}\,c_1^{p(1-s)}c_4^p\,4^{2(d-s)+p}c_5^2}{p}\sum\limits_{j=j_0}^\infty 2^{2j(d-s)+jp-jp(1-s)}\sum\limits_{Q\in J_j} \\
				&\qquad\qquad\cdot\int\limits_{Q} \int\limits_{ \substack{(B_{c_3l_{Q}}(q_{Q}))^2\\  \abs{z_1-z_2}\le 2\sqrt{d}c_3\,2^{-j}}}  \abs{g(z_1)-g(z_2)}^p (\widetilde{\mu}_s\otimes \widetilde{\mu}_s)(\d (z_1,z_2)) \d x \,.
			\end{align*}
			We define the functions $f_j:\Omega\to \R$, $h_j:\Omega^c\to \R$ via
			\begin{align*}
				h_j(z_1)&:= \int\limits_{ \abs{z_1-z_2}\le 2\sqrt{d} c_3\,2^{-j}}  \abs{g(z_1)-g(z_2)}^p\widetilde{\mu}_s(\d z_2)\qquad (z_1\in \Omega^c),\\
				f_j(x)&:= \int\limits_{ B_{c_3l_{Q}}(q_{Q}) } h_j(z_1)\widetilde{\mu}_s(\d z_1)\qquad (x\in \Omega)
			\end{align*}
			whenever there exists $Q\in J_j$ such that $x\in \overset{\circ}{Q}$ and else we set $f_j=0$. With this notation we estimate $(\text{II}_1)$ using \autoref{lem:basic_properties_extns_2} and $d_z\le c_3\sqrt{d}2^\kappa$ for $z\in B_{c_3l_Q}(q_Q)\cap \Omega^c$ and $Q\in \cW_{\le \kappa}$
			\begin{align*}
				(\text{II}_1) &\le \frac{\omega_{d-1}\,c_1^{p(1-s)}c_4^p\,4^{2(d-s)+p}c_5^2}{p}\sum\limits_{j=j_0}^\infty 2^{2j(d-s)+jps} \int\limits_{D_j} f_j(x) \d x\\
				&\le \frac{\omega_{d-1}\,(c_1\vee 1)^{p}c_4^p\,4^{2d+p}c_5^2}{p}\,\,c_6\, \sum\limits_{j=j_0}^\infty 2^{j(d+ s(p-2))} \int\limits_{\Omega_{\sqrt{d}c_3 2^{-j}}^{\text{ext}}} h_j(z_1)\widetilde{\mu}_s(\d z_1)\\
				&\le \frac{c_8}{d+s(p-2)}\!\!\! \int\limits_{\substack{\Omega_{\sqrt{d} c_3 }^{\text{ext}}\times \Omega^c\\ \abs{z_1-z_2}\le 2\sqrt{d}c_3}} \!\!\!\!\! \frac{\abs{g(z_1)-g(z_2)}^p}{((\abs{z_1-z_2}+d_{z_1}+d_{z_2})\wedge 1)^{d+s(p-2)}}(\widetilde{\mu}_s\otimes \widetilde{\mu}_s)(\d (z_1,z_2)).
			\end{align*}
			In the last inequality we used \autoref{lem:new_norm_trace}. Here $c_6=c_6(d,c_3)>0$ is the constant from \autoref{lem:basic_properties_extns_2}, $c_7=c_7 (d,p, c_3)>0$ is the constant from \autoref{lem:new_norm_trace} and
			\begin{equation*}
				c_8=c_8(d,p,c_1,c_3,c_5,c_6, c_7):= \omega_{d-1}\,(c_1\vee 1)^{p}c_4^p\,4^{2d+p}c_5^2/p\,(c_3\sqrt{d}2^\kappa+1)^{2d+2(p-1)} \,c_6\,c_7.
			\end{equation*}
			Just as in the proof of the estimate of $(\text{II}_1)$ we apply the fundamental theorem of calculus and \autoref{lem:basic_properties_extns} (d), \eqref{eq:estimate_a_{Q,s}}, \autoref{lem:basic_properties_extns_2} to estimate $(\text{II}_2)$ by
			\begin{align*}
				(\text{II}_2)&\le \frac{\omega_{d-1} (c_12^{-j_0})^{p(1-s)}}{p} \sum\limits_{j=j_0}^{j_1-1} \sum\limits_{Q \in J_j} \int\limits_{Q} \max\limits_{ y\in Q^\star }\big| \nabla \extns(g)(y)\big|^p  \d x\\
				&\le  \frac{\omega_{d-1} (c_12^{-j_0})^{p(1-s)}}{p} \Big(c_4 \frac{2^{(j_1+1)}}{\sqrt{d}}\Big)^p c_5 (\sqrt{d}\, 2^{-j_1+1})^{s-d} \sum\limits_{j=j_0}^{j_1-1} \sum\limits_{Q \in J_j} \int\limits_{Q} \int\limits_{ B_{c_3 l_Q}(q_Q) } \abs{g(z)}^p \widetilde{\mu}_s(\d z)\d x \\
				&\le \frac{\omega_{d-1} (c_12^{-j_0})^{p(1-s)}}{p} \Big(c_4 \frac{2^{(j_1+1)}}{\sqrt{d}}\Big)^p c_5 (\sqrt{d}\, 2^{-j_1+1})^{s-d}  \sum\limits_{j=j_0}^{j_1-1} c_6 \int\limits_{ \Omega_{\sqrt{d}c_3 2^{-j }}^{\text{ext} } } \abs{g(z)}^p \widetilde{\mu}_s(\d z)\\
				&\le \frac{\omega_{d-1} (c_12^{-j_0}\vee 1)^{p}}{p} \Big(c_4 \frac{2^{(j_1+1)}}{\sqrt{d}}\Big)^p c_5  \frac{(1+\sqrt{d}c_3 2^{-{j_0}})^{d+p-1}}{(\sqrt{d}\, 2^{-j_1+1} \wedge 1)^{d}} c_6 (j_1-j_0) \norm{g}_{L^{p}(\Omega_{\sqrt{d}c_3 2^{-j_0} }^{\text{ext}};\,\mu_s )}^p.
			\end{align*}
			
				\textit{Estimate of (III):} We put $h_m:= c_1\,2^{-m}$ and write
			\begin{align*}
				(\text{III})&= (1-s)\, \sum\limits_{j=j_0}^\infty \sum\limits_{m=0}^{j-1}\, \int\limits_{ D_j} \int\limits_{\{h_{m+1}\le \abs{h}<h_{m}\}}\frac{\abs{\extns(g)(x)-\extns(g)(x+h)}^p}{\abs{h}^{d+sp}}\d h \d x\\
				&\le (1-s)\,\sum\limits_{m=0}^{\infty}\sum\limits_{j=m+j_0}^\infty\, \int\limits_{ D_j} \int\limits_{\{h_{m+1}\le \abs{h}<h_{m}\}}\frac{\abs{\extns(g)(x)-\extns(g)(x+h)}^p}{\abs{h}^{d+sp}}\d h \d x\\
				&=(1-s)\,\sum\limits_{m=0}^{\infty}\, \int\limits_{ D_{\ge m+j_0}} \int\limits_{\{h_{m+1}\le \abs{h}<h_{m}\}}\frac{\abs{\extns(g)(x)-\extns(g)(x+h)}^p}{\abs{h}^{d+sp}}\d h \d x\\
				&\le (1-s)\,\sum\limits_{m=0}^{\infty}\,h_{m+1}^{-d-sp}\Bigg( \int\limits_{ D_{\ge m+j_0}} \int\limits_{\Omega}\abs{\extns(g)(x)-\extns(g)(y)}^p\1_{\abs{x-y}\le h_m}\d y \d x\\
				&\qquad+\int\limits_{ D_{\ge m+j_0}} \int\limits_{ \Omega^c}\abs{\extns(g)(x)-g(y)}^p\1_{\abs{x-y}\le h_m}\d y \d x \Bigg)\\
				&=: (1-s)\Big((\text{III}_1)+(\text{III}_2)\Big).
			\end{align*}
			\textit{Estimate of ($\text{III}_1$):} For $x\in Q\subset D_{\ge m+j_0}$ and $y\in \Omega$ such that $\abs{x-y}\le h_m$, we have $d_y\le d_x+ \abs{x-y}\le \distw{Q}{\partial \Omega} + \diam(Q)+ h_m \le (5 \sqrt{d} + c_1) 2^{-m-j_0}\le \sqrt{d}\,2^{3-j_0-m}$ and, therefore, $y\in D_{\ge m+j_0-3}$ by \eqref{eq:distance_diam_comparable}. We calculate
			\begin{align}\label{eq:seminorm_extension_cont_interior_far_away_1}
				(\text{III}_1)&\le \sum\limits_{m=0}^{\infty}\sum\limits_{k=m+j_0}^\infty\sum\limits_{n= m+j_0-3}^\infty\,h_{m+1}^{-d-sp}  \iint\limits_{ \substack{ D_{n}\,D_{k}\\ \abs{x-y}\le h_m}}\abs{\extns(g)(x)-\extns(g)(y)}^p\d x\d y \nonumber \\
				&=\sum\limits_{m=j_1-j_0+3}^{\infty}\sum\limits_{k=m+j_0}^\infty\sum\limits_{n= m+j_0-3}^\infty\,h_{m+1}^{-d-sp}  \iint\limits_{ \substack{ D_{n}\,D_{k}\\ \abs{x-y}\le h_m}}\abs{\extns(g)(x)-\extns(g)(y)}^p\d x\d y \nonumber \\
				&\quad+ \sum\limits_{m=0}^{j_1-j_0+2}\sum\limits_{k=m+j_0}^{\infty}\sum\limits_{n= m+j_0-3}^{\infty}\,h_{m+1}^{-d-sp}  \iint\limits_{ \substack{ D_{n}\,D_{k}\\ \abs{x-y}\le h_m}}\abs{\extns(g)(x)-\extns(g)(y)}^p\d x\d y\nonumber\\
				&=: (\text{III}_{1,1})+ (\text{III}_{1,2}).
			\end{align}
			We estimate $(\text{III}_{1,1})$ first. For $m\ge j_1-j_0+3=11$, $x\in Q_k\in J_k$ and $y\in Q_n\in J_n$, $n\ge m+j_0-3$, $k\ge m+j_0$ such that $\abs{x-y}\le h_m$ and $z_1\in B_{30l_{Q_k}}(q_{Q_k})$, $z_2\in B_{30l_{Q_n}}(q_{Q_n})$ we have $n,k\ge j_1$ and  $\abs{z_1-z_2}\le \abs{z_1-q_{Q_k}}+ \abs{q_{Q_k}-x}+\abs{x-y}+ \abs{y-q_{Q_n}}+\abs{q_{Q_n}-z_2}\le 31\sqrt{d} 2^{-k}+ c_1 2^{-m}+ 31\sqrt{d} 2^{-n} \le c_9\,2^{-m}$, where $c_9:= 31\sqrt{d}2^{-j_0}+c_1+31\sqrt{d} 2^{3-j_0}$. Note that $s_{Q_k}, s_{Q_n}\le 2^{\kappa-2}$. \autoref{lem:basic_properties_extns} (a) yields 
			\begin{align}\label{eq:seminorm_extension_cont_interior_far_away_2}
				\abs{\extns(g)(x)-\extns(g)(y)}^p&\le c_4^p\, a_{Q_k,s}a_{Q_n,s}  \!\!\! \iint\limits_{ \substack{ B_{30l_{Q_k}}(q_{Q_k}) \\B_{30l_{Q_n}}(q_{Q_n}) } } \!\!\!\! \abs{g(z_1)-g(z_2)}^p\1_{\abs{z_1-z_2}\le c_{9}2^{-m}} \widetilde{\mu}_s(\d z_2)\widetilde{\mu}_s(\d z_1)\nonumber\\
				&=  c_4^p\, a_{Q_k,s}a_{Q_n,s} \tilde{f}_m(x)
			\end{align}
			where the $\tilde{f}_m:\R^d\to \R$ is defined by 
			\begin{equation*}
				\tilde{f}_m(x):=  \int\limits_{B_{30l_{Q}}(q_{Q})  } \tilde{h}_m(z_1)\widetilde{\mu}_s(\d z_1) 
			\end{equation*}
			for $x\in \R^d$ whenever there exists $Q\in J_k$ such that $x\in \overset{\circ}{Q}$ and else we set $\tilde{f}_m=0$. Here $\tilde{h}_m:\Omega^c\to \R$ is defined by
			\begin{align*}
				\tilde{h}_m(z_1)&:= \int\limits_{B_{30 l_{Q_n}  }(q_{Q_n})} \abs{g(z_1)-g(z_2)}^p\1_{\abs{z_1-z_2}\le c_{9}2^{-m}} \widetilde{\mu}_s(\d z_2)\qquad (z_1\in \Omega^c).
			\end{align*}
			 Thus, \autoref{lem:basic_properties_extns_2} together with \eqref{eq:seminorm_extension_cont_interior_far_away_2} yields 
			\begin{align}\label{eq:seminorm_extension_cont_interior_far_away_2_1}
				 &\iint\limits_{ \substack{D_{n}\, D_{k}\\ \abs{x-y}\le h_m}}\abs{\extns(g)(x)-\extns(g)(y)}^p\d y \d x \le c_4^p a_{Q_k,s}a_{Q_n,s}\, \iint\limits_{ \substack{D_{n}\, D_{k}\\ \abs{x-y}\le h_m}} \tilde{f}_m(x)\d x \d y\nonumber\\
				 &\qquad\le c_4^p c_{10}\,2^{-kd}  a_{Q_k,s}a_{Q_n,s}\, \int\limits_{D_{n}} \int\limits_{B_{h_m+a_12^{-k}}(y)}\tilde{h}_m(z_1)\widetilde{\mu}_s(\d z_1) \d y=:(\widetilde{\text{III}}_{1,1}).
			\end{align}
			Here $c_{10}=c_{10}(d)>0$ and $a_1=a_1(d)>0$ are the constants from \autoref{lem:basic_properties_extns_2}. For $y\in D_n$ and $z_1\in B_{h_m+a_1 2^{-k}}(y)\cap \Omega^c$ we find $d_{z_1}\le \abs{y-z_1}\le (c_1+a_1 2^{-j_0})2^{-m}$ and set $ c_{11}:=(c_1+a_1 2^{-j_0})$. Then, $(\widetilde{\text{III}}_{1,1})$ becomes after applying \autoref{lem:basic_properties_extns_2} again 
			\begin{align}\label{eq:seminorm_extension_cont_interior_far_away_3}
				(\widetilde{\text{III}}_{1,1})&\le \frac{c_4^p c_{10}}{2^{kd}} \,  a_{Q_k,s}a_{Q_n,s}\, \int\limits_{D_{n}} \int\limits_{B_{30 l_{Q_n}}(q_{Q_n})  }\int\limits_{\Omega_{c_{11}2^{-m}}^{\text{ext}}} \abs{g(z_1)-g(z_2)}^p\1_{\abs{z_1-z_2}\le c_92^{-m}} \widetilde{\mu}_s(\d z_1) \widetilde{\mu}_s(\d z_2) \d y\nonumber\\
				&\le  \frac{c_4^p c_{12}c_{10}}{2^{(k+n)d}}  a_{Q_k,s}a_{Q_n,s}\,\int\limits_{\Omega^c} \int\limits_{\Omega_{c_{11}2^{-m}}^{\text{ext}}} \abs{g(z_1)-g(z_2)}^p\1_{\abs{z_1-z_2}\le c_92^{-m}} \widetilde{\mu}_s(\d z_1) \widetilde{\mu}_s(\d z_2)\nonumber \\
				&\le \frac{c_4^p c_5 c_{12}c_{10}}{2^{s(k+n)}} \,\int\limits_{\Omega^c} \int\limits_{\Omega_{c_{11}2^{-m}}^{\text{ext}}} \abs{g(z_1)-g(z_2)}^p\1_{\abs{z_1-z_2}\le c_92^{-m}} \widetilde{\mu}_s(\d z_1) \widetilde{\mu}_s(\d z_2).
			\end{align}
			In the last inequality we used \eqref{eq:estimate_a_{Q,s}} and $c_{12}=c_{12}(d)>0$ is the constant from \autoref{lem:basic_properties_extns_2}. We set $c_{13}:= c_4^p  c_5 c_{10} c_{12} 2^{d+p}\, (c_0 \wedge 1)^{-d-p}$. Recall that $j_1-j_0=8$. The estimates \eqref{eq:seminorm_extension_cont_interior_far_away_2_1} and \eqref{eq:seminorm_extension_cont_interior_far_away_3} yield 
			\begin{align*}
				&(\text{III}_{1,1})\le c_{13} \sum\limits_{m=j_1-j_0+3}^{\infty}\sum\limits_{k=m+j_0}^\infty\sum\limits_{n= m+j_0-3}^\infty\,2^{m(d+sp)} \,2^{-s(k+n)}  \\
				&\qquad\qquad\times \quad  \int\limits_{\Omega^c} \int\limits_{\Omega_{c_{11}2^{-m}}^{\text{ext}}} \abs{g(z_1)-g(z_2)}^p\1_{\abs{z_1-z_2}\le c_92^{-m}} \widetilde{\mu}_s(\d z_1) \widetilde{\mu}_s(\d z_2)\\
				&= \frac{c_{13}}{2^{2sj_0-3s}}  \Big(\frac{2^s}{2^s-1}\Big)^2 \sum\limits_{m=11}^{\infty}2^{m(d+s(p-2))} \int\limits_{\Omega^c} \int\limits_{\Omega_{c_{11}2^{-m}}^{\text{ext}}} \abs{g(z_1)-g(z_2)}^p\1_{\abs{z_1-z_2}\le c_92^{-m}} \widetilde{\mu}_s(\d z_1) \widetilde{\mu}_s(\d z_2)\\
				&\le c_{14} \frac{(1+c_9+c_{11})^{2d+2(p-1)}}{s^2(d+s(p-2))}  \,\iint\limits_{\substack{\Omega_{c_{11}  }^{\text{ext}} \, \Omega^c\\ \abs{z_1-z_2}\le c_9}  } \frac{\abs{g(z_1)-g(z_2)}^p}{((\abs{z_1-z_2}+d_{z_1}+d_{z_2})\wedge 1)^{d+s(p-2)}}\mu_s(\d z_1) \mu_s(\d z_2).
			\end{align*}
			In the last estimate we used \autoref{lem:new_norm_trace}. Here $c_{14}:= 2^{5-2(j_0\wedge 0)}c_{13} c_{15}$ and $c_{15}=c_{15}(d,p,c_9,c_{11})>0$ is the constant from \autoref{lem:new_norm_trace}. This is the desired estimate for $(\text{III}_{1,1})$. To estimate $(\text{III}_{1,2})$, we calculate
			\begin{align*}
				(\text{III}_{1,2})&\le \frac{j_1-j_0+2}{\big(c_1 2^{-(j_1-j_0-1)}\big)^{d+sp}} \sum\limits_{k=j_0}^{\infty}\sum\limits_{n= j_0-3}^{\infty}  \iint\limits_{ \substack{ D_{k}\, D_{n}\\ \abs{x-y}\le c_1}}\abs{\extns(g)(x)-\extns(g)(y)}^p\d x\d y\\
				&\le 2^{p+1}\frac{10}{\big(c_1 2^{-7}\big)^{d+sp}} \sum\limits_{k,n= j_0-3}^{\infty}  \iint\limits_{ \substack{ D_{k}\, D_{n}\\ \abs{x-y}\le c_1}}\abs{\extns(g)(x)}^p\d x\d y\\
				&\le \frac{2^{p+5}}{\big(c_1 2^{-7}\big)^{d+sp}} \omega_{d-1}c_1^{d} \norm{\extns(g)}_{L^p(\Omega)}^p\le \frac{2^{p+5}\,\omega_{d-1}c_1^{d}}{\big(c_1 2^{-7} \wedge 1\big)^{d+p}}  \,\frac{c_2^p}{s} \norm{g}_{L^{p}(\Omega_{3 \inr{\Omega}}^{\text{ext}};\,\mu_s)}^p. 
			\end{align*}
			Here we used \autoref{prop:extension_continuity_Lp_part}. Combining the estimates of $(\text{III}_{1,1})$ and $(\text{III}_{1,2})$ yields the desired estimate of $(\text{III}_1)$. 
			
			\textit{Estimate of $(\text{III}_2)$:} In contrast to ($\text{III}_1$) we integrate over $\Omega^c$ where the extension is just $g$. For $x\in Q\in J_{k}$, $k\ge m+j_0$ and $y\in \Omega^c$ such that $\abs{x-y}\le h_m= c_1 2^{-m}$, the distance of $y$ to the center of the cube $q_{Q}$ is smaller than $d_y\le \abs{y-q_Q}\le \abs{y-x}+ \abs{x-q_Q}\le c_1 2^{-m} + \sqrt{d}2^{-k}\le \sqrt{d}(c_1 + 2^{-j_0}) 2^{-m}$. Set $c_{16}:= 2c_1^{-1}\sqrt{d}(c_1 + 2^{-j_0}+1)$. Thus, $(1-s)c_{16}^{-s} h_{m+1}^{-s}\le (1-s)d_y^{-s}$. We use this and split $(1-s) \,(\text{III}_2)$ into 

			\begin{align}\label{eq:seminorm_extension_cont_exterior}
				&(1-s) (\text{III}_2) \le \sum\limits_{m=0}^{j_1-j_0-1}\frac{(1-s)}{h_{m+1}^{d+sp}}\sum\limits_{k=m+j_0}^\infty\sum\limits_{Q\in J_k} \int\limits_{ Q} \int\limits_{ B_{c_{16}2^{-m}}(q_Q)}\hspace{-0.6cm} \abs{\extns(g)(x)-g(y)}^p\1_{\abs{x-y}\le h_m} \d y \d x\nonumber\\
				&\quad+ c_{16}^{s}  \sum\limits_{m=j_1-j_0}^{\infty}h_{m+1}^{-d-s(p-1)}\sum\limits_{k=m+j_0}^\infty\sum\limits_{Q\in J_k} \int\limits_{ Q} \int\limits_{ B_{c_{16}2^{-m}}(q_Q)}\hspace{-0.6cm} \abs{\extns(g)(x)-g(y)}^p\1_{\abs{x-y}\le h_m} \widetilde{\mu}_s(\d y) \d x\nonumber\\
				&=:  (\text{III}_{2,1})+(\text{III}_{2,2}).
			\end{align}	
				
			We estimate $(\text{III}_{2,1})$ by
			\begin{align*}
				(\text{III}_{2,1})&\le  2^p  (j_1-j_0)\Bigg( (1-s)\sum\limits_{k=j_0}^\infty\sum\limits_{Q\in J_k} \int\limits_{ Q} \int\limits_{ B_{c_{16}}(q_Q)}\hspace{-0.3cm} \abs{\extns(g)(x)}^p\1_{\abs{x-y}\le c_1} \d y \d x \\
				&\qquad\qquad\qquad\qquad +c_{16}^{d+sp}  \sum\limits_{k=j_0}^\infty\sum\limits_{Q\in J_k} \int\limits_{ Q} \int\limits_{ B_{c_{16}}(q_Q)}\hspace{-0.3cm} \abs{g(y)}^p\1_{\abs{x-y}\le c_1} \mu_s(\d y) \d x\Bigg)\\
				&\le  2^p  \,8\Big( (1-s)\omega_{d-1} c_1^{d}\norm{\extns(g)}_{L^p(\Omega)}^p+ c_{16}^{d+sp} \omega_{d-1} c_1^{d} \norm{g}_{L^{p}(\Omega_{c_{16}}^{\text{ext}};\,\mu_s) }^p\Big)\\
				&\le 2^p  \,8\Big( (1-s)\omega_{d-1} c_1^{d} \frac{c_2}{s}\norm{g}_{L^{p}(\Omega_{3 \inr{\Omega}}^{\text{ext}};\,\mu_s)}^p+ c_{16}^{d+sp} \omega_{d-1} c_1^{d} \norm{g}_{L^{p}(\Omega_{c_{16}}^{\text{ext}};\,\mu_s) }^p\Big).
			\end{align*}
			
			Here, we used \autoref{prop:extension_continuity_Lp_part}. We apply \autoref{lem:basic_properties_extns} (c) and \eqref{eq:estimate_a_{Q,s}} to estimate $(\text{III}_{2,2})$.
			\begin{align}\label{eq:seminorm_extension_cont_exterior_1}
				(\text{III}_{2,2})\le  \sum\limits_{m=j_1-j_0}^{\infty}\frac{c_4^p\,c_5 c_{16}^{s} }{h_{m+1}^{d+s(p-1)}}\sum\limits_{k=m+j_0}^\infty 2^{k(d-s)} \sum\limits_{Q\in J_k}\int\limits_{Q}\int\limits_{ B_{30l_Q}(q_Q)} \int\limits_{ B_{\frac{c_{16}}{2^{m}}  }(q_Q)} \hspace{-0.5cm}\abs{g(z)-g(y)}^p\widetilde{\mu}_s(\d y) \widetilde{\mu}_s(\d z)\d x.
			\end{align}
			 For $z\in B_{30\sqrt{d} 2^{-k}}(q_Q)$ and $y\in \Omega^c$ with $\abs{q_Q-y}\le c_{16}2^{-m}$ we have $ \abs{y-z}\le (c_{16}+ \sqrt{d}30)2^{-m}$. We write $c_{17}:= c_{16}+\sqrt{d}30$. Now, we apply \autoref{lem:basic_properties_extns_2} with $r = +\infty$ and conclude with a positive constant $c_{18}=c_{18}(d)$ 
			\begin{align}\label{eq:seminorm_extension_cont_exterior_2}
			\begin{split}
				&\sum\limits_{Q\in J_k}\int\limits_{Q}\int\limits_{ B_{30l_Q}(q_Q)} \int\limits_{ B_{c_{16}2^{-m}  }(q_Q)} \hspace{-0.5cm}\abs{g(z)-g(y)}^p\widetilde{\mu}_s(\d y) \widetilde{\mu}_s(\d z)\d x \\
				&\qquad\le c_{18}2^{-kd}\int\limits_{\Omega_{30\sqrt{d} 2^{-m} }^{\text{ext}} } \int\limits_{B_{c_{17}2^{-m}}(z)} \hspace{-0.5cm}\abs{g(z)-g(y)}^p\widetilde{\mu}_s(\d y) \widetilde{\mu}_s(\d z) \,.
				\end{split}
			\end{align}
			By \eqref{eq:seminorm_extension_cont_exterior_1}, \eqref{eq:seminorm_extension_cont_exterior_2} and \autoref{lem:new_norm_trace},
			\begin{align*}
				(\text{III}_{2,2})&\le \frac{c_4^p c_5 c_{16}^{s}}{c_1^{d+s(p-1)}} \frac{2^{-s(j_0\wedge 0)}}{1-2^{-s}} c_{18} \sum\limits_{m=0}^{\infty}\,2^{m(d+s(p-2))} \!\!\!\!\! \!\! \int\limits_{\Omega_{30\sqrt{d} 2^{-m} }^{\text{ext}} } \int\limits_{B_{c_{17}2^{-m}}(z)} \hspace{-0.5cm}\abs{g(z)-g(y)}^p\widetilde{\mu}_s(\d y) \widetilde{\mu}_s(\d z)\\
				&\le \frac{c_{20}}{s(d+s(p-2))} \int\limits_{ \Omega_{30\sqrt{d}}^{\text{ext}}} \int\limits_{B_{c_{17}}(z)} \frac{\abs{g(z)-g(y)}^p}{((\abs{y-z}+d_z+d_y)\wedge 1)^{d+s(p-2)}}\mu_s(\d y) \mu_s(\d z).
			\end{align*}
			Here the constant is
			\begin{equation*}
				c_{20}:= 2^{2-(j_0\wedge 0)}\,(c_1\wedge 1)^{-d-(p-1)}c_4^p\,c_5 c_{16} (60\sqrt{d} +2c_{17})^{d+\abs{p-2}} c_{19}
			\end{equation*}
			where $c_{19}=c_{19}(d,p,c_{17})>0$ is the constant from \autoref{lem:new_norm_trace}. Combining $(\text{III}_{2,1})$ and $(\text{III}_{2,2})$ yields the desired estimate of $(\text{III}_{2})$. Further, the previous estimates of $(\text{III}_{1})$ and $(\text{III}_{2})$ finish the proof of the bound on $(\text{III})$.
		\end{proof}

\enlargethispage*{6ex}
\begin{proof}[Proof of \autoref{th:trace-extension} and \autoref{th:trace-p-1}]
	The proof of \autoref{th:trace-extension} follows from \autoref{prop:trace_Lp_estimate}, \autoref{prop:trace_seminorm_estimate}, \autoref{lem:estimate_a_{Q,s}}, \autoref{prop:extension_continuity_Lp_part} and \autoref{prop:extension_continuity_seminorm}. The proof of \autoref{th:trace-extension} does not require \autoref{prop:trace_seminorm_estimate}.
\end{proof}

	\section{Nonlocal to Local}\label{sec:convergence_pointwise}
	In this section, we prove the convergence of the trace spaces $\cT^{s,p}(\Omega^c)\to W^{1-1/p,p}(\partial \Omega)$ as $s\to 1^{-}$ as claimed in \autoref{th:asym_trace}.
	
	\medskip
	
	The following lemma is a minor extension of \cite[Lemma 4.1]{GrHe22} to a Lipschitz domains. Note that the term $\frac{1-s}{d_x^s}$ in the measure $\mu_s$ is responsible for the reduction of $\Omega^c$ to $\partial \Omega$. Recall the definition of the sets $\Omega_r^{\text{ext}}, \Omega^r_{\text{ext}}$ in \eqref{eq:def-Omega-decomp} for given $r>0$.
	
	\begin{lemma}\label{lem:weak_convergence_measures}
		Let $\Omega\subset \R^d$ be a Lipschitz domain, $d\ge 2$ and $0<r\le \infty$. We define a family of measures $\bar{\mu}_s$ on $(\R^d, \cB(\R^d))$ via
		\begin{align*}
			\bar{\mu}_s(x):= \frac{1-s}{ d_x^{s}} \1_{\Omega_{r}^{\text{ext}}}(x).
		\end{align*}
		Let $\sigma$ be the surface measure on the Lipschitz submanifold $\partial \Omega$ and set $\sigma(D):= \sigma(\partial \Omega \cap D)$ for all sets $D\in \cB(\R^d)$. Then $\{\bar{\mu}_s\}_s$ converges weakly to $\sigma$ as $s\to 1^{-}$, i.e. when integrated against test functions in $C_c(\R^d)$. 
	\end{lemma}

\begin{remark*}{\ }
	\begin{enumerate} 
		\item In dimension $d=1$ the previous convergence result reads
\begin{equation*}
\overline{\mu}_s\to \sum\limits_{x_0\in \partial \Omega} \delta_{x_0} \text{ weakly ,}
\end{equation*}
i.e. when tested against $C_c(\R)$ functions. Here $\delta_{x_0}$ is the Dirac measure in the boundary point $x_0\in \partial \Omega$. 
 \item In \cite[Lemma 4.1]{GrHe22} the first author and his coauthor proved \autoref{lem:weak_convergence_measures} for bounded $C^1$-domains. Below, we adopt this proof and explain necessary differences to \cite[Lemma 4.1]{GrHe22}. 
\end{enumerate}

\end{remark*}

	\begin{proof}
		Let $f\in C_c(\R^d)$. We have shown in \cite[Lemma 4.1]{GrHe22} without using that the boundary $\partial \Omega$ is $C^1$ that $\int_{\Omega_{\text{ext}}^{\eps} }\abs{f}\overline{\mu}_s\to 0$ as $s\to 1^{-}$ for any $\eps>0$. Thus, the problem localizes. Without loss of generality we find a cube $Q= (-\rho,\rho)^d$ and a Lipschitz continuous map $\phi:\R^{d-1}\to \R$ such that $\Omega \cap Q = \{ (x',x_d)| x_d<\phi(x') \}\cap Q$. Furthermore, on this cube we have
		\begin{align*}
			\int\limits_{Q\cap \Omega^c} f \d \overline{\mu}_s = \int\limits_{(-\rho,\rho)^{d-1}} \int\limits_{\phi(x')}^\rho f(x',x_d) \frac{1-s}{d_{(x',x_d)}^s} \d x_d \d x'.
		\end{align*}
		Since $\frac{1-s}{x_d-\phi(x')}$ is an approximate identity evaluating at $x_d= \phi(x')$ as $s\to 1^{-}$, it remains to show that 
		\begin{equation}\label{eq:convergence_measure_lts}
			\frac{x_d-\phi(x')}{d_{(x',x_d)}}\to \sqrt{1+\abs{\nabla \phi(x')}^2 } \qquad \text{ as }x_d\to \phi(x')
		\end{equation}
		for almost every $x'\in (-\rho,\rho)^{d-1}$. Since $\phi$ is Lipschitz continuous, it is differentiable at almost every point $x'\in (-\rho, \rho)^{d-1}$ by Rademacher's theorem. In \cite[Lemma 4.1]{GrHe22} we used continuous differentiability of $\phi$ to show \eqref{eq:convergence_measure_lts}. Here we only assume $\phi$ to be Lipschitz continuous. We fix $x'\in(-\rho,\rho)^{d-1}$ such that $\phi$ is differentiable at $x'$. For $x_d> \phi(x')$ such that $x_d< \rho$ we pick a minimizer of the distance of $x=(x',x_d)$ to the surface $\partial \Omega$ and call it $y=y(x_d)=(y',\phi(y'))$. Analogous to the estimate \eqref{eq:distance_estimate} we have
		\begin{align}\label{eq:upper_bound_distance_to_vertical_distance}
			\abs{x_d-\phi(x')}&\le 2 \big(1+[\phi]_{C^{0,1}}\big)d_x
		\end{align}
		Furthermore, we define the hyperplane $P:=\{ (z',\phi(x')+(z'-x')\cdot\nabla \phi(x'))\,|\, z'\in (-\rho,\rho)^{d-1} \}$ which is tangential to the surface at $x$. Let $z=z(x_d)=(z',z_d)$ be the minimizer of $x=x(x',x_d)$ to the plane $P$. A small calculation yields 
		\begin{gather*}
			z'=x'+\frac{x_d-\phi(x')}{1+\abs{\phi(x')}^2} \nabla\phi(x'), \quad z_d= \phi(x')+ \frac{\abs{\nabla \phi(x')}^2}{1+\abs{\nabla\phi(x')}^2} (x_d-\phi(x')),\\
			\distw{x}{P}= \abs{x-z}= \frac{x_d-\phi(x')}{\sqrt{1+\abs{\nabla \phi(x')}^2  }}.
		\end{gather*}
		Since $\phi$ is differentiable, the error function $r:(-\rho,\rho)^{d-1}\to \R$ given by $r(w'):= \phi(w')-\phi(x')-(w'-x')\cdot \nabla \phi(x')$ satisfies 
		\begin{equation*}
			\frac{r(w')}{\abs{w'-x'}}\to 0 \quad \text{ as }w'\to x'.
		\end{equation*}
		Since $\abs{y'-x'}\le d_x \le \abs{x_d-\phi(x')}$, we know 
		\begin{equation}\label{eq:convergence_measure_remainder}
			\frac{r(y')}{\abs{x_d-\phi(x')}}\to 0 \quad \text{ as }x_d\to \phi(x').
		\end{equation}
		Now, we will estimate $\distw{x}{P}$ by $d_x$ and an error and vice versa. Let $y_d\in \R$ such that $(y',y_d) \in P$. By triangle inequality, we have \begin{equation*}
			\abs{(y',y_d)- (x',x_d)}\le \abs{(y',\phi(y'))-(x',x_d) }+ \abs{(y',y_d)-(y',
				\phi(y')} = d_x + r(y').
		\end{equation*}
		Since $(y',y_d)\in P$ and $z$ is the minimizer of the distance of $x$ to $P$, we have $\distw{x}{P}\le d_x+ r(y')$. Again by triangle inequality 
		\begin{equation*}
			d_x\le \abs{(z',\phi(z')) - (x',x_d) }\le \abs{(z',z_d)- (x',x_d)}+ \abs{(z',\phi(z')) - (z',z_d) }= \distw{x}{P} + r(z').
		\end{equation*}
		Therefore, 
		\begin{align*}
			\abs{1 - \frac{d_x}{\distw{x}{P}}}&= \frac{\abs{\distw{x}{P}- d_x}}{\distw{x}{P}} \le \frac{\max \{\abs{r(y')}, \abs{r(z')}\}}{\distw{x}{P}}
			\\
			&= \sqrt{1+\abs{\nabla \phi(x')}^2} \frac{\max \{\abs{r(y')}, \abs{r(z')}\}}{\abs{x_d-\phi(x')}} \to 0 
		\end{align*}
		as $x_d\to \phi(x')$ by \eqref{eq:convergence_measure_remainder}, the choice of $z'$ and the properties of the error function $r$. By the previous calculation of $\distw{x}{P}$ and this convergence, \eqref{eq:convergence_measure_lts} follows. Since $(1-s)/\abs{x_d-\phi(x')}^{s}$ is an approximate identity, we have for any $x'\in (-\rho,\rho)^{d-1}$ such that $\phi$ is differentiable at $x'$
		\begin{align*}
			\int\limits_{\phi(x')}^{\rho} \frac{1-s}{d_{(x',x_d)}^s}f(x',x_d)\d x_d \to f(x',0)\sqrt{1+\abs{\nabla \phi(x')}^2  } \quad \text{ as } s\to 1^{-}.
		\end{align*}
		Since $f$ has compact support, there exists $R>0$ such that $\supp(f)\subset B_R(0)$. By \eqref{eq:upper_bound_distance_to_vertical_distance}, we have
		\begin{align*}
			\bigg|	\int\limits_{\phi(x')}^{\rho} \frac{1-s}{d_{(x',x_d)}^s}f(x',x_d)\d x_d\bigg|&\le \norm{f}_{L^\infty}\1_{B_R(0)}(x')\, \int\limits_{\phi(x')}^{\rho} \frac{1-s}{d_{(x',x_d)}^s}\d x_d\\
			&\le 2(1+[\phi]_{C^{0,1}})  \norm{f}_{L^\infty}\1_{B_R(0)}(x')\, \int\limits_{\phi(x')}^{\rho} \frac{1-s}{\abs{x_d-\phi(x')}^s }\d x_d\\
			&\le 2(1+[\phi]_{C^{0,1}}) (\rho+1) \norm{f}_{L^\infty}\1_{B_R(0)}(x').
		\end{align*}
		By dominated convergence, 
		\begin{equation*}
			\int\limits_{Q\cap \Omega^c} f \d \overline{\mu}_s \to \int\limits_{\partial \Omega \cap Q} f \d \sigma \quad \text{ as }s \to 1^{-}.
		\end{equation*}
		 Following the proof of \cite[Lemma 4.1]{GrHe22}, the result follows. 
	\end{proof}

	We are now in the position to prove the convergence theorem. We use similar arguments as in the proof of \cite[Proposition 4.2, Theorem 1.4]{GrHe22}.
	
	\medskip
	
	\begin{proof}[Proof of \autoref{th:asym_trace}] First, we prove the convergence result for $u\in C_c^{0,1}(\R^d)$. Then we apply a density argument
	
	\textbf{Step 1: } Let us assume $u\in C_c^{0,1}(\R^d)$. 
	
	\medskip
	
	\textbf{$L^p$-part: }We split $\Omega^c= \Omega_1^{\text{ext}}\cup \Omega^1_{\text{ext}}$. On $\Omega^1_{\text{ext}}$ we apply the estimates from the proof of the trace continuity. By \eqref{eq:trace_estimate_Lp_Omega^1}, there exists a constant $c_1=c_1(d,\Omega, p)>0$ such that
	\begin{equation*}
		\lim\limits_{s\to 1^{-}} \norm{\trns u}_{L^{p}(\Omega^1_{\text{ext}};\,\mu_s )}^p \le c_1 \lim\limits_{s\to 1^{-}} [u]_{\vs{s}{p}(\Omega\,|\, \Omega^1_{\text{ext}})}^p. 
	\end{equation*}
	Notice that 
	\begin{align}\label{eq:convergence_Wsp_Omega_Omega^1}
		[u]_{\vs{s}{p}(\Omega\,|\, \Omega^1_{\text{ext}})}^p&\le (1-s) \int\limits_{B_1(0)^c}\int\limits_{\Omega} \frac{\abs{u(y)-u(y+h)}^p}{\abs{h}^{d+sp}} \d y \d h \nonumber\\
		&\le 2^p\, (1-s) \norm{u}_{L^p(\R^d)}^p \int\limits_{B_1(0)^c} \frac{1}{\abs{h}^{d+sp}}  \d h= \tfrac{2^p\, \omega_{d-1} }{p}(1-s) \norm{u}_{L^p(\R^d)}
	\end{align}
and, thus, $\lim\limits_{s\to 1^{-}} \norm{\trns u}_{L^{p}(\Omega^1_{\text{ext}};\,\mu_s)}^p=0$. On $\Omega_1^{\text{ext}}$ we consider the family of measures $\{ \bar{\mu}_s \}$ from \autoref{lem:weak_convergence_measures} with $r=1$. This family converges weakly to the surface measure on $\partial \Omega$, which we denote by $\sigma$. Thus, we conclude the claim via
\begin{equation*}
	\norm{\trns u}_{L^{p}(\Omega_1^{\text{ext}};\,\mu_s)}^p= \int\limits_{\R^d} \frac{\abs{u(x)}^p}{(1+d_x)^{d+s(p-1)}}\bar{\mu}_s(\d x) \le \int\limits_{\R^d} \abs{u(x)}^p\bar{\mu}_s(\d x) \to \int\limits_{\partial \Omega} \abs{u(x)}^p \d \sigma(x)
\end{equation*}
and similarly
\begin{align*}
	\norm{\trns u}_{L^{p}(\Omega_1^{\text{ext}};\,\mu_s)}^p&\ge \int\limits_{\R^d} \frac{\abs{u(x)}^p}{(1+d_x)^{d+p-1}}\bar{\mu}_s(\d x) \to \int\limits_{\partial \Omega} \abs{u(x)}^p \d \sigma(x) \text{ as $s\to 1^{-}$}
\end{align*} 
because $x\mapsto (1+d_x)^{-d-p+1}$ is continuous.

\medskip

\textbf{Seminorm-part: } The main task is to show

\begin{align*}
\iint\limits_{\Omega^c \times \Omega^c}  \frac{\abs{u(x)-u(y)}^{p} (1+d_x)^{-d-s(p-1)} (1+d_y)^{-d-s(p-1)} }{ ((\abs{x-y}+d_x+d_y)\wedge 1)^{d+s(p-2)}} \d (\bar{\mu}_s \otimes \bar{\mu}_s)((x,y)) \\
\longrightarrow \quad \iint\limits_{\partial \Omega\times \partial \Omega} \frac{\abs{u(x)-u(y)}^p}{\abs{x-y}^{d+p-2}} \d (\sigma\otimes \sigma )(x,y) \quad \text{ as } s \to 1^{-}
\end{align*}
for some $r>0$ in the definition of $\bar{\mu}_s$. The choice of $r$ is arbitrary and will be made later. Let $\bar{\mu}_s$ be the measure from \autoref{lem:weak_convergence_measures} to the parameter $r$. As in the proof of \autoref{prop:trace_seminorm_estimate} we split $\Omega^c\times \Omega^c = \Omega_{r}^{\text{ext}}\times\Omega_{r}^{\text{ext}}\cup \Omega^c\times\Omega^{r}_{\text{ext}}\cup \Omega^{r}_{\text{ext}}\times \Omega^c$. The proof in the case $\Omega^c\times\Omega^{r}_{\text{ext}}$ equals the one in the case $\Omega^{r}_{\text{ext}}\times \Omega^c$ and shows that $[\trns u]_{\cT^{s,p}(\Omega^{r}_{\text{ext}}\,|\,\Omega^c)}^p$ converges to zero. By \eqref{eq:trace_estimate_Omega^1_Omega^c_prev}, \eqref{eq:convergence_Wsp_Omega_Omega^1} and a calculation similar to \eqref{eq:trace_estimate_Omega^1_Omega^c_prev_2} we find a constant $c_2=c_2(\Omega,p)>0$ such that 
\begin{align*}
[\trns u]_{\cT^{s,p}(\Omega^{r}_{\text{ext}}\,|\,\Omega^c)}^p &\le 2^p \Big( (1-s)\norm{\trns u}_{L^{p}(\Omega^c;\,\mu_s)}^p \int\limits_{\Omega^{r}_{\text{ext}}} d_x^{-d-sp}\d x +  \norm{\trns u}_{L^{p}(\Omega^{r}_{\text{ext}};\,\mu_s)}^p \mu_s(\Omega^c)   \Big)\\
	&\le  2^p \Big( \frac{(1-s)c_2r^{-sp}}{s}\norm{\trns u}_{L^{p}(\Omega^c;\,\mu_s)}^p +  \frac{c_2}{s}\norm{\trns u}_{L^{p}(\Omega^{r}_{\text{ext}};\,\mu_s)}^p    \Big)\to 0 \text{ as }s\to 1^{-}.
\end{align*}
Now, we consider the interesting part $\Omega_{r}^{\text{ext}}\times \Omega_{r}^{\text{ext}}$. We would like to apply \autoref{lem:weak_convergence_measures} to the function
\begin{equation*}
	h_s(x,y):= \frac{\abs{u(x)-u(y)}^{p} (1+d_x)^{-d-s(p-1)} (1+d_y)^{-d-s(p-1)} }{ ((\abs{x-y}+d_x+d_y)\wedge 1)^{d+s(p-2)}}
\end{equation*}
since
\begin{equation*}
[\trns u]_{\cT^{s,p}(\Omega_{r}^{\text{ext}}\,|\, \Omega_{r}^{\text{ext}})}^p = \iint\limits_{\Omega_{r}^{\text{ext}}\times \Omega_{r}^{\text{ext}}}  h_s(x,y) \d (\bar{\mu}_s \otimes \bar{\mu}_s)((x,y)) 
\end{equation*}
and $h_s(x,y)\to \abs{u(w)-u(z)}^p/ \abs{w-z}^{d+p-2}$ for $x\to w\in \partial \Omega$, $y\to z\in \partial \Omega$, $s \to 1^{-}$ for $w\ne z$. \autoref{lem:weak_convergence_measures} cannot be applied directly because $h_s$ is neither continuous on $\Omega^c\times \Omega^c$ nor independent of $s$. We resolve this issue by arguments similar to the ones used in \cite[Proposition 4.2]{GrHe22}. Let us fix a radial bump function $\varphi\in C_c^\infty(\overline{B_2(0)})$, $0\le \varphi\le 1$, $\varphi=1$ in $B_{1}$ and define $\varphi_\eps(x,y):= \varphi(\abs{x-y}/\eps)$ for $\eps\in(0,1)$. We set 
\begin{align*}
	a_s^\star&:=\begin{cases}
	1 & p\ge 2\\
	\eps^{-(1-s)(2-p)}& 1\le p<2
	\end{cases},\quad
	a_{s,\star}:=\begin{cases}
	\eps^{(1-s)(p-2)} & p\ge 2\\
	1& 1\le p<2
	\end{cases},\\
	h_{\eps}^\star(x,y)&:= \frac{\abs{u(x)-u(y)}^p}{(\abs{x-y}\wedge 1)^{d+p-2}}\, (1-\varphi_\eps(x,y)),\\
	h_{\eps,\star}(x,y)&:=\frac{\abs{u(x)-u(y)}^p((1+d_x)(1+d_y))^{-d-p+1}}{((\abs{x-y}+d_x+d_y)\wedge 1)^{d+p-2}}\, (1-\varphi_\eps(x,y)),\\
	g_{s,\eps} (x,y)&:= h_s(x,y)\varphi_\eps(x,y).\\
\end{align*}
For every $s \in (0,1)$ and every $\eps > 0$ we have $a_{s,\star}\,h_{\eps, \star}+ g_{s,\eps}\le h_s\le a_{s}^\star\,h_{\eps}^\star+ g_{s,\eps}$. We will now consider the limit $s \to 1^{-}$ and subsequently $\eps \to 0^{+}$ of the upper and lower bound in this inequality integrated against $\bar{\mu}_s \otimes  \bar{\mu}_s$. Both, $h_{\eps,\star}$ and $h_{\eps}^\star$ are continuous and bounded on $\R^d\times \R^d$. By \autoref{lem:weak_convergence_measures}, $\{\mu_s\}_s$ converges weakly to $\sigma$ and, thus, the sequence of product measures $\{ \mu_s\otimes \mu_s \}$ converges weakly to $\sigma\otimes \sigma$. Therefore, 
\begin{align*}
	\iint h_{\eps,\star}(x,y)\d (\bar{\mu}_s\otimes \bar{\mu}_s)((x,y)) &\to\!\! \iint\limits_{\partial \Omega\times \partial \Omega} \frac{\abs{u(x)-u(y)}^p}{\abs{x-y}^{d+p-2}}(1-\varphi_\eps(x,y)) \d (\sigma\otimes \sigma )(x,y) \!\quad \text{ as }s\to 1^{-},\\
	&\to \!\!\iint\limits_{\partial \Omega\times \partial \Omega} \frac{\abs{u(x)-u(y)}^p}{\abs{x-y}^{d+p-2}} \d (\sigma\otimes \sigma )(x,y) \quad \text{ as } \eps\to 0+.
\end{align*}
The same is true for $h_\eps^\star$. Furthermore, $a_s^\star, a_{s,\star}\to 1^{-}$. Now, we show that $\lim_{s\to 1^{-}}\iint g_{s,\eps} \d (\bar{\mu}_s\otimes \bar{\mu}_s)(x,y)\to 0$ as $\eps\to 0+$. Note
\begin{align*}
	\iint g_{s,\eps}(x,y) \d (\bar{\mu}_s\otimes \bar{\mu}_s)(x,y)\le [u]_{C^{0,1}} \iint\limits_{ \substack{\Omega_{r}^{\text{ext}}\times \Omega_{r}^{\text{ext}}\\ \abs{x-y}<2\eps  }} \abs{x-y}^{-d-s(p-2)+2}\d  (\bar{\mu}_s\otimes \bar{\mu}_s)(x,y) \,.
\end{align*}
The problem localizes. Since $\Omega$ is a bounded Lipschitz domain, we find a uniform localization radius $r_0>0$. Let $\cQ$ be a cover of $\partial \Omega$ with balls $B\in \cQ$ with radius $r_0>0$ such that the union of these balls with half their radii still contains $\partial \Omega$. Now, we fix $r>0$ from the beginning of the proof such that $\bigcup_{B\in \cQ} \frac{1}{2} B \supset \Omega_{r}^{\text{ext}}$ where $\frac{1}{2}B$ is the ball with half the radius. We assume $\eps<r_0/4$ such that for any $x,y\in \Omega_{r}^{\text{ext}}$ satisfying $x\in 1/2 B$, $B\in \cQ$, $\abs{x-y}<2\eps$ then $y \in B$. Thus, we only need to consider one ball $B\in \cQ$. After translation we assume with loss of generality $B=B_{r_0}(0)$. We flatten the boundary $\partial \Omega$ that lies in $B$ with the function $\phi\in C^{0,1}(\R^{d-1})$ such that $\Omega\cap B = \{ (x',x_d)\in B \,|\, x_d<\phi(x') \}$. Since $\Omega$ has a uniform Lipschitz boundary, the Lipschitz constant of $\phi$ does not depend on the position of $B$. For any $x\in B\cap \Omega^c$, $d_x\ge (1+\norm{\phi}_{C^{0,1}})^{-1}\abs{x_d-\phi(x')}$. Therefore, 
\begin{align*}
	&\iint\limits_{ (\Omega^c\cap Q)\times (\Omega^c \cap Q)}\1_{B_{2\eps}(x)}(y) \abs{x-y}^{-d-s(p-2)+2p}\d  (\bar{\mu}_s\otimes \bar{\mu}_s)(x,y)\\
	&\le (1+\norm{\phi}_{C^{0,1}})^2 \int\limits_{B^{(d-1)}_{r_0}(0)}\int\limits_{\phi(x')}^{2r_0} \int\limits_{B^{(d-1)}_{2\eps}(x')}\int\limits_{\phi(y')}^{2r_0}  \frac{(1-s)^2 \abs{x'-y'}^{-d-s(p-2)+2p}}{(x_d-\phi(x'))^s(y_d-\phi(y'))^s}\d y_d \d y' \d x_d \d x'\\
	&\le r_0^{2-2s} (1+\norm{\phi}_{C^{0,1}})^2 \frac{\omega_{d-2}^2}{d-1}r_0^{d-1} \int\limits_{0}^{2\eps} t^{(1-s)(p-2)}\d t \\
	&= r_0^{2-2s} (1+\norm{\phi}_{C^{0,1}})^2 \frac{\omega_{d-2}^2}{d-1}r_0^{d-1} \frac{(2\eps)^{(1-s)(p-2)+1}}{(1-s)(p-2)+1} \\
	&\quad \to  \quad (1+\norm{\phi}_{C^{0,1}})^2 \frac{\omega_{d-2}^2}{d-1}r_0^{d-1} 2\eps \quad \to \quad 0.
\end{align*}
In the last line, we first consider the limit $s\to 1^{-}$ and then the limit $\eps\to 0+$. Thus, 
 \begin{equation*}
 	\iint g_{s,\eps}(x,y) \d (\bar{\mu}_s\otimes \bar{\mu}_s)(x,y)\le \sum_{B\in \cQ} \,\iint\limits_{  Q\times  Q}\1_{B_{2\eps}(x)}(y) \abs{x-y}^{-d-s(p-2)+2p}\d  (\bar{\mu}_s\otimes \bar{\mu}_s)(x,y)\to 0. 
 \end{equation*}
 The result for $C_c^{0,1}(\R^d)$ functions follows from $a_{s,\star}\,h_{\eps, \star}+ g_{s,\eps}\le h_s\le a_{s}^\star\,h_{\eps}^\star+ g_{s,\eps}$. The proof of $[u]_{\cT^{s,1}(\Omega^c)}\to [u]_{B_{1,1}^0(\partial \Omega)}$ follows analogous

\textbf{Step 2: } Let $1<p<\infty$. We generalize the result from step 1 to all functions in $W^{1,p}(\R^d)$ via a density argument. Firstly, there exists a constant $c_3=c_3(d,p)>0$ such that $\norm{u}_{\VspOm}\le c_3 \norm{u}_{W^{s,p}(\R^d)}$. Combining this estimate with \eqref{eq:trace_Lp_continuity} and \eqref{eq:trace_continuity} yields 
\begin{align*}
	\norm{\trns u}_{\cT^{s,p}(\Omega^c)}\le c_4\,\norm{u}_{\VspOm}\le c_3 c_4 \norm{u}_{W^{1,p}(\R^d)}.
\end{align*}
Here $c_4>0$ is the sum of the constants from \autoref{prop:trace_Lp_estimate} and \autoref{prop:trace_seminorm_estimate}. Now take any $u\in W^{1,p}(\R^d)$. Since $C_c^{0,1}(\R^d)$ is dense in $W^{1,p}(\R^d)$ we find a sequence $u_n\in C_c^{0,1}(\R^d)$ such that $\norm{u-u_n}_{W^{1,p}(\R^d)}\to 0$ as $n\to \infty$. Since the classical trace is continuous the mapping \begin{equation*}
	\gamma_0: W^{1,p}(\R^d) \overset{\text{cts.}}{\hookrightarrow} W^{1,p}(\Omega) \overset{\gamma}{\to} W^{1-1/p,p}(\partial \Omega)
\end{equation*} 
is linear and continuous. Thus, uniformly in $s\in(0,1)$ 
\begin{align*}
	\norm{\trns u- \trns u_n}_{\cT^{s,p}(\Omega^c)}&\le c_1 \norm{u-u_n}_{W^{1,p}(\R^d)}\to 0,\\
	\norm{\gamma_0 u - \gamma_0 u_n}_{W^{1-1/p,p}(\partial \Omega)}&\le C_3 \norm{u-u_n}_{W^{1,p}(\R^d)} \to 0 \text{ as }n \to \infty.
\end{align*}
Finally, step 1 yields $\norm{u_n}_{\cT^{s,p}(\Omega^c)}\to \norm{u_n}_{W^{1-1/p,p}(\partial \Omega)}$ as $s\to 1^{-}$. The proof of the statement for $d=1$ follows with minor changes and we omit it. Notice in this case that functions in $W^{s,p}(\Omega)$ for $s>1/p$ are continuous by Morrey's inequality. Lastly, the proof of $\norm{\trns u}_{L^1(\Omega;\, \mu_s)}\to \norm{\gamma u}_{L^1(\partial \Omega)}$ follows analogously to the proof in step 2 using the uniform trace embedding \autoref{prop:trace_Lp_estimate}. 

\end{proof}

\appendix
\section{Reflected Whitney cubes}\label{sec:whitney}
	The following results are taken from \cite[Section 3.2]{DyKa19}. Throughout this section we fix a Lipschitz domain $\Omega\subset \R^d$. We fix a Whitney decomposition $\cW(\R^d\setminus\overline{\Omega})$ of the open set with Lipschitz boundary $\R^d\setminus \overline{\Omega}$, \ie each cube $Q\in \cW(\R^d\setminus\overline{\Omega})$ satisfies $\diam{Q}\le \distw{Q}{\partial \Omega}\le 4 \diam{Q}$. We denote the center of a Whitney cube $Q\in \cW(\R^d\setminus\overline{\Omega})$ by $q_Q\in Q$. These cubes satisfy 
	\begin{equation}\label{eq:prop_cubes_1}
	\sum_{Q\in \cW(\R^d\setminus\overline{\Omega})}\1_{Q}= \1_{\R^d\setminus\overline{\Omega}}.
	\end{equation}
	Bounded Lipschitz domains are both interior and exterior thick, see \cite[Definition 14, 15]{DyKa19}. Thereby, we find a constant $M> 1$ and a reflected Whitney cube $\widetilde{Q}\subset \Omega$ for any $Q\in \cW(\R^d\setminus\overline{\Omega})$ such that $\diam{Q}\le \inr{\Omega}=\sup\{ r\,|\, B_r\subset \Omega \}$ satisfying 
	\begin{align}
	\diam{\widetilde{Q}}\le \distw{\widetilde{Q}}{\partial \Omega}\le 4 \diam{\widetilde{Q}},\nonumber\\
	M^{-1}\diam{Q}\le \diam{\widetilde{Q}}\le M \diam{Q},\label{eq:reflected_cubes_comparable}\\
	\distw{Q}{\widetilde{Q}}\le M\distw{Q}{\partial \Omega}.\nonumber
	\end{align}
	Again we denote the centers of the reflected cubes by $q_{\widetilde{Q}}\in \widetilde{Q}$. The collection of these reflected cubes cover the domain $\Omega$, \ie
	\begin{equation}\label{eq:prop1_reflec_cubes}
	\bigcup\limits_{\substack{Q\in \cW(\R^d\setminus\overline{\Omega})\\ \diam{Q}\le \inr{\Omega}}} \widetilde{Q}= \Omega.
	\end{equation}
	Additionally, the reflected cubes satisfy the bounded overlap property, \ie there exists a constant $N\ge 1$ such that 
	\begin{equation}\label{eq:bounded_overlap}
	\sum\limits_{\substack{Q\in \cW(\R^d\setminus\overline{\Omega})\\ \diam{Q}\le \inr{\Omega}}} \1_{\widetilde{Q}}\le N \1_{\Omega},
	\end{equation}
	see \cite[Remark 19]{DyKa19}. We define for short $\cW_{\inr{\Omega}}(\R^d\setminus \Omega):= \{ Q\in \cW(\R^d\setminus \Omega)\,|\, \diam{Q}\le \inr{\Omega} \}$.

	\section{Hardy inequality for the half space}
	The following Hardy inequality for the half space is proven in \cite{FrSe10} and \cite{BoDy11} in the case $p=2$. See \cite{DyKi22} for a corresponding weighted Hardy inequality. Note that the constant $\cD_{s,p}$ is optimal.
	\begin{theorem}[{\cite[Theorem 1.1]{FrSe10}, \cite[Theorem 1]{DyKi22}}]\label{th:hardy_half_space}
		Let $0<s<1$, $d\in \N$, $p\in [1,\infty)$ with $ps\ne 1$. Then 
		\begin{equation}
		\cD_{s,p}\int\limits_{\R_+^d} \frac{\abs{u(x)}^p}{x_d^{sp}}\d x \le \int\limits_{\R_+^d\times \R_+^d} \frac{\abs{u(x)-u(y)}^p}{\abs{x-y}^{d+sp}}\d (x,y)
		\end{equation}
		for any $u\in W_0^{s,p}(\R^d_+)= \overline{C_c^\infty(\R_+^d)}^{\norm{\cdot}_{W^{s,p}}}$. The constant is given by
		\begin{equation}\label{eq;constant}
		\cD_{s,p}:=2\pi^{(d-1)/2}\, \frac{\Gamma(\frac{1+sp}{2})}{\Gamma(\frac{d+sp}{2})} \, \int\limits_0^1 \frac{\abs{1-t^{(ps-1)/p}}^p}{(1-t)^{1+ps}}\d t.
		\end{equation}
		Furthermore, in the case $p=1$ and $d=1$ the inequality only holds for functions that are proportional to a non-increasing function.
	\end{theorem}
	
	\begin{lemma}\label{lem:constant}
		There exists a constant $C=C(d)\ge 1$ such that for all $0<s<1$ 
		\begin{equation*}
		C^{-1}\le s \, \cD_{s,1}\le C
		\end{equation*}
		where $\cD_{s,1}$ is the constant defined in \eqref{eq;constant}
	\end{lemma}
	
	\begin{proof}
		We split the integral in \eqref{eq;constant} into two parts. 
		\begin{align*}
		\int\limits_0^{1/2} \frac{\abs{1-t^{s-1}}}{(1-t)^{1+s}}\d t\le 2^{1+s} \int\limits_0^{1/2}t^{s-1}\d t \le 4\frac{(1/2)^s}{s}\le \frac{4}{s}.
		\end{align*}
		An lower bound in the same term is calculated similarly:
		\begin{align*}
		\int\limits_0^{1/2} \frac{\abs{1-t^{s-1}}}{(1-t)^{1+s}}\d t\ge (1-(1/2)^{1-s}) \int\limits_0^{1/2} t^{s-1}\d t=\frac{2^{1-s}-1}{2s}\ge \frac{1-s}{4s}
		\end{align*}
		We move to the remaining integral. 
		\begin{equation*}
		\int\limits_{1/2}^{1} \frac{\abs{1-t^{s-1}}}{(1-t)^{1+s}}\d t \le2^{1-s}\int\limits_{1/2}^{1} \frac{1}{(1-t)^{1+s}}\big( (1-s)\int\limits_{t}^1 r^{-s} \d r \big)\d t \le 2\int\limits_{1/2}^{1} \frac{1-s}{(1-t)^{s}}\d t \le 2.
		\end{equation*}
		And an lower bound is calculated in a similar fashion.
		\begin{equation*}
		\int\limits_{1/2}^{1} \frac{\abs{1-t^{s-1}}}{(1-t)^{1+s}}\d t \ge\int\limits_{1/2}^{1} \frac{1}{(1-t)^{1+s}}\big( (1-s)\int\limits_{t}^1 r^{-s} \d r \big)\d t \ge \int\limits_{1/2}^{1} \frac{1-s}{(1-t)^{s}}\d t = (1/2)^{1-s}\ge 1/2
		\end{equation*}
		Therefore, 
		\begin{equation*}
		2\pi^{(d-1)/2} \frac{1}{\Gamma(d/2)\vee \Gamma((d+1)/2)} \frac{1}{4s}\le \cD_{s,1}\le 2\pi^{(d-1)/2} \frac{\Gamma(1/2)}{\Gamma(d/2)\wedge \Gamma((d+1)/2)}\frac{6}{s}.
		\end{equation*}
	\end{proof}
	
	\medskip
	
%	\bibliographystyle{alpha}
%	\bibliography{literatur}

\begin{thebibliography}{BGPPR23}
	
	\bibitem[ABH19]{ABH19}
	Gabriel Acosta, Juan~Pablo Borthagaray, and Norbert Heuer.
	\newblock Finite element approximations of the nonhomogeneous fractional
	{D}irichlet problem.
	\newblock {\em IMA J. Numer. Anal.}, 39(3):1471--1501, 2019.
	
	\bibitem[Aro55]{trace_aronzajin}
	Nachman Aronszajn.
	\newblock Boundary values of functions with finite {Dirichlet} integral.
	\newblock Conference on partial differential equations, {Univ}. {Kansas},
	{Summer} 1954, 77-93 (1955)., 1955.
	
	\bibitem[BBM01]{BBM01}
	Jean Bourgain, Haim Brezis, and Petru Mironescu.
	\newblock Another look at {S}obolev spaces.
	\newblock In {\em Optimal control and partial differential equations}, pages
	439--455. IOS, Amsterdam, 2001.
	
	\bibitem[BD11]{BoDy11}
	Krzysztof Bogdan and Bart{\l}omiej Dyda.
	\newblock The best constant in a fractional {H}ardy inequality.
	\newblock {\em Math. Nachr.}, 284(5-6):629--638, 2011.
	
	\bibitem[BDL{\etalchar{+}}23]{BDLMV23}
	Claudia Bucur, Serena Dipierro, Luca Lombardini, Jos\'{e}~M. Maz\'{o}n, and
	Enrico Valdinoci.
	\newblock {$(s, p)$}-harmonic approximation of functions of least
	{$W^{s,1}$}-seminorm.
	\newblock {\em Int. Math. Res. Not. IMRN}, (2):1173--1235, 2023.
	
	\bibitem[BGPPR23]{BGPR20b}
	Krzysztof Bogdan, Tomasz Grzywny, Katarzyna Pietruska-Pa{\l}uba, and Artur
	Rutkowski.
	\newblock Nonlinear nonlocal {D}ouglas identity.
	\newblock {\em Calc. Var. Partial Differential Equations}, 62(5):Paper No. 151,
	2023.
	
	\bibitem[BGPR20]{BGPR20}
	Krzysztof Bogdan, Tomasz Grzywny, Katarzyna {Pietruska-Pa{\l}uba}, and Artur
	Rutkowski.
	\newblock Extension and trace for nonlocal operators.
	\newblock {\em J. Math. Pures Appl. (9)}, 137:33--69, 2020.
	
	\bibitem[BIN75]{BIN75}
	Oleg~V. Besov, Valentin~P. Il{'}in, and Sergei~M. Nikol{'}ski\u{\i}.
	\newblock {\em {I}ntegral{'}nye predstavleniya funktsi{\u{\i}} i teoremy
		vlozheniya}.
	\newblock Izdat. ``Nauka'', Moscow, 1975.
	
	\bibitem[BL76]{BeLo76}
	J{\"{o}}ran Bergh and J{\"{o}}rgen L{\"{o}}fstr{\"{o}}m.
	\newblock {\em Interpolation spaces. {A}n introduction}.
	\newblock Grundlehren der Mathematischen Wissenschaften, No. 223.
	Springer-Verlag, Berlin-New York, 1976.
	
	\bibitem[BLS18]{BLS18}
	Lorenzo Brasco, Erik Lindgren, and Armin Schikorra.
	\newblock Higher {H}\"{o}lder regularity for the fractional {$p$}-{L}aplacian
	in the superquadratic case.
	\newblock {\em Adv. Math.}, 338:782--846, 2018.
	
	\bibitem[BS56]{trace_Slobodeckij_original}
	V.M. Babich and L.N. Slobedeckij.
	\newblock Sulla limitatezza dell’integrale di dirichlet (in russo).
	\newblock {\em Doklady Akademii Nauk S.S.S.R.}, 26:604--8, 1956.
	
	\bibitem[D{\'{a}}v02]{Dav02}
	Juan~Diego D{\'{a}}vila.
	\newblock On an open question about functions of bounded variation.
	\newblock {\em Calc. Var. Partial Differential Equations}, 15(4):519--527,
	2002.
	
	\bibitem[DCKP16]{DKP16}
	Agnese Di~Castro, Tuomo Kuusi, and Giampiero Palatucci.
	\newblock Local behavior of fractional {$p$}-minimizers.
	\newblock {\em Ann. Inst. H. Poincar\'{e} C Anal. Non Lin\'{e}aire},
	33(5):1279--1299, 2016.
	
	\bibitem[Din96]{Din96}
	Zhonghai Ding.
	\newblock A proof of the trace theorem of {S}obolev spaces on {L}ipschitz
	domains.
	\newblock {\em Proc. Amer. Math. Soc.}, 124(2):591--600, 1996.
	
	\bibitem[DK19]{DyKa19}
	Bart{\l}omiej Dyda and Moritz Kassmann.
	\newblock Function spaces and extension results for nonlocal {D}irichlet
	problems.
	\newblock {\em J. Funct. Anal.}, 277(11):108134, 22, 2019.
	
	\bibitem[DK22a]{DyKi22b}
	Bart{\l}omiej Dyda and Micha{\l} Kijaczko.
	\newblock On density of compactly supported smooth functions in fractional
	{S}obolev spaces.
	\newblock {\em Ann. Mat. Pura Appl. (4)}, 201(4):1855--1867, 2022.
	
	\bibitem[DK22b]{DyKi22}
	Bart{\l}omiej {Dyda} and Micha{\l} {Kijaczko}.
	\newblock {Sharp weighted fractional Hardy inequalities}.
	\newblock {\em arXiv e-prints}, page arXiv:2210.06760, October 2022.
	
	\bibitem[DMT22]{DMT22}
	Qiang Du, Tadele Mengesha, and Xiaochuan Tian.
	\newblock Fractional {H}ardy-type and trace theorems for nonlocal function
	spaces with heterogeneous localization.
	\newblock {\em Anal. Appl. (Singap.)}, 20(3):579--614, 2022.
	
	\bibitem[DROV17]{DRV17}
	Serena Dipierro, Xavier Ros-Oton, and Enrico Valdinoci.
	\newblock Nonlocal problems with {N}eumann boundary conditions.
	\newblock {\em Rev. Mat. Iberoam.}, 33(2):377--416, 2017.
	
	\bibitem[DTWY22]{DTWY22}
	Qiang Du, Xiaochuan Tian, Cory Wright, and Yue Yu.
	\newblock Nonlocal trace spaces and extension results for nonlocal calculus.
	\newblock {\em J. Funct. Anal.}, 282(12):Paper No. 109453, 63, 2022.
	
	\bibitem[Dyd04]{Dyd04}
	Bart{\l}omiej Dyda.
	\newblock A fractional order {H}ardy inequality.
	\newblock {\em Illinois J. Math.}, 48(2):575--588, 2004.
	
	\bibitem[Fed69]{Fed69}
	Herbert Federer.
	\newblock {\em Geometric measure theory}.
	\newblock Die Grundlehren der mathematischen Wissenschaften, Band 153.
	Springer-Verlag New York, Inc., New York, 1969.
	
	\bibitem[FK22]{FoKa22}
	Guy {Foghem} and Moritz {Kassmann}.
	\newblock {A general framework for nonlocal Neumann problems}.
	\newblock {\em arXiv e-prints}, 2022.
	\newblock arXiv:2204.06793, to appear in Comm. Math. Sci.
	
	\bibitem[FKV15]{FKV15}
	Matthieu Felsinger, Moritz Kassmann, and Paul Voigt.
	\newblock The {D}irichlet problem for nonlocal operators.
	\newblock {\em Math. Z.}, 279(3-4):779--809, 2015.
	
	\bibitem[Fog20]{Fog20}
	Guy Foghem.
	\newblock {\em $L^2$-Theory for Nonlocal Operators on Domains}.
	\newblock Universit{\"a}tsbibliothek Bielefeld, 2020.
	
	\bibitem[Fog23a]{Foghem_2023}
	Guy Foghem.
	\newblock A remake of {B}ourgain--{B}rezis--{M}ironescu characterization of
	{S}obolev spaces.
	\newblock {\em Partial Differ. Equ. Appl.}, 4(2):Paper No. 16, 2023.
	
	\bibitem[{Fog}23b]{Fog23}
	Guy {Foghem}.
	\newblock {Stability of complement value problems for $p$-L{\'e}vy operators}.
	\newblock {\em arXiv e-prints}, page arXiv:2303.03776, March 2023.
	
	\bibitem[Fos21]{Fos21}
	Mikil Foss.
	\newblock Traces on general sets in {$\mathbb{R}^n$} for functions with no
	differentiability requirements.
	\newblock {\em SIAM J. Math. Anal.}, 53(4):4212--4251, 2021.
	
	\bibitem[FS10]{FrSe10}
	Rupert~L. Frank and Robert Seiringer.
	\newblock Sharp fractional {H}ardy inequalities in half-spaces.
	\newblock In {\em Around the research of {V}ladimir {M}az'ya. {I}}, volume~11
	of {\em Int. Math. Ser. (N. Y.)}, pages 161--167. Springer, New York, 2010.
	
	\bibitem[FVV22]{Vu22}
	Leonhard {Frerick}, Christian {Vollmann}, and Michael {Vu}.
	\newblock {The nonlocal Neumann problem}.
	\newblock {\em arXiv e-prints}, page arXiv:2208.04561, August 2022.
	
	\bibitem[Gag57]{Gag57}
	Emilio Gagliardo.
	\newblock Caratterizzazioni delle tracce sulla frontiera relative ad alcune
	classi di funzioni in {$n$} variabili.
	\newblock {\em Rend. Sem. Mat. Univ. Padova}, 27:284--305, 1957.
	
	\bibitem[GH22]{GrHe22}
	Florian {Grube} and Thorben {Hensiek}.
	\newblock {Robust nonlocal trace spaces and Neumann problems}.
	\newblock {\em arXiv e-prints}, page arXiv:2209.04397, September 2022.
	
	\bibitem[Gri11]{Gri85}
	Pierre Grisvard.
	\newblock {\em Elliptic problems in nonsmooth domains}, volume~69 of {\em
		Classics in Applied Mathematics}.
	\newblock Society for Industrial and Applied Mathematics (SIAM), Philadelphia,
	PA, 2011.
	\newblock Reprint of the 1985 original [ MR0775683], With a foreword by Susanne
	C. Brenner.
	
	\bibitem[Gru15]{Grub15}
	Gerd Grubb.
	\newblock Fractional {L}aplacians on domains, a development of
	{H}\"{o}rmander's theory of {$\mu$}-transmission pseudodifferential
	operators.
	\newblock {\em Adv. Math.}, 268:478--528, 2015.
	
	\bibitem[Jak67]{Jakovlev}
	G.~N. Jakovlev.
	\newblock Traces of functions from the space {$w_{p}^{l}$} onto piecewise
	smooth surfaces.
	\newblock {\em Mat. Sb. (N.S.)}, 74 (116):526--543, 1967.
	
	\bibitem[Jon94]{Jon94}
	Alf Jonsson.
	\newblock Besov spaces on closed subsets of {${\mathbb{ R}}^n$}.
	\newblock {\em Trans. Amer. Math. Soc.}, 341(1):355--370, 1994.
	
	\bibitem[JW78]{JoWa78}
	A.~Jonsson and H.~Wallin.
	\newblock A {W}hitney extension theorem in {$L_{p}$} and {B}esov spaces.
	\newblock {\em Ann. Inst. Fourier (Grenoble)}, 28(1):vi, 139--192, 1978.
	
	\bibitem[JW84]{JoWa84}
	Alf Jonsson and Hans Wallin.
	\newblock Function spaces on subsets of {${\bf R}^n$}.
	\newblock {\em Math. Rep.}, 2(1):xiv+221, 1984.
	
	\bibitem[Kim07]{Kim07}
	Doyoon Kim.
	\newblock Trace theorems for {S}obolev-{S}lobodeckij spaces with or without
	weights.
	\newblock {\em J. Funct. Spaces Appl.}, 5(3):243--268, 2007.
	
	\bibitem[KKP16]{KKP16}
	Janne Korvenp\"{a}\"{a}, Tuomo Kuusi, and Giampiero Palatucci.
	\newblock The obstacle problem for nonlinear integro-differential operators.
	\newblock {\em Calc. Var. Partial Differential Equations}, 55(3):Art. 63, 29,
	2016.
	
	\bibitem[KKP17]{KKP17}
	Janne Korvenp\"{a}\"{a}, Tuomo Kuusi, and Giampiero Palatucci.
	\newblock Fractional superharmonic functions and the {P}erron method for
	nonlinear integro-differential equations.
	\newblock {\em Math. Ann.}, 369(3-4):1443--1489, 2017.
	
	\bibitem[Leo17]{Leo17}
	Giovanni Leoni.
	\newblock {\em A first course in {S}obolev spaces}, volume 181 of {\em Graduate
		Studies in Mathematics}.
	\newblock American Mathematical Society, Providence, RI, second edition, 2017.
	
	\bibitem[LL17]{LiLi17}
	Erik Lindgren and Peter Lindqvist.
	\newblock Perron's method and {W}iener's theorem for a nonlocal equation.
	\newblock {\em Potential Anal.}, 46(4):705--737, 2017.
	
	\bibitem[LM72]{lions}
	J.-L. Lions and E.~Magenes.
	\newblock {\em Non-homogeneous boundary value problems and applications. {V}ol.
		{I}}.
	\newblock Die Grundlehren der mathematischen Wissenschaften, Band 181.
	Springer-Verlag, New York-Heidelberg, 1972.
	\newblock Translated from the French by P. Kenneth.
	
	\bibitem[Mar87]{Mar87}
	J{\"{u}}rgen Marschall.
	\newblock The trace of {S}obolev-{S}lobodeckij spaces on {L}ipschitz domains.
	\newblock {\em Manuscripta Math.}, 58(1-2):47--65, 1987.
	
	\bibitem[MD16]{MeDu16}
	Tadele Mengesha and Qiang Du.
	\newblock Characterization of function spaces of vector fields and an
	application in nonlinear peridynamics.
	\newblock {\em Nonlinear Anal.}, 140:82--111, 2016.
	
	\bibitem[MSS18]{MSS18}
	Luk\'{a}\v{s} Mal\'{y}, Nageswari Shanmugalingam, and Marie Snipes.
	\newblock Trace and extension theorems for functions of bounded variation.
	\newblock {\em Ann. Sc. Norm. Super. Pisa Cl. Sci. (5)}, 18(1):313--341, 2018.
	
	\bibitem[Ne{\v{c}}67]{Nec67}
	Jind{\v{r}}ich Ne{\v{c}}as.
	\newblock {\em Les {m\'{e}}thodes directes en th\'{e}orie des {\'{e}}quations
		elliptiques}.
	\newblock Masson et Cie, {\'{E}}diteurs, Paris; Academia, {\'{E}}diteurs,
	Prague, 1967.
	
	\bibitem[Ne{\v{c}}12]{Nec12}
	Jind{\v{r}}ich Ne{\v{c}}as.
	\newblock {\em Direct methods in the theory of elliptic equations}.
	\newblock Springer Monographs in Mathematics. Springer, Heidelberg, 2012.
	\newblock Translated from the 1967 French original by Gerard Tronel and Alois
	Kufner, Editorial coordination and preface by {\v{S}}{\'{a}}rka
	Ne{\v{c}}asov{\'{a}} and a contribution by Christian G. Simader.
	
	\bibitem[NLM88]{NLW88}
	S.~M. Nikol{'}ski\u{\i}, P.~I. Lizorkin, and N.~V. Miroshin.
	\newblock Weighted function spaces and their applications to the investigation
	of boundary value problems for degenerate elliptic equations.
	\newblock {\em Izv. Vyssh. Uchebn. Zaved. Mat.}, 32(8):4--30, 1988.
	
	\bibitem[Pal18]{Pal18}
	Giampiero Palatucci.
	\newblock The {D}irichlet problem for the {$p$}-fractional {L}aplace equation.
	\newblock {\em Nonlinear Anal.}, 177(part B):699--732, 2018.
	
	\bibitem[Pee79]{Pee79}
	Jaak Peetre.
	\newblock A counterexample connected with {G}agliardo's trace theorem.
	\newblock {\em Comment. Math. Special Issue}, 2:277--282, 1979.
	\newblock Special issue dedicated to W\l adys\l aw Orlicz on the occasion of
	his seventy-fifth birthday.
	
	\bibitem[PP17]{PiPu17}
	Paolo Piersanti and Patrizia Pucci.
	\newblock Existence theorems for fractional {$p$}-{L}aplacian problems.
	\newblock {\em Anal. Appl. (Singap.)}, 15(5):607--640, 2017.
	
	\bibitem[Pro56]{trace_prodi}
	Giovanni Prodi.
	\newblock Tracce sulla frontiera delle funzioni di {Beppo} {Levi}.
	\newblock {\em Rend. Semin. Mat. Univ. Padova}, 26:36--60, 1956.
	
	\bibitem[SD23]{ScDu23}
	James~M. {Scott} and Qiang {Du}.
	\newblock {Nonlocal boundary-value problems with local boundary conditions}.
	\newblock {\em arXiv e-prints}, page arXiv:2301.02923, January 2023.
	
	\bibitem[Slo58]{trace_Slobodeckij}
	L.~N. Slobodeckij.
	\newblock S. {L}. {Sobolev}'s spaces of fractional order and their application
	to boundary problems for partial differential equations.
	\newblock {\em Dokl. Akad. Nauk SSSR}, 118:243--246, 1958.
	
	\bibitem[SV12]{SeVa12}
	Raffaella Servadei and Enrico Valdinoci.
	\newblock Mountain pass solutions for non-local elliptic operators.
	\newblock {\em J. Math. Anal. Appl.}, 389(2):887--898, 2012.
	
	\bibitem[SV13]{SeVa13}
	Raffaella Servadei and Enrico Valdinoci.
	\newblock Variational methods for non-local operators of elliptic type.
	\newblock {\em Discrete Contin. Dyn. Syst.}, 33(5):2105--2137, 2013.
	
	\bibitem[SV14]{SeVa14}
	Raffaella Servadei and Enrico Valdinoci.
	\newblock Weak and viscosity solutions of the fractional {L}aplace equation.
	\newblock {\em Publ. Mat.}, 58(1):133--154, 2014.
	
	\bibitem[Tai64]{Tai64}
	Mitchell~H. Taibleson.
	\newblock On the theory of {L}ipschitz spaces of distributions on {E}uclidean
	{$n$}-space. {I}. {P}rincipal properties.
	\newblock {\em J. Math. Mech.}, 13:407--479, 1964.
	
	\bibitem[TD17]{DuTi17}
	Xiaochuan Tian and Qiang Du.
	\newblock Trace theorems for some nonlocal function spaces with heterogeneous
	localization.
	\newblock {\em SIAM J. Math. Anal.}, 49(2):1621--1644, 2017.
	
	\bibitem[Tri95]{Tri95}
	Hans Triebel.
	\newblock {\em Interpolation theory, function spaces, differential operators}.
	\newblock Johann Ambrosius Barth, Heidelberg, second edition, 1995.
	
	\bibitem[Tri10]{triebel_1}
	Hans Triebel.
	\newblock {\em Theory of function spaces}.
	\newblock Modern Birkh\"{a}user Classics. Birkh\"{a}user/Springer Basel AG,
	Basel, 2010.
	\newblock Reprint of 1983 edition [MR0730762], Also published in 1983 by
	Birkh\"{a}user Verlag [MR0781540].
	
	\bibitem[Von21]{Von21}
	Zoran Vondra\v{c}ek.
	\newblock A probabilistic approach to a non-local quadratic form and its
	connection to the {N}eumann boundary condition problem.
	\newblock {\em Math. Nachr.}, 294(1):177--194, 2021.
	
\end{thebibliography}

\newcommand{\etalchar}[1]{$^{#1}$}

\end{document}